\documentclass[letter,11pt]{amsart}
\usepackage{graphicx} 
\usepackage{overpic} 
\usepackage{amsmath,amsthm,amssymb}
\usepackage{multicol}
\usepackage{xcolor}
\usepackage{tikz,tikz-cd}
\tikzcdset{
  every label/.append style = {
    font = \everymath\expandafter{\the\everymath\textstyle},
  },
}

\usepackage{natbib}

\usepackage{scalerel} 
\usepackage{upgreek} 

\usepackage{stmaryrd} 

\usepackage[left=1.5in,right=1.5in,top=1.5in,bottom=1.5in]{geometry}

\title{Asymptotically rigid mapping class groups of infinite graphs}
\author{Thomas Hill, Sanghoon Kwak, Brian Udall, and Jeremy West}

\usepackage[colorlinks,citecolor=blue!70!black,linkcolor=red!60!black]{hyperref}
\usepackage[nameinlink]{cleveref}

\usepackage{tabularx}
\usepackage{hyperref}
\usepackage{cleveref} 

\newcounter{relctr}

\newcommand{\relationlabel}[2]{%
  \refstepcounter{relctr}%
  \label{#1}%
  #2%
}

\newcommand{\refrel}[1]{%
  (\hyperref[#1]{$r_{\ref*{#1}}$})%
}

\usepackage{enumitem}

\theoremstyle{plain}
\newtheorem{THM}{Theorem}[section]

\newtheorem{PROP}[THM]{Proposition}
\newtheorem{LEM}[THM]{Lemma}

\newtheorem{COR}[THM]{Corollary}

\theoremstyle{plain} 
\newcommand{\thistheoremname}{}
\newtheorem{genericthm}[THM]{\thistheoremname}

\newtheorem*{genericthm*}{\thistheoremname}
\newenvironment{namedthm*}[1]
  {\renewcommand{\thistheoremname}{#1}%
   \begin{genericthm*}}
  {\end{genericthm*}}

\Crefname{THM}{Theorem}{Theorems}
\Crefname{PROP}{Proposition}{Propositions}
\Crefname{LEM}{Lemma}{Lemmas}
\Crefname{FACT}{Fact}{Facts}
\Crefname{COR}{Corollary}{Corollaries}
\Crefname{CONJ}{Conjecture}{Conjectures}
\Crefname{CLAIM}{Claim}{Claims}
\Crefname{CASE}{Case}{Cases}

\theoremstyle{definition}
\newtheorem{DEF}[THM]{Definition}
\newtheorem{RMK}[THM]{Remark}
\newtheorem{EX}[THM]{Example}

\newtheorem{Q}[THM]{Question}

\newtheorem{NOTN}[THM]{Notation}

\Crefname{DEF}{Definition}{Definitions}
\Crefname{RMK}{Remark}{Remarks}
\Crefname{EX}{Example}{Examples}
\Crefname{EXER}{Exercise}{Exercises}
\Crefname{Q}{Question}{Questions}

\newlist{LEMenum}{enumerate}{1} 
\setlist[LEMenum]{label=(\roman*),ref=\theLEM\,(\roman*)}
\Crefname{LEMenumi}{Lemma}{Lemmas}
\crefname{LEMenumi}{}{}

\definecolor{marine}{HTML}{2c61b5}
\definecolor{forest}{HTML}{228b22}

\newcounter{COMMENTS}
\newcommand{\newComment}[3]{
    \expandafter\newcommand\csname #1\endcsname[1]{ 
        \textbf{%
		\color{#3}(\uppercase{#2}\theCOMMENTS)%
        }
        \marginpar{\scriptsize\raggedright\textbf{%
			{\color{#3}(\uppercase{#2}\theCOMMENTS)#1: 
	}} ##1}
	    \stepcounter{COMMENTS}
    }
    \expandafter\newcommand\csname #2\endcsname[1]{{\color{#3} ##1}} 
}

\newComment{Brian}{Bu}{red}
\newComment{Jeremy}{Jw}{forest}
\newComment{Sanghoon}{Sk}{marine}
\newComment{Thomas}{Th}{blue}

\newcommand{\G}{\Gamma}
\newcommand{\R}{\mathbb{R}}
\newcommand{\Z}{\mathbb{Z}}
\newcommand{\N}{\mathbb{N}}
\newcommand{\XX}{\mathcal{X}}
\newcommand{\YY}{\mathcal{Y}}
\newcommand{\ZZ}{\mathcal{Z}}
\newcommand{\BB}{\mathcal{B}}
\newcommand{\ov}[1]{\overline{#1}}

\renewcommand{\lambda}{\uplambda}

\newcommand{\mapstoandfrom}{\raisebox{-1ex}{$\overset{\textstyle{}\mapsto{}}{{}\mapsfrom{}}$}}

\makeatletter
\newcommand{\cdotfill}{\leavevmode\cleaders\hb@xt@.44em{\hss$\cdot$\hss}\hfill\kern\z@}
\makeatother

\DeclareMathOperator{\SO}{SO}
\DeclareMathOperator{\Twists}{Twists}

\DeclareMathOperator{\PMap}{PMap}
\DeclareMathOperator{\Map}{Map}
\DeclareMathOperator{\Homeo}{Homeo}
\DeclareMathOperator{\Out}{Out}
\DeclareMathOperator{\Aut}{Aut}
\DeclareMathOperator{\id}{Id}
\DeclareMathOperator{\rk}{rk}
\DeclareMathOperator{\brH}{br\textit{H}}
\DeclareMathOperator{\cork}{cork}
\DeclareMathOperator{\lk}{Lk}
\DeclareMathOperator{\lkd}{Lk^{\downarrow}}
\DeclareMathOperator{\Phid}{\Phi^{\downarrow}}
\newcommand{\<}{\langle}
  \renewcommand{\>}{\rangle}

\let\oldtocsection=\tocsection

\let\oldtocsubsection=\tocsubsection

\let\oldtocsubsubsection=\tocsubsubsection

\renewcommand{\tocsection}[2]{\hspace{0em}\oldtocsection{#1}{#2}}
\renewcommand{\tocsubsection}[2]{\hspace{1em}\oldtocsubsection{#1}{#2}}
\renewcommand{\tocsubsubsection}[2]{\hspace{2em}\oldtocsubsubsection{#1}{#2}}

\begin{document}

\begin{abstract}
    We introduce and study asymptotically rigid mapping class groups of certain infinite graphs. We determine their finiteness properties and show that these depend on the number of ends of the underlying graph. In a special case where the graph has finitely many ends, we construct an explicit presentation for the so-called \emph{pure graph Houghton group} and investigate several of its algebraic and geometric properties. Additionally, we show that the graph Houghton groups are not commensurable with other known Houghton-type groups, namely the classical, surface, braided, handlebody, and doubled handlebody Houghton groups, demonstrating that this graph-based construction defines a genuinely new class of groups.
\end{abstract}

\maketitle

\section{Introduction}

The group $\Out(F_n)$ is a classical and well-studied object in geometric group theory that is analogous in many ways to mapping class groups of surfaces.  Extending this analogy to the setting of ``big" mapping class groups, in \cite{algom-kfir2025groups} Algom-Kfir and Bestvina 
 introduced the group $\Map(\Gamma)$ for a locally finite, infinite graph $\Gamma$ as a ``big'' analog of $\Out(F_n)$. This group has been shown to exhibit many features reminiscent of mapping class groups of infinite-type surfaces (see, e.g., \cite{algom-kfir2025groups,domat2023coarse,domat2025generating,udall2024the-sphere,hill2024automorphismsspherecomplexinfinite, dickmann2024surfacesproperhomotopyequivalent}).

Like big mapping class groups of surfaces, $\Map(\Gamma)$ is uncountable.  Despite this, the \textbf{asymptotically rigid} subgroups of mapping class groups of infinite-type surfaces with either a Cantor set or finitely many ends, all accumulated by genus, form rich countable subgroups \cite{aramayona2023surface,aramayona2024asymptotic}.  
This paper aims to describe a similar class of subgroups in the context of locally finite, infinite graphs. We postpone the definitions to \Cref{sec:GraphHoughtonGroups,sec:CantorGraphs}. The group $\Map(\Gamma)$ is also closely related to the mapping class group of a 3-manifold, called the \emph{doubled handlebody} associated to $\Gamma$.  Specifically, Laudenbach’s classical results relating mapping class groups of doubled handlebodies and $\Out(F_n)$ in \cite{laudenbach1974topologie} were extended to $\Map(\Gamma)$ by Udall in \cite{udall2024the-sphere}. Thus, we also define a corresponding asymptotically rigid subgroup of the mapping class group of the doubled handlebody (see \Cref{sec:doubledHandleHoughton}). These groups are closely related to some described in \cite{aramayona2024asymptotic}.

Asymptotically rigid mapping class groups have been studied in several settings. When the underlying space has finitely many maximal ends, they exhibit similar geometric properties to classical Houghton groups, after which they are modeled. This justifies the convention of calling these asymptotically rigid mapping class groups ``Houghton'' groups.    
In particular, the classical Houghton groups \cite{houghton1978first,BROWN198745}, the braided Houghton groups  \cite{degenhardt2000endlichkeitseigeinschaften,Genevois_2022}, the surface Houghton groups \cite{aramayona2023surface}, and the handlebody Houghton groups \cite{domingozubiaga2025finitenesspropertiesasymptoticallyrigid}, are all known to be of type $F_{r-1}$ but not of type $FP_r$, where $r$ is the number of maximal ends of the space.

We show an analogous result in the context of graphs and doubled handlebodies. In keeping with the conventions mentioned above, when the graph (doubled handlebody) has finitely many ends, we call the asymptotically rigid mapping class group the \emph{graph (doubled handlebody) Houghton group}.

\begin{THM}\label{thm:finitenessGraphHoughton}
    For a graph (resp. doubled handlebody) with $r < \infty$ many ends, all accumulated by loops (resp. genus), the graph (resp. doubled handlebody) Houghton group is of type $F_{r-1}$ but not of type $FP_{r}$.  When the end space is homeomorphic to a Cantor set, the asymptotically rigid subgroup is of type $F_\infty$. 
\end{THM}

To prove this, we construct a cube complex on which these groups act and apply Brown's Criterion (\Cref{thm:BrownsCriterion,thm:BrownFl}), mimicking the approaches of \cite{Genevois_2022,aramayona2024asymptotic,domingozubiaga2025finitenesspropertiesasymptoticallyrigid}.  The cube complex we construct is inspired by the Stein--Farley complexes associated with Higman--Thompson groups \cite{stein1992groups,farley2005homological}. 

  Although these various Houghton-type groups share the same finiteness properties, and various other similarities, we show:

\begin{THM} \label{thm:MainThm_NonCommensurable}
    The classical, braided, surface, doubled handlebody, and graph Houghton groups are not pairwise commensurable. 
\end{THM}
\noindent We can also show that the handlebody Houghton group is not commensurable with all but the surface and braided Houghton groups. Whether the handlebody Houghton is commensurable to one of these is not known to the authors.  

While \Cref{thm:finitenessGraphHoughton} implies that the graph Houghton group is finitely presented when the graph has at least $3$ ends, we go further in \Cref{sec:finitePresentation} by providing an explicit presentation of the pure subgroup:

\begin{THM}
    \label{thm:PBrPresentation}
    For $r\geq 3$, the pure graph Houghton group $PB_r$ has the presentation
    $$PB_r\cong \<h_2, \ldots, h_{r}, \sigma, \tau, \eta \mid P_r\>$$
    where $P_r$ consists of (in each case, $2\leq i<j\leq n$)
    
\renewcommand{\arraystretch}{1.8}
\begin{tabularx}{\textwidth}{@{}X X X@{}}
{$r_1 : \sigma^2 = 1$} &
{$r_2 : [\sigma, \sigma^{\overline{h}_i^2}] = 1$} &
{$r_3 : (\sigma\sigma^{\overline{h}_i})^3 = 1$} \\
{$r_4 : \sigma = [h_i,h_j]$} &
{$r_5 : \sigma^{\overline{h}_i} = \sigma^{\overline{h}_j}$} &
{$r_6 : \tau^2 = 1$} \\
{$r_7 : [\tau,\sigma]^2 = 1$} &
{$r_8 : [\tau,\sigma^{\overline{h}_i}] = 1$} &
{$r_9 : \tau^{\overline{h}_i} = \tau^{\sigma}$} \\
{$r_{10} : \eta^2 = 1$} &
{$r_{11} : (\sigma\eta)^3 = 1$} &
{$r_{12} : [\eta,\sigma^{h_i}]^2 = 1$} \\
{$r_{13} : [\eta,\sigma^{h_i^2}] = 1$} &
{$r_{14} : \eta^{h_i} = \eta^{\sigma\sigma^{h_i}}$} &
{$r_{15} : [\eta, \eta^{\overline{h}_i^2}] = 1$} \\
{$r_{16} : [\eta, \tau^{h_i}] = 1$} &
{$r_{17} : ((\eta\tau)^2\tau^{\sigma})^2 = 1$} &
\\
\multicolumn{3}{@{}l@{}}{
  {$r_{18} : \sigma\eta\sigma^{\sigma^{\overline{h}_i}}\tau^{\sigma}\eta\sigma(\sigma^{\overline{h}_i}\eta\sigma^{\sigma^{\overline{h}_i}}\tau^{\sigma}\eta)^2 = 1$}
} \\
\end{tabularx}
\end{THM}
\noindent The proof of \Cref{thm:PBrPresentation} draws on the relationship of the graph Houghton group with the classical Houghton group and $\Aut(F_n)$, along with known presentations for these groups (see \cite{lee2012geometry, armstrong2008a-presentation}).
As a consequence of the proof, the pure graph Houghton group contains $\Aut(F_n)$ for every $n$. It follows from a result of Higman \cite{higman1961subgroups} that there is a finitely presented group containing $\Aut(F_n)$ for all $n$, and the pure graph Houghton group is an concrete example of such a group. As far as we are aware, this is the first explicit construction of such a group in the literature. We thank Matt Zaremsky for pointing this out to us.  

We use this presentation to investigate several classical algebraic and geometric properties of this group in \Cref{sec:GGTprops}. 
Specifically, we consider the number of ends (\Cref{thm:1ended}), the (strong) Tits alternative (\Cref{thm:TitsAlternative}), BNSR-Invariants (\Cref{thm:BSNRInvs}), the solvability of the word problem (\Cref{cor:wordproblem}), and the Dehn function (\Cref{sec:WordProblemDF}).

\subsection*{Open questions}
In addition to our main results, several natural questions arise concerning the broader structure of graph Houghton groups. Like their surface counterparts, graph Houghton groups contain, up to conjugation, all end-periodic graph maps in the sense of \cite{meadow2024end}, recently studied in \cite{he2024relative}. Given the explicit nature of these groups, it is natural to ask whether the graph (surface) Houghton groups can be used to study end-periodic maps in some meaningful way. In the context of surfaces, one might wish to obtain an explicit presentation with which to study end-periodic surface mapping classes in an analogous way. To do this, one would need presentations for finite type surfaces that are compatible with the rigid structure, and to determine what the commutator of shifts is in terms of Dehn twists.

Beyond this dynamical connection, structural questions arise from the observation that the graph Houghton group contains $\Aut(F_n)$ for all $n$ (as implied by the proof of \Cref{thm:PBrPresentation}).
One might ask whether the Boone--Higman conjecture~\cite{boone1974algebraic} for the graph Houghton group is true, which would yield a new proof of the Boone--Higman conjecture for $\Aut(F_n)$, recently established in \cite{belk2025boonehigmanembeddingsmathrmautfnmapping}. More broadly, one may ask whether the Boone--Higman conjecture holds for all Houghton-type groups.

The graph and surface Houghton groups are dense in their corresponding mapping class groups. One could also ask if other dense subgroups which are finitely generated (or have higher finiteness properties) exist in mapping class groups of graphs and surfaces with more complicated end spaces. For example, in the two ended case is there a finitely presented dense subgroup? For the case with $r<\infty$ many maximal ends, one can also ask if the Houghton-type groups are the only dense subgroups of type $F_{r-1}$ (up to conjugation).

In \Cref{sec:WordProblemDF}, we discuss an open question related to a strategy for computing upper bounds on the Dehn function of the pure graph Houghton group. 

\subsection*{Outline of paper}
In \Cref{sec:GraphHoughtonGroups}, we introduce the infinite-type graphs and their doubled handlebodies that we study, along with their rigid structures, and the graph and doubled handlebody Houghton groups that arise from these structures. We compare these groups to four other variations of Houghton groups in \Cref{sec:HoughtonVariants}, showing that the graph and doubled handlebody Houghton groups are not commensurable to each other and are not commensurable to any of the other variants.

Next we begin the in-depth study of graph Houghton groups, producing a finite presentation when the graph has at least $3$ ends in \Cref{sec:finitePresentation}. We show that the graph and doubled handlebody Houghton groups with $r$ ends are of type $F_{r-1}$ but not of type $FP_{r}$ in \Cref{sec:BrownsCrit} and \Cref{sec:proofOfMainThm} by analyzing their action on a cube complex. We study the asymptotically rigid mapping class groups of the graph with a Cantor set of ends all accumulated by loops in \Cref{sec:CantorGraphs}, and show that these groups are of type $F_{\infty}$. For a particular rigid structure on this graph, we show that its homology groups are isomorphic to the stable homology groups of $\{\Aut(F_n)\}_{n\geq 1}$ (see \Cref{thm:StableHomology}). 

We conclude with a study of some basic geometric group theory properties of these Houghton group variants, focusing on the graph Houghton group, in \Cref{sec:GGTprops}.

\subsection*{Acknowledgments}
 We thank Javier Aramayona, Mladen Bestvina, Noel Brady, Luke Dechow, George Domat, Jes\'us Hern\'andez-Hern\'andez, Chris Leininger, Rachel Skipper, and Jing Tao for helpful conversations. 
We also thank Anthony Genevois, Priyam Patel, and Matt Zaremsky for insightful comments on the earlier draft of this paper. 
 We thank the organizers of the 2025 Wasatch Topology Conference, where the final steps of the project were completed. 

The first author acknowledges support from NSF grants RTG DMS-1840190 and NSF DMS–2046889.  The second acknowledges partial support from the New Faculty Startup Fund (700-20250069) at Seoul National University and KIAS Individual Grant (HP098501) via the June E Huh Center for Mathematical Challenges at Korea Institute for Advanced Study. The third author acknowledges support from NSF
grant DMS-1745670. The fourth author acknowledges support from NSF grant DMS-2304920.

\tableofcontents

\section{Graph and doubled handlebody Houghton groups}\label{sec:GraphHoughtonGroups}

\subsection{Mapping class group of locally finite, infinite graphs} 

Let $\Gamma$ be a connected, locally finite, infinite graph. The \emph{end space} of $\Gamma$ is defined as
$E(\Gamma) := \varprojlim_i \,  \pi_0 (\Gamma \setminus K_i),$
where $K_1 \subset K_2 \subset \dots \Gamma$ form a compact exhaustion of $\Gamma$ and the limit is taken over the inverse system $\{\pi_0(\Gamma\setminus K_i) \twoheadrightarrow \pi_0(\Gamma\setminus K_{i -1})\}$. We take the set $E(\G)$ to be topologized by this inverse limit. It is well known that the natural topology on $\G \cup E(\G)$ is compact. An end is \emph{accumulated by loops} if every neighborhood of it in $\G$ has infinite rank. The closed set $E_\ell(\Gamma) \subset E(\Gamma)$ denotes the subset of all ends accumulated by loops.  

A proper map $f: \Gamma \to \Gamma$ is a \emph{proper homotopy
equivalence} if there exists a proper map $g: \Gamma \to \Gamma$ such that $f\circ g$ and $g\circ f$ are
properly homotopic to the identity.  In \cite{ayala1990proper}, locally finite, infinite graphs are classified up to proper homotopy equivalence by the homeomorphism type of the pair $(E(\Gamma), E_\ell(\Gamma))$, and their rank (possibly infinite). The proof of this classification uses the fact that every locally finite infinite graph has a \emph{standard model}. A slightly more generalized notion is the \textbf{standard form} of a graph, introduced in \cite{domat2023coarse}, which is a tree with loops attached to some of the vertices.

The \emph{mapping class group} of a locally finite, infinite graph $\Gamma$ is defined to be the set of proper homotopy equivalences of $\Gamma$, up to proper homotopy (see \cite{algom-kfir2025groups}); that is,  
\begin{equation*}
        \Map(\Gamma) = \{\text{proper homotopy equivalences } \Gamma \to \Gamma\}/\text{proper homotopy}.
\end{equation*}
The subgroup fixing the set of ends of $\Gamma$ pointwise is called the \emph{pure mapping class group}, and is denoted $\PMap(\Gamma)$.  The subgroup of proper homotopy classes with a \emph{compactly supported} representative is denoted $\Map_c(\Gamma)$. 
Note $\Map_c(\G) \le \PMap(\G)$.

    \subsection{Graph Houghton groups}
    Let $\G_r$ be a locally finite, infinite graph with $r$ ends, all of which are accumulated by loops. We may assume it is in standard form, so that each loop is incident to a vertex, and further assume the center vertex $x_0$ has no loop attached. For each $h>0$, let $Y=Y^h$ be a finite graph of rank $h$, obtained from a straight line segment with $h+2$ vertices, among which the middle $h$ vertices are incident to single-edged loops. Denote the two valence 1 vertices by $\partial_+Y$ and $\partial_-Y$.

    For $g \ge 0$ and $h >0$, consider a decomposition of $\G_r$ into subgraphs: $\G_r = C \cup \bigcup_{j \in J} Y_j$,
    where the \emph{core} $C$ is a connected subgraph of rank $g$ containing $x_0$ with $r$ valence 1 vertices, $J$ is some countable set, and each $Y_j$ is called a \emph{piece}, which is a subgraph of $\G$ isomorphic to $Y$; each piece is equipped with a choice of \emph{marking}, i.e., a proper homotopy equivalence $\phi_i: Y_i \to Y^h$ preserving boundaries $\partial_{\pm}Y_i \mapsto \partial_{\pm}Y^h$, along with a choice of homotopy inverse, denoted $\phi_i^{-1}$ (see \Cref{fig:rigidstructure}).

    \begin{figure}[ht!]
        \centering
        \includegraphics[width=.6\textwidth]{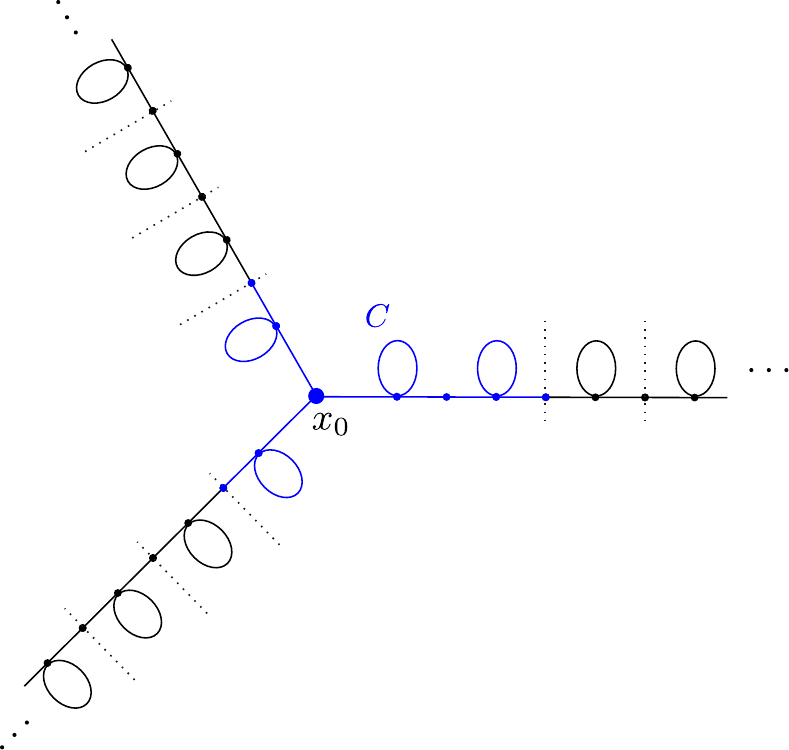}
        \caption{The graph $\G_3$ with a $(4,1)$-rigid structure. Note the core $C$ is required to contain the center point $x_0$.}
        \label{fig:rigidstructure}
    \end{figure}
    
    Such a decomposition, along with a choice of markings, is called a \emph{$(g,h)$-rigid structure} on $\G_r$ if they satisfy the following conditions:
    \begin{itemize}
        \item All subgraphs in the decomposition have disjoint interiors,
        \item Each piece $Y_j$ is either disjoint from $C$, or meets $C$ at one point 
        $\partial_-Y_j \in \partial C$, where $\partial C$ is the frontier of $C$.
    \end{itemize}
    A \emph{suited subgraph} $Z \subset \G_r$ is a connected union of the core and a (possibly empty) finite set of pieces.

    Given a $(g,h)$-rigid structure on $\G_r$, an \emph{asymptotically rigid mapping class} is the proper homotopy class of a proper homotopy equivalence $f \colon \G_r \to \G_r$, such that for some suited subgraph $Z$, $f|_Z$ is a proper homotopy equivalence onto its image, $f(Z)$ is a suited subgraph, and for any piece $Y_j \subset \ov{\G_r \setminus
          Z}$, there is a piece $Y_i$ in $\ov{\G_r \setminus f(Z)}$ so that
        $f(Y_j) = Y_i$ and $f|_{Y_j} = \phi_i^{-1}\phi_j$. Call such a suited
        subgraph $Z$ a \emph{defining graph} for $f$. We call the group of all such maps the \emph{graph Houghton group} (with respect to the triple $(g, h, r)$), which we denote $B(g,h,r)$. For each $r$, $B(g,h,r)$ depends on the rigid structure on $\G_r$, but as we will see in \Cref{rmk:conjugateHoughtonGroups}, all possible choices are conjugate in $\Map(\G_r)$.
        \begin{RMK}
            Note that in our definition of asymptotic rigidity, we have specified that the image of a piece outside a defining graph must be a piece outside the image of the defining graph. This is automatically true for homeomorphisms, but not for proper homotopy equivalences. This is important in \Cref{sec:BrownsCrit}. For example, let $Z$ be the core union a single piece, and let $f$ map the loop in $Z$ over itself and the next loop in its ray. Then $Z$ is not a defining graph for $f$, but the core union both loops is.
        \end{RMK}
      
\begin{EX}\label{ex:shiftMaps}

An important class of elements in $\Map(\Gamma_r)$ are the \emph{loop shifts}.  These are the graph analog of the handle shifts in the context of surfaces introduced in \cite{Patel_2018}.  Heuristically, they can be thought of as a translation of loops between distinct ends accumulated by loops as in \Cref{fig:b2}.  See \cite[Definition 3.11]{domat2023coarse} for a precise definition.  Labeling the ends of $\G_r$, $e_1, \dots, e_{r}$, we denote by $h_i$, the loop shift from $e_1$ to $e_i$, where $i = 2, \dots, r$ (see \Cref{fig:example,fig:br}).

Given a $(g,h)$-rigid structure on $r$, compositions of $h$-th powers of loop shifts $(h_i^h)^k = h_i^{kh}$, for $k \in \mathbb{Z}$, are examples of asymptotically rigid maps.
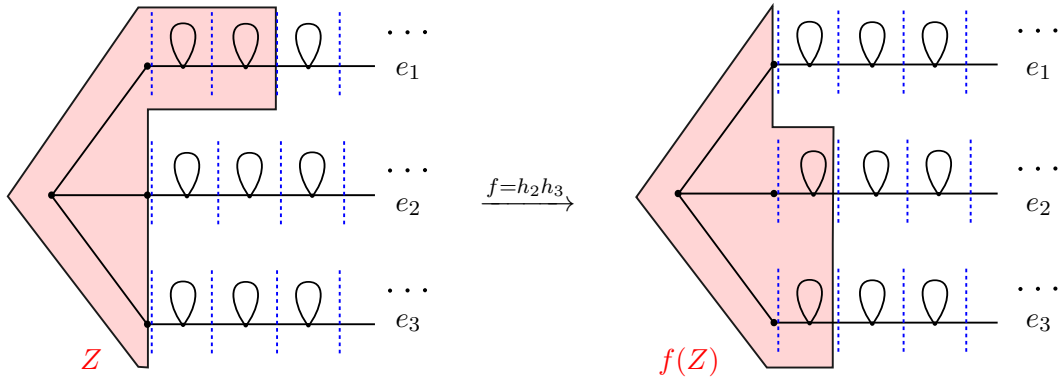
\begin{figure}[ht!]
        \centering
        \begin{overpic}[width=\textwidth]{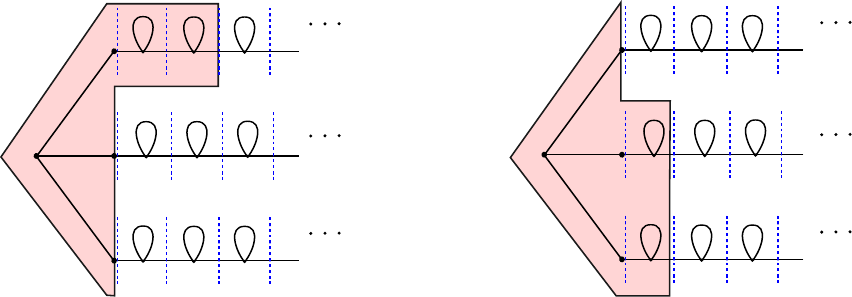}
        \put(45,15){$\xrightarrow{f = h_2 h_3}$}
        \put(7,0){\color{red}$Z$}
        \put(62,0){\color{red}$f(Z)$}
        \put(37,28){$e_1$}
        \put(37,15){$e_2$}
        \put(37,4){$e_3$}
        \put(97,28){$e_1$}
        \put(97,15){$e_2$}
        \put(97,4){$e_3$}
        \end{overpic}
        \caption{The suited subgraph $Z$ on the left is a defining graph for the product of loop shifts $f = h_2 h_3$.}
        \label{fig:example}
    \end{figure}
\end{EX}
       
        \begin{RMK}\label{rmk:conjugateHoughtonGroups} 
        For a fixed pair $(g,h)$, different $(g,h)$-rigid structures on $\G_r$ define conjugate graph Houghton groups in $\Map(\G_r)$.  Indeed, suppose that two $(g,h)$-rigid structures $\G_r = C \cup \bigcup_{i,j \in \N\times \{1,\dots,r\}} Y_{i,j}$ with $\phi_{i,j} : Y_{i,j} \to Y^h$, and $\G_r = C' \cup \bigcup_{i,j \in \N\times \{1,\dots,r\}} Y'_{i,j}$ with $\phi'_{i,j} : Y'_{i,j} \to Y^h$ define $B(g,h,r)$ and $B'(g,h,r)$ respectively. Here, $Y_{i,j}$ is the $i$th piece out from the core in the $j$th end. We find a $\psi \in \Map(\G_r)$ whose inner automorphism restricts to an isomorphism $B(g,h,r) \cong B'(g,h,r)$.

        Since the two cores $C$ and $C'$ have the same rank $g$, let $\psi|_{C}$ be a proper homotopy equivalence that sends $C$ to $C'$, and such that $\psi|_{C}(\partial_- Y_{1,j})
        = \partial_- Y'_{1,j}$.  
        On the piece $Y_{i,j}$, define $\psi|_{Y_{i,j}} = (\phi_{i,j}')^{-1} \circ \phi_{i,j}$.  This way, $\psi$ is compatible with the markings, i.e., the following diagram commutes.
\[\begin{tikzcd}[ampersand replacement=\&,cramped]
	\& {Y^h} \\
	{Y_{i,j}} \&\& {Y'_{i,j}}
	\arrow["{\phi_{i,j}}", from=2-1, to=1-2]
	\arrow["{\psi|_{Y_{i,j}}}"', from=2-1, to=2-3]
	\arrow["{\phi_{i,j}'}"', from=2-3, to=1-2]
\end{tikzcd}\]

        Now let $f \in B(g,h,r)$ and take a defining graph $Z$ for $f$.  On $Z$, let $f' = \psi \circ f \circ \psi^{-1}$, and outside of $Z$ assume that $f'$ is rigid.  Clearly $f' \in B'(g,h,r)$, as $\psi(Z)$ and $f'(\psi(Z)) = \psi(f(Z))$ are suited subgraphs in the second rigid structure, so $\psi(Z)$ is a defining graph for $f'$. By the construction of $\psi$ we see $f = \psi^{-1} \circ f' \circ \psi$ on all of $\Gamma_r$.  
 
        \end{RMK}

        Thus, the following definition gives a well-defined subgroup of $\Map(\G_r)$ up to conjugation.   
        
\begin{DEF}
        Let $g \ge 0, h>0$ and $r \ge 1$. 
        Denote by $B(g,h,r)$ the \emph{graph Houghton group} of some $(g,h)$-rigid structure on $\G_r$. Let the \emph{pure graph Houghton group}, $PB(g,h,r)$, be the kernel of the action of $B(g,h,r)$ on the (finite) end space of $\G_r$.
    \end{DEF}

\subsection{Mapping class group of doubled handlebodies}\label{sec:MCGofDoubHB}

Fix a locally finite graph $\G$. We associate to it a $3$-manifold called the \textit{doubled handlebody}, denoted $M_{\G}$. This manifold is defined as the union of two copies of a regular neighborhood of $\G$, thought of as being embedded in $\R^3$, glued along their boundaries by the identity map.  Recall that the mapping class group of $M_\G$, denoted $\Map(M_\G)$, is the group of orientation-preserving homeomorphisms of $M_\G$ up to isotopy. 
\par 
We think of $\G$ as sitting in $M_{\G}$, and we let $i:\G \to M_{\G}$ denote the inclusion map.
There is a natural retraction $\rho:M_{\G}\to \G$ defined as follows. 
Consider a closed regular neighborhood $B$ of $\G$ in $\R^3$. 
There is a deformation retraction of $B$ onto $\G$. We define $\rho$ on $B$ as the terminal map of this deformation retraction. We can write $M_\G = B \cup B'$ for some $B' \cong B$ where $B' \cap B$ is a surface. On $B'$, define $\rho$ by first reflecting this set onto $B$ via a map which is the identity on the surface $B \cap B'$, followed by the retraction onto $\G$ as it is defined on $B$ (see \Cref{fig:DHB}). 
\begin{figure}[ht!] 
    \centering
    \begin{overpic}[width=.8\textwidth]{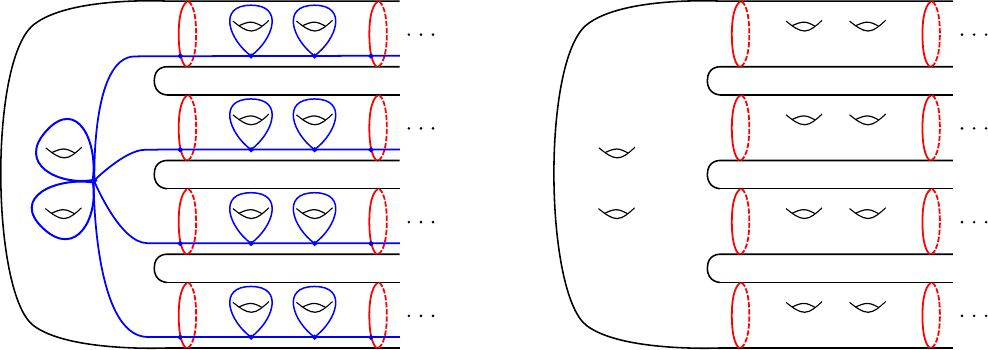}
        \put(-2,31){$B$}
        \put(53.5,31){$B'$}
        \put(5,8){\color{blue}{$\G_4$}}
    \end{overpic}
    \caption{A $(2,2)$-rigid structure on $M_{\G_4} = B \cup B'$. }
    \label{fig:DHB}
\end{figure}
\par
To state the next theorem, we must define the subgroup of \textit{sphere twists}. Choose a generator $\ell:[0,1] \to \text{SO}(3)$ of $\pi_1(\text{SO}(3), id)\cong \Z/2$. Given a smoothly embedded essential non-peripheral sphere $S$ of $M_{\G}$, fix a regular neighborhood $N \cong S\times [0,1]$ of $S$. The \emph{sphere twist} $T_S$ is defined as
\begin{equation*}
    T_S(p) = \begin{cases}
        (\ell(t)\cdot x, t) & \text{if } p=(x,t)\in N \\ 
        id & \text{if } p\notin N 
    \end{cases}
\end{equation*}
where $\ell(t)\cdot x$ denotes the action of $\ell(t)\in \SO(3)$ on $x\in S\cong \mathbb{S}^2 \subset \R^3$. We denote by $\Twists(M_{\G})$ the closure of the subgroup of $\Map(M_{\G})$ generated by compositions of (isotopy classes of) sphere twists on finite collections of disjoint isotopy classes of spheres.  
\begin{THM}[{\cite[Theorem 1.1]{udall2024the-sphere}}]\label{thm:doubledHandlebodytoGraph}
    The map $\rho$ induces a surjective homomorphism 
    $$ \Psi:\Map(M_{\G})\to \Map(\G),$$
    where $\Psi(f)= \rho \circ f \circ i$. The kernel of $\Psi$ is $\Twists(M_{\G})$, and the corresponding short exact sequence splits.
\end{THM}

\subsection{Doubled handlebody Houghton groups} \label{sec:doubledHandleHoughton}
One can define the doubled handlebody Houghton group $\BB(g,h,r)$ as a subgroup
of $\Map(M_{\G_r})$ in the same way as $B(g,h,r) \le \Map(\G_r)$ is defined. To
do this, in the definition for $B(g,h,r)$, we replace subgraphs of $\G_r$ with
submanifolds of $M_{\G_r}$, and make the natural modification on the definition
of the rigid structure. Namely, we define the \emph{core} to be a connected
submanifold which is a doubled handlebody of rank
$g$ (i.e., its fundamental group has rank $g$) with $r$ boundary components. The \emph{pieces} are rank $h$ doubled
handlebodies with $2$ boundary components. The \emph{$(g,h)$-rigid structure}
on $M_{\G_r}$ is defined in the same way, replacing subgraphs with submanifolds,
and also assuming that pieces are either disjoint or meet in at precisely one of
their boundary components (see \Cref{fig:DHB}).

Then the \emph{doubled handlebody Houghton group}
$\BB(g,h,r)$ is the subgroup of $\Map(M_{\G_r})$ consisting of mapping classes with a representative homeomorphism $f:M_{\G_r} \to M_{\G_r}$ such that for some suited submanifold $\ZZ$, $f(\ZZ)$ is also a suited submanifold, and so that any piece $\YY$ outside $\ZZ$ is mapped to a piece respecting the gluing maps from the reference piece to each piece in $M_{\G_r}$. Say such a suited submanifold $\ZZ$ is a \emph{defining submanifold} for $f$.

In \cite{aramayona2024asymptotic}, the groups they consider (which are analogous to those just defined) are groups of homeomorphisms up to \textit{proper isotopy}, instead of just isotopy. By proper isotopy, we mean an isotopy that respects the rigid structure. The only difference this extra assumption makes is that the mapping class group of suited submanifolds \emph{injects} into the group being considered. In our case, for each suited submanifold, there is a finite kernel generated by the sphere twists on the boundary components of the suited submanifold. This kernel will not be important for us to consider.

\begin{RMK}
    Henceforth, plain-text $Y, Z$ etc.\ shall denote pieces and suited subgraphs in the graph, and calligraphic $\YY,\ZZ$ shall be used for the doubled handlebody.
\end{RMK}

\par 
We fix a rigid structure on $M_{\G_r}$ so that it is compatible with the retraction $\rho:M_{\G_r} \to \G_r$ and the rigid structure on $\G_r$ in the following sense. The retraction $\rho$ sends the core $\mathcal{C}$ in $M_{\G_r}$ to the core $C$ in $\G_r$, and $\rho^{-1}(C)=\mathcal{C}$. Further, $\rho$ induces a bijection between the pieces of the rigid structures, and if $\mathcal{Y}$ is a piece in $M_{\G_r}$ and $Y=\rho(\mathcal{Y})$, then $\rho^{-1}(Y)=\mathcal{Y}$. Also, if $f\in \BB(g,h,r)$ is rigid outside of $\ZZ$ and $\YY_j=\rho^{-1}(Y_j)\subset \overline{M_{\G_r}\setminus \ZZ}$ is a piece so that $f(\YY_j)=\YY_i=\rho^{-1}(Y_i)$, then $\Psi(f)|_{Y_j}=\phi_i^{-1}\phi_j$. The fact that we can choose these rigid structures so that $\Psi(f)|_{Y_j}=\phi_i^{-1}\phi_j$ follows from the surjectivity from \Cref{thm:doubledHandlebodytoGraph} applied to $Y^h$.
\begin{LEM}\label{lem:psirestriction}
    The homomorphism $\Psi$ from \Cref{thm:doubledHandlebodytoGraph} restricts to a surjective homomorphism 
    $$\Psi_B: \BB(g,h,r)\to B(g,h,r)$$
    whose kernel consists of finite products of sphere twists.
\end{LEM}
\begin{proof}
    The statement about the kernel is immediate, as $\BB(g,h,r)$ contains all finite products of sphere twists, and it contains no infinite products as such elements are not eventually the identity outside of some compact submanifold of $M_{\G_r}$.
    \par 
    To see that $\Psi(\BB(g,h,r))\subset B(g,h,r)$, note that if $f\in \BB(g,h,r)$ has a defining submanifold $\ZZ$, then by construction we have that  $\Psi(f)|_{\rho(\ZZ)}$ is a proper homotopy equivalence and $\Psi(f)(\rho(\ZZ))=\rho(f(\ZZ))$ is also a suited subgraph of $\G_r$. Further, the compatibility of the rigid structures ensures that $\Psi(f)$ has the required properties outside of $\rho(\ZZ)$. 
    \par 
    Now we show $\Psi_B$ is surjective. Suppose $g\in B(g,h,r)$. We can give an explicit element $g' \in \BB(g,h,r)$ so that $\Psi(g')=g$ as follows: Let $Z$ be a defining graph for $g$. By the surjectivity of the map arising from  \Cref{thm:doubledHandlebodytoGraph} for finite type graphs, we can define $g'$ on $\rho^{-1}(Z)$ so that $\Psi(g')|_Z=g$. On any piece $\YY_j=\rho^{-1}(Y_j)\subset \overline{M_{\G_r}\setminus \rho^{-1}(Z)}$ so that $g(Y_j)=Y_i$, define $g'$ to be such that $g'|_{\YY_j}= (\phi_i')^{-1}\phi_j'$, where the maps $\phi_i'$ and $\phi_j'$ are those that appear in the rigid structure for $\BB(g,h,r)$.  Then by construction $g'\in \BB(g,h,r)$, and $\Psi(g')=g$. 
\end{proof}

\begin{RMK}\label{rmk:SphereTwistsGenerators}
    The set of generators for the sphere twists can be chosen to occur about a fixed choice of disjoint spheres, where for each piece $\mathcal{Y}$ in the rigid structure of $\BB(g, h, r)$, there are $h$ generators appearing.
\end{RMK}

\begin{NOTN} \label{rmk:notations}
For the rest of the paper, we fix the notation for inverses, commutators, and conjugations as follows: Let $G$ be a group and $a,b \in G$.
\begin{itemize}
    \item The inverse $\bar{a}:= a^{-1}$;
    \item the commutator $[a,b]:= \bar{a}\bar{b}ab$; 
    \item the conjugation $a^b:= ba\bar{b}$, and towers of exponents are always read top to bottom, e.g., $a^{b^c} = a^{(b^c)}$. 
\end{itemize}
 In addition, group composition is considered functionally, and group actions are always on the left.
\end{NOTN}
\section{Relation to variants of the Houghton group}
        \label{sec:HoughtonVariants}

        Graph Houghton groups are closely related to several other families of groups that generalize the classical Houghton groups, $H_r$, by changing the underlying topological space. These include the braided Houghton group $\brH_r$, the surface Houghton group $S_r$, the handlebody Houghton group $\mathcal{H}_r$, and the doubled handlebody Houghton group $\mathcal{B}_r$. In this section, we recall these groups and describe their relationships with the graph Houghton group (see \Cref{fig:HoughtonVariants}). While closely related, none of these variants of Houghton groups are commensurable with the graph Houghton group, nor with one another (except possibly $\mathcal{H}_r$ with $S_r$ or $\brH_r$) as we shall see in \Cref{thm:MainThm_NonCommensurable}.

\begin{figure}[ht!]
    \centering
\[\begin{tikzcd}[ampersand replacement=\&]
	{S_r} \& {\mathcal{H}_r} \& {\mathcal{B}_r} \\
	{\brH_r} \& {B_r} \\
	{H_r}
	\arrow["{(1)}"', hook', from=1-2, to=1-1]
	\arrow["{(2)}", from=1-2, to=1-3]
	\arrow["{(4)}"', from=1-2, to=2-2]
	\arrow["{(5)}", shift left=3, two heads, from=1-3, to=2-2]
	\arrow["{(3)}", hook, from=2-1, to=1-1]
	\arrow["{(6)}", hook, from=2-1, to=2-2]
	\arrow["{(7)}"', two heads, from=2-1, to=3-1]
	\arrow[hook', from=2-2, to=1-3]
	\arrow["{(8)}"', hook, from=3-1, to=2-2]
\end{tikzcd}\]
    \caption{(1) is obtained by restricting the action of a mapping class to its surface boundary (see \cite[Section 1.5]{domingozubiaga2025finitenesspropertiesasymptoticallyrigid} for details); (2) is induced by taking the double of the map; (3) follows from \cite[Section 5]{torgerson2024bnsrinvariantssurfacehoughtongroups}; (4) is induced by retraction; (5) is the map in \Cref{lem:psirestriction} and the splitting in \Cref{thm:doubledHandlebodytoGraph}; (6) is obtained from \cite[Theorem 1.7]{dickmann2024surfacesproperhomotopyequivalent} as explained in \Cref{ssec:BraidedHoughton}; (7) is the surjection induced by forgetting the braiding; (8) is the inclusion coming from identifying the permutations of points with the permutations of loops. 
    We do have that (5) $\circ$ (2) = (4), but no other triangles in the diagram commute.  
    }
    \label{fig:HoughtonVariants}
\end{figure}

    \subsection{Classical Houghton groups}\label{ssec:ClassicalHoughton} Houghton groups were originally introduced in \cite{houghton1978first}.  
    For a formal definition, see \cite[Definition 2.1]{lee2012geometry}. They can be roughly defined as follows: For a positive integer $r$, consider in the plane $r$ distinct ordered rays of isolated points going to infinity. The Houghton group $H_r$ is the group of bijections on the union of these rays of points which are eventually translations on each ray. 

    Houghton's groups are well studied. 
    Their finiteness properties mirror \Cref{thm:finitenessGraphHoughton}, namely Brown \cite{BROWN198745} proved $H_r$ is of type $F_{r -1}$ and not of type $FP_{r}$.  Explicit presentations and bounds on the Dehn function are proven and discussed in \cite{lee2012geometry}. 

    For $r\geq 3$, the subgroup of $B_r$ generated by $h_2, \ldots, h_{r}$ (the maps discussed in \Cref{ex:shiftMaps}) is isomorphic to $H_r$. Indeed, the shift maps preserve the loops and their orientations, and there is a natural bijection between the loops and the points making up the rays that $H_r$ acts on. The induced action of the $h_i$'s on the points coming from this bijection agrees with the generating set of $H_r$ appearing in \cite{lee2012geometry}. A similar discussion holds for $r=2$ using $h_2$ and $\sigma$, a transposition of two neighboring points. For $r=1$, one can use the compactly supported symmetric group on $\N$ acting on the loops of $\Gamma_1$.

    \subsection{Braided Houghton groups}\label{ssec:BraidedHoughton} The \textit{braided Houghton group} $\brH_r$ was introduced by Degenhardt in his thesis \cite{degenhardt2000endlichkeitseigeinschaften}. In \cite[Theorem 5.24]{Genevois_2022}, Genovois--Lonjou--Urech determine the finiteness properties and prove that $\brH_r$ is of type $F_{r-1}$ but not of type $FP_{r}$. In \cite{funar2007braided}, Funar showed that this group is isomorphic to the asymptotically rigid mapping class group of the planar surface $\Sigma_r$, defined as follows.

    Let the tree $T_r$ be the wedge of $r$ rays at a common point. Viewing $T_r$ as a graph embedded in the plane, we obtain a surface $\Sigma_r$ (with noncompact boundary) by taking a regular neighborhood of $T_r$ and introducing a puncture for each edge of $T_r$. A rigid structure on $\Sigma_r$ is a decomposition into polygons via a collection of pairwise non-intersecting arcs with endpoints on the boundary of $\Sigma_r$, such that each arc contains a puncture and crosses $T_r$ once transversely (see \Cref{fig:BraidedHoughton}).

    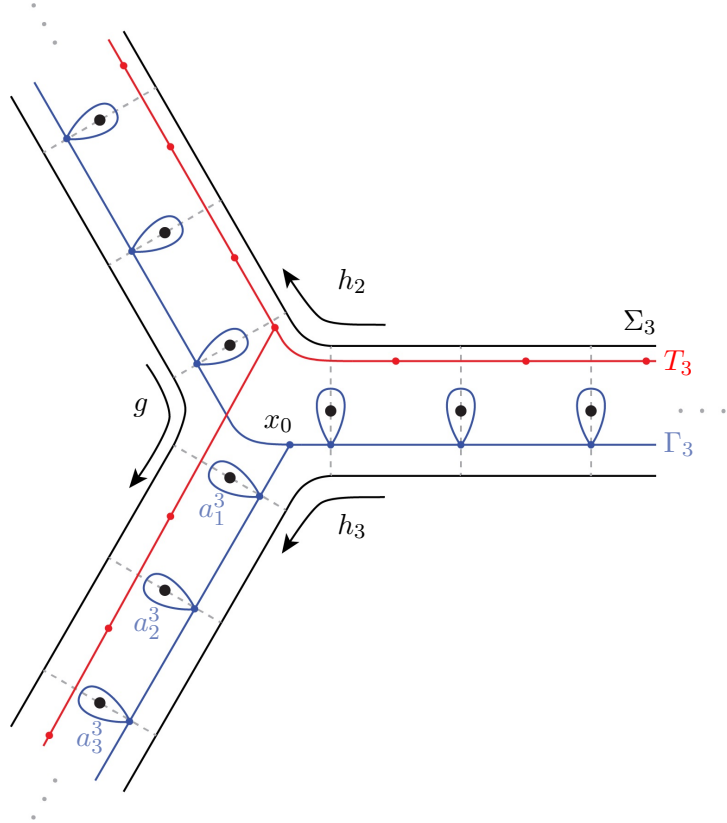
\begin{figure}[ht!]
    \centering
    \begin{overpic}[width=.68\textwidth]{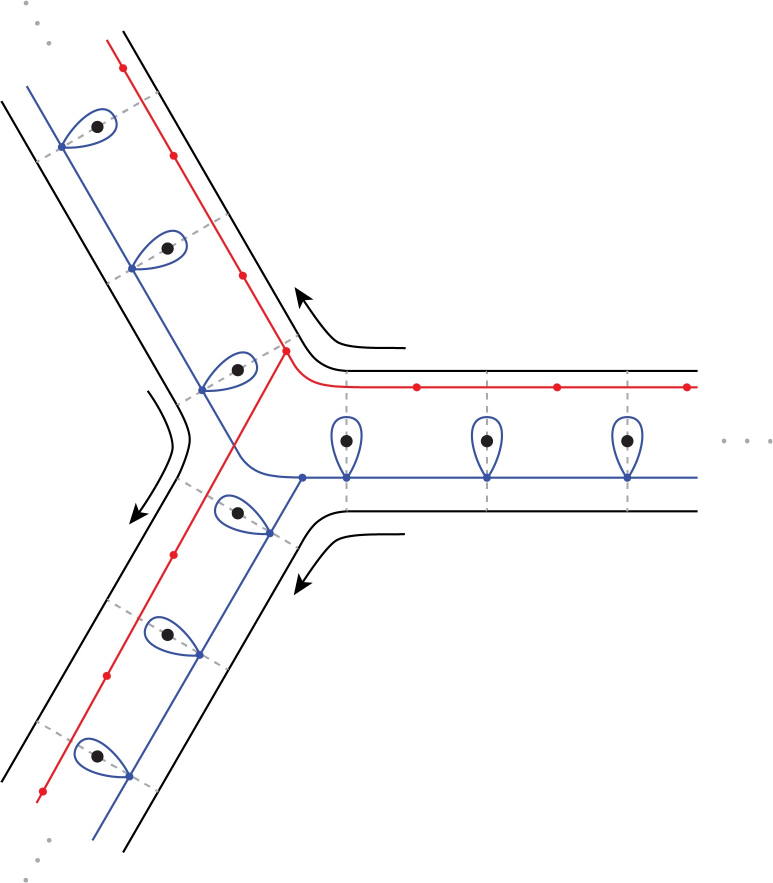}
        \put(80,55){\textcolor{red}{$T_3$}}
        \put(80,45){\color[HTML]{6b80bc}{$\Gamma_3$}}
        \put(75,60){$\Sigma_3$}
        \put(40,65){$h_2$}
        \put(40,35){$h_3$}
        \put(31,48){$x_0$}
        \put(15,50){$g$}
        \put(23,37){\color[HTML]{6b80bc}{$a_1^3$}}
        \put(15,23){{\color[HTML]{6b80bc}{$a_2^3$}}}
        \put(8,9){\color[HTML]{6b80bc}{$a_3^3$}}
    \end{overpic}
    \caption{The surface $\Sigma_3$ is shown with the core tree $T_3$ in red and the homotopy equivalent graph $\Gamma_3$ in blue. The bold dots represent the punctures of $\Sigma_3$, and gray dotted lines illustrate the rigid structure on $\Sigma_3$. The puncture shifts $h_2$, $h_3$, and $g$ are also indicated. Among these, only $h_2$ induces a loop shift, but $h_3$ does not since it acts by conjugation as $(h_3)_*(a_1^3) = a_1^3 a_2^3 \overline{a_1^3}$, where $a_i^j \in \pi_1(\Gamma_3, x_0)$ is based at the center point $x_0$.}
    \label{fig:BraidedHoughton}
\end{figure}

    A homeomorphism of $\Sigma_r$ is called \emph{asymptotically rigid} if it maps all but finitely many polygons of the rigid structure to other polygons. The \emph{asymptotically rigid mapping class group} is the subgroup of $\Map(\Sigma_r)$ consisting of all mapping classes with an asymptotically rigid representative.  The \emph{braided Houghton group} $\brH_r$ is the subgroup of the asymptotically rigid mapping class group that stabilizes the maximal ends of $\Sigma_r$.  It has index $r$ in the asymptotically rigid mapping class group.   
    Note that the rigid structure on $\Sigma_r$ can be defined in other equivalent ways (see \cite[Remark 2.6]{Genevois_2022}).  We opt for this definition to most obviously illustrate the connection to the graph Houghton group $B_r$.

    By retracting the surface $\Sigma_r$ onto the core graph $\Gamma_r$ shown in \Cref{fig:BraidedHoughton}, we see that Dickmann--Hoganson--Kwak \cite[Theorem 1.7]{dickmann2024surfacesproperhomotopyequivalent} implies that $\brH_r$ embeds into $\Map(\Gamma_r)$. Moreover, the image of this embedding lies in $B_r$, since the rigid structures on the surface and the graph (though defined differently) are compatible in the following sense: every suited subgraph of $\Gamma_r$ is contained in the retract of some sufficiently large suited subsurface of $\Sigma_r$, and conversely, the retract of any suited subsurface of $\Sigma_r$ is contained in a sufficiently large suited subgraph of $\Gamma_r$. 

    While the (proper) deformation retraction of $\Sigma_r$ onto $\Gamma_r$ gives rise to an embedding $\brH_r \hookrightarrow B_r$, puncture shifts in $\Sigma_r$ do not always correspond directly to loop shifts in $\Gamma_r$. Some puncture shifts do induce the standard loop shifts—for example, the element $h_2$ in \Cref{fig:BraidedHoughton}. However, other shifts behave differently. Shifts between ends located on opposite sides of the core tree, or between ends on the same side as the central node, induce loop shifts post-composed with partial conjugations (as seen with $h_3$ and $g$ in \Cref{fig:BraidedHoughton}). Although these puncture shifts still send each loop to a conjugate of the expected one, the conjugators themselves vary depending on the loop being shifted. Since these conjugations are not given by inner automorphisms of the whole group, they yield distinct elements in the outer automorphism group. This discrepancy is reflected in the fact that commutators of generating shifts in $\brH_r$ are infinite order, but in $B_r$ commutators of the same form are of order $2$ (see \Cref{ex:loopshift_commutator}).

\subsection{Surface Houghton groups}

The surface Houghton group, $S_r$, was first introduced in \cite{aramayona2023surface}. Its definition closely parallels that of the graph Houghton group. These groups share the same finiteness properties as the variants discussed in this section (see \cite[Theorem 1.1]{aramayona2023surface}). Surface Houghton groups are classified up to isomorphism, commensurability, and quasi-isometry in \cite{aramayona2024isomorphisms}, and their BNSR invariants are computed in \cite{torgerson2024bnsrinvariantssurfacehoughtongroups}.   Furthermore, the surface and braided Houghton groups are closely related, as discussed in \cite[Section 5]{torgerson2024bnsrinvariantssurfacehoughtongroups}. 

\subsection{Handlebody Houghton groups} 
Asymptotically rigid handlebody groups are defined and studied by Domingo-Zubiaga in \cite{domingozubiaga2025finitenesspropertiesasymptoticallyrigid}.  Restricting elements of this subgroup of the handlebody group to the boundary surface induces an embedding into the asymptotically rigid surface mapping class group.  This relationship is not sufficient to deduce the finiteness properties of handlebody groups, as explained in \cite[Section 1.5]{domingozubiaga2025finitenesspropertiesasymptoticallyrigid}. 
However, the finiteness properties of asymptotically rigid handlebody groups mirror those of the groups discussed in this section.  In the case when the handlebody has finitely many ends, the asymptotically rigid handlebody group $\mathcal{H}_r$ might, by analogy, be called the ``handlebody Houghton group''. By \cite[Theorem II]{domingozubiaga2025finitenesspropertiesasymptoticallyrigid}, $\mathcal{H}_r$, is of type $F_{r-1}$ but not $ FP_r$.

\subsection{Non-commensurability}
\label{ssec:noncommensurability}

While the various different Houghton groups share similar properties, they are not commensurable with one another.

\begin{namedthm*}{\Cref{thm:MainThm_NonCommensurable}}
    For any $r_1, \ldots, r_5\geq 1$,
    the groups $H_{r_1}, S_{r_2}, \brH_{r_3}$, $B_{r_4}$, and $\BB_{r_5}$ are pairwise not commensurable to each other.
\end{namedthm*}

\begin{proof}
    As each group is of type $F_{r_i-1}$ but not of type $FP_{r_i}$ for $i=1, \ldots, 5$ 
    and finiteness properties are a commensurability invariant \cite{alonso1994finiteness}, it suffices to assume $r=r_1=\cdots=r_5$.

     Given a free abelian subgroup $\mathbb{Z}^k < H_r$, we must have that $k \le \lfloor\frac{r}{2}\rfloor$. The bound arises because shift maps between disjoint pairs of ends commute, and there are at most $\lfloor\frac{r}{2} \rfloor$ such independent pairs. On the other hand, the other four groups contain free abelian subgroups of arbitrary rank. In $S_r$, such groups arise via looking at families of disjoint curves and Dehn twists on them. For $\brH_r$, one can consider families of disjoint pairs of punctures and half twists which permute the points in each pair. In $B_r$ and $\BB_r$ one can use Nielsen generators with disjoint supports, for example. As the maximal rank of abelian subgroups of a group is a commensurability invariant, it follows that $H_r$ is not commensurable to any of the other four groups.

     On the other hand, $B_r$ and $\BB_r$ are not commensurable to $S_r$ and $\brH_r$. To see this, it suffices to first restrict to their pure subgroups $PB_r, P\BB_r, PS_r$, and $P\brH_r$. Then the latter two groups are torsion free. Indeed, for an element $\phi\neq \mathrm{id}$ in either $P\brH_r$ ($PS_r$), either it lies in the kernel of the abelianization map  $P\brH_r \to \Z^{r-1}$ \cite[Theorem 1]{funar2007braided} ($PS_r \to \Z^{r-1}$ \cite[Theorem 1.2]{aramayona2023surface}), or it does not. Note that in the case of $P\brH_2$, this map is not the abelianization, but still exists, and its kernel is the compactly supported mapping class group.
     If $\phi$ does not lie in the kernel, then its image in $\Z^{r-1}$ is nontrivial, and thus $\phi$ cannot be torsion. On the other hand, if $\phi$ lies in the kernel, then as the compactly supported surface and braid groups are both torsion free, $\phi$ still is not torsion. But as $B_r$ and $\BB_r$ both contain infinite torsion subgroups, so do any finite index subgroups. The claim follows.

     To see that $B_r$ and $\BB_r$ are not commensurable, note that by \Cref{lem:psirestriction}, $\BB_r$ contains a normal subgroup isomorphic to a countable direct sum of copies of $\Z/2$. In particular, every finite index subgroup of $\BB_r$ contains a normal $2$-torsion subgroup. On the other hand, suppose $K$ is a finite index normal subgroup of $B_r$, and fix a element $\gamma\in K$ of order $2$. We will show that the normal closure of $\gamma$ in $K$ contains an infinite order element.

    First note that $\gamma \in \Map_c(\G_r)$, as if not it would have nontrivial image in the map appearing in  \Cref{prop:PBfluxSES} and thus would not be torsion.  Also, by the proof of \cite[Theorem D] {domat2025generating}, $\Map_c(\G_r)$ contains a unique finite index subgroup, denoted $\mathrm{SAut}_{\infty}(\G_r)$. This subgroup can be thought of as the kernel of the determinant map arising from the action of $\Map_c(\G_r)$ acting on $H_1(\G_r)$, the first homology of $\G_r$. It follows that $\mathrm{SAut}_{\infty}(\G_r)<K$ as $\mathrm{SAut}_{\infty}(\G_r)$ has no finite index subgroups. 

    There is an $n$ so that $\Aut(F_n)$ embeds as a subgroup of $\Map_c(\G_r)$ and $\tau \in \Aut(F_n)$ (see the proof of \Cref{lem:compactsuppPresentation}, for example). By ~\cite{shen2025}, there is an element $\delta\in \Aut(F_n)$ which has infinite order of the form
    \begin{equation*}
        \delta=\Pi_{i=1}^{\ell} g_i\gamma g_i^{-1}
    \end{equation*}
     for some $g_i\in \Aut(F_n)$, $i=1, \ldots, \ell$. Up to increasing $n$, we can choose a $\kappa \in \mathrm{SAut}_{\infty}(\G_r)\cap \Aut(F_n)$ commuting with $\delta$, $\gamma$, and each $g_i$. In particular, either $g_i\in \mathrm{SAut}(\G_r)$ or $\kappa g_i\in \mathrm{SAut}(\G_r)$. For example, choose $\kappa$ to be a loop flip whose support is disjoint from the supports and images of each of the maps. Then
     \begin{equation*}
         \delta=\kappa \delta \kappa^{-1}= \Pi_{i=1}^{\ell} f_i \gamma f_i^{-1}
     \end{equation*}
     where $f_i=g_i$ or $f_i=\kappa g_i$, with the choice made so that $f_i \in \mathrm{SAut}(\G_r)$, and in particular $f_i\in K$. Thus, we have exhibited the desired infinite order element in the normal closure of $\gamma$ in $K$, showing that $B_r$ and $\BB_r$ are not commensurable.

     Finally, by \cite[Corollary 1.3]{aramayona2023surface}, $S_r$ and $\brH_r$ are not commensurable, finishing the proof.
\end{proof}

Note that arguments analogous to those in \Cref{thm:MainThm_NonCommensurable} show that $\mathcal{H}_r$ is not commensurable to $B_r$ and $\BB_r$, as $\mathcal{H}_r$ is torsion free. Also, $\mathcal{H}_r$ is not commensurable to $H_r$ as the former has free abelian subgroups of infinite rank. However, the above techniques do not suffice to distinguish the commensurability class of $\mathcal{H}_r$ with that of $S_r$ or $\brH_r$.

Although these groups are not commensurable, the parallels between them seem to suggest a possible underlying relationship. A natural next question is:

\begin{Q}
    For any $r\geq 1$,
    are any of the groups $S_{r}, \brH_{r}$, $B_{r}$, $\mathcal{H}_{r}$, or $\BB_{r}$ quasi-isometric to each other?  Are any of the embeddings discussed in \Cref{fig:HoughtonVariants} quasi-isometric embeddings?
\end{Q}

The group $H_r$ is not included in the first part of the question as it is amenable, while the others are not. For the second part of the question, we do have an answer in the case of the graph and doubled handlebody Houghton groups. The short exact sequence corresponding to the map (5) in \Cref{fig:HoughtonVariants} splits by \Cref{lem:psirestriction}. In particular, this implies that the graph Houghton group is a direct factor, and therefore a retract, of the doubled handlebody Houghton group. It follows that the section $B_r \hookrightarrow \mathcal{B}_r$ is a quasi-isometric embedding. 

\section{Finite generation and presentation of $PB_r$}\label{sec:finitePresentation}
In this section and throughout the rest of the paper, we adopt the following simplified notation used for the graph Houghton group of the $(0,1)$-rigid structure on $\Gamma_r$: 
\begin{align*}
    B_r &:= B(0, 1, r), \\
    PB_r &:= PB(0, 1, r).
\end{align*}
Recall $PB(0,1,r)$ is the subgroup of $B(0,1,r)$  fixing the ends of $\G_r$ pointwise.
Note that since $\G_r$ has finitely many ends, $PB_r$ has finite index $r!$ in $B_r$.
We begin this section by giving a finite generating set of $B_r$ when $r \ge 2$, then we explicitly construct a presentation of $PB_r$ for $r \ge 3$.

\subsection{Finite generation of $B_r$, $r \ge 2$}

We follow the ideas presented in \cite[Section 2.3]{lee2012geometry}.
We warm up with the following observation:
\begin{PROP} \label{prop:B1NotFinGen}
    $B_1$ is not finitely generated.
\end{PROP}
\begin{proof}
    Since a proper homotopy equivalence for a rigid mapping class preserves the rank of a defining subgraph, it follows that every element of $B_1$ is compactly supported. Note for $f,g \in B_1$, the support of $fg$ is contained in the union of the supports of $f$ and $g$. Hence, by considering elements of $B_1$ with larger and larger supports, it follows that $B_1$ cannot be generated by finitely many elements.
\end{proof}

\begin{DEF}[Flux maps] \label{def:fluxmap}
Consider the $(r-1)$-fold \emph{flux map} $\Phi$ associated with the following splitting short exact sequence given in \cite[Theorem B]{domat2025generating}:
    \[
        1 \longrightarrow \ov{\Map_c(\G_r)} \longrightarrow \PMap(\G_r) \overset{\Phi}{\longrightarrow} \Z^{r-1} \longrightarrow 1
    .\]
    For complete details, see \cite[Section 7]{domat2023coarse} and \cite[Section 3.2]{domat2025generating}.  Heuristically, these maps measure the rate of loops flowing from end $e_1$ to end $e_j$, for $j = 2, \ldots,r$, under a pure mapping class.  For example, under the flux map  $\Phi: \PMap(\G_2) \to \Z$, the loop shift $h$ shown in \Cref{fig:b2} has $\Phi(h) = 1$.  In fact $\ker \Phi = \ov{\Map_c(\G_r)}$, the closure of the compactly supported mapping class group of $\G_r$.

\end{DEF}

\begin{figure}[ht!]
    \centering
    \includegraphics[width=.7\textwidth]{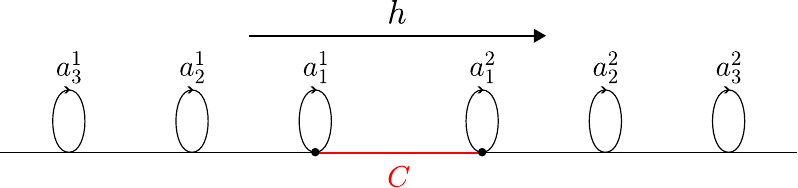}
    \caption{The graph $\G_2$ and its core $C$. The loops are labeled $a^i_{j}$, where $i$  indicates which component of $\G_2 \setminus C$ they are in. The shift $h \in B_2$ translates the loops to the right by one click as shown.}
    \label{fig:b2}
\end{figure}

Of particular relevance to us is that the flux map restricts to the pure graph Houghton group in a natural way.  
Since maps in $\ov{\Map_c(\G_n)}$ are pure, and each map of flux 0 in $PB_n$ is compactly supported, we have $\ov{\Map_c(\G_n)} \cap B_n = \Map_c(\G_n) \cap PB_n=\Map_c(\G_n)$ for $n \ge 1$.  

In the following proposition we need to further analyze the flux maps. To do this, fix a free group $B$. Let $A$ be a free factor of $B$, and write $B=A \ast C$ for another free factor $C$ of $B$. Then we first define the \emph{corank} of $A$ in $B$, denoted $\cork(B, A)$, to be the rank of $C$. Then we have the following.

\begin{PROP}\label{prop:PBfluxSES}
The flux map $\Phi$ restricts to $PB(g,h,r) \subset \PMap(\Gamma_r)$ with the following short exact sequence
\[
    1 \longrightarrow \Map_c(\G_r) \longrightarrow PB(g,h,r) \overset{\Phi\mid_{\textrm PB}}{\longrightarrow} (h\Z)^{r-1} \longrightarrow 1.
    \] 
    
\end{PROP}

\begin{proof}
    Recall that the flux map is defined as follows: Identify the ends of $\G_r$ as $e_1,\ldots,e_r$. Then denote by $P_2,\ldots,P_{r}$ the 2-partitions of $E=\{e_1,\ldots,e_{r}\}$, given by $P_i:= (E \setminus \{e_{i}\}) \sqcup \{e_{i}\}$ for $i \in \{2\ldots,r\}$. For each $i \in \{2,\ldots, r\}$, we can exhaust $\G_r$ by a bi-infinite, increasing sequence of subgraphs $\{\G^{(i)}_k\}_{k=-\infty}^\infty$, each of which has end space $E \setminus \{e_{i}\}$, inducing the partition $P_i$ of $E$. This bi-infinite sequence of subgraphs induces a flux map $\Phi_i : \PMap(\G_r) \to \Z$, defined as for $[f] \in \PMap(\G_r)$,
    \begin{align}\label{eq:cork}
        \Phi_i([f]) := \cork(\pi_1(\G_m^{(i)}),\pi_1(\G_n^{(i)})) - \cork(\pi_1(\G_m^{(i)}), f_*\pi_1(\G_n^{(i)})), \tag{$\star$}
    \end{align}
    where $m>n$ is such that $f_*\pi_1(\G_n^{(i)})$ is a free factor of $\pi_1(\G_m^{(i)})$. Despite the definition, as noted in \cite[Lemma 7.10]{domat2023coarse}, $\Phi_i$ does not depend on the choice of $(m,n)$ or the sequence $\{\G_k^{(i)}\}$, as long as the sequence gives the same partition $E_i$. Recall for $j=2,\ldots,r$, the loop shift maps $h_{j}$ defined in \Cref{ex:shiftMaps}. Then we have $\Phi_i(h_j) = \delta_{ij}$, i.e., 1 if $i=j$ and $0$ otherwise.
    To define $\Phi,$ simply set $\Phi:\PMap(\G_r) \to \Z^{r-1}$ as
    \[
        \Phi([f]):= (\Phi_2([f]), \Phi_3([f]), \ldots, \Phi_{r}([f])).
    \]
    Note that $\Phi$ is surjective because of the subgroup $\langle h_2,\ldots,h_{r}\rangle$.
    
    Now, by restricting $\Phi$ to $PB(g,h,r) \subset \PMap(\G_r)$, we obtain
    \[
    1 \longrightarrow \Map_c(\G_r) \longrightarrow PB(g,h,r) \longrightarrow (h\Z)^{r-1} \longrightarrow 1.
    \]
    Indeed, for $i=2,\ldots,{r}$, a pure rigid proper homotopy equivalence of zero $\Phi_i$-flux fixes pieces near $e_{1}$ and $e_{i}$. This implies the kernel of the restricted $(r-1)$-fold flux map $\Phi|_{PB(g,h,r)}$ consists of compactly supported maps. On the other hand, for a pure $(g,h)$-rigid proper homotopy equivalence $f$ with its defining graph $Z$, and $i=2,\ldots,r$, there exist \emph{suited} $\G_m^{(i)}$ and $\G_n^{(i)}$ such that $\G_m^{(i)} \supset \G_n^{(i)} \cup f(\G_n^{(i)})$ and $\G_n^{(i)} \supset Z$. Then $f(\G_n^{(i)})$ is also $(g,h)$-suited, implying both of the corank terms in \eqref{eq:cork} are multiples of $h$. This implies the image of $\Phi|_{PB(g,h,r)}$ lies in $(h\Z)^{r-1}$. As the $h$-th powers of loop shifts $h_2^{h},\ldots,h_r^{h}$ are $(g,h)$-rigid, it follows that $\Phi|_{PB(g,h,r)}$ surjects onto $(h\Z)^{r-1}$. 
\end{proof}

Now consider $B_2$. Recall we denote by $C$ the core of $\G_2$, which is just an edge in this case.
Label the loops of $\G_2$ by $\{a^1_{i}\}_{i=1}^\infty \cup \{a^2_{j}\}_{j=1}^\infty$ so that each component of $\G_2 \setminus C$ contains the entirety of either $\{a^1_{i}\}_{i=1}^\infty$ or $\{a^2_{j}\}_{j=1}^\infty$. 
Let $h\in B_2$ be the loop shift as in \Cref{fig:b2}, $\tau\in B_2$ be the map that flips (i.e.\ changes the orientation of) the loop $a^1_{1}$, and $\sigma \in B_2$ be the loop swap between $a^1_{1}$ and $a^1_{2}$. Then let $\eta$ be the proper homotopy equivalence that induces an automorphism on $\pi_1(\G_2)$ such that $a^1_1 \mapsto \ov{a^1_2} a^1_1$, $a^1_2 \mapsto a^1_2$ and identity otherwise.
Finally, let $\rho$ be an element that swaps the two ends of $\G_2$. We claim:

\begin{PROP}\label{prop:B2finitegen}
\begin{align*} 
    PB_2 &= \<h, \tau, \sigma, \eta\>,\\
    B_2 &= \<h, \tau, \sigma, \eta, \rho\>.
\end{align*}
\end{PROP}
\begin{proof}
    Since $B_2 \setminus PB_2$ is the coset $\rho PB_2$, it suffices to show the statement for $PB_2$. Let $g \in PB_2$ and say $\Phi(g) = m \in \Z$. Since $\Phi$ is a homomorphism, $f:=h^{-m}g \in \ov{\Map_c(\G_2)} \cap PB_2 \subset \Map_c(\G_2) \cap PB_2$. Suppose $f$ is supported on a subgraph $K$ of $\G_2$. By enlarging $K$, we may assume it is connected with $|\partial K|=2$. Writing $k := \rk(K)$, then $\Map(K, \partial K) = F_k \rtimes \Aut(F_k)$. We first show that every element in $\Aut(F_k)$ is generated by $h,\tau,\sigma,\eta$. The Armstrong--Forrest--Vogtmann \cite{armstrong2008a-presentation} generating set says that the $k$ flips on each generator of $F_k$, the $k-1$ transpositions of (adjacent) generators of $F_k$, and the map $\eta$ suffice. In fact, every flip in $\Aut(F_k)$ can be represented by a proper homotopy equivalence of the form $h^{-t}\tau h^t$ for some $t \in \Z$. Similarly, every (adjacent) transposition in $\Aut(F_k)$ is of the form $h^{-t}\sigma h^t$ for some $t \in \Z$. It follows that $\Aut(F_k) \le \<h,\tau,\sigma,\eta\>$.

    The $F_k$ part, when restricted to $K$, consists of (proper homotopy
    classes of) maps supported
    on edges incident to $\partial K$.
    Enlarge $K$ to $K'$, to include those edges
      incident to $\partial K$.
    Note that these maps have support and image in
      the core of $K'$, so they can be seen as
      elements of $\Aut(F_{k+2})$.  The argument in the previous paragraph does not
      depend on $k$, so it follows that every element in $\Aut(F_{k+2})$ is
      generated by $h, \tau, \sigma, \eta$. In particular, we have
      \[
      F_k
      \rtimes \Aut(F_k) \le \Aut(F_{k+2}) \le \<h, \tau, \sigma,
      \eta\>,\]
      concluding the proof.
\end{proof}

Before considering the finite generation of $B_r$ for $r \ge 3$, we remark an interesting relation between two different loop shifts, also witnessed from the usual Houghton group. To describe it, the following theorem is useful to recover the proper homotopy equivalence from the automorphism on $\pi_1(\G_r)$.

\begin{THM}[{\cite[Theorem 3.1]{algom-kfir2025groups}}]\label{thm:AB_Alexander}
    Suppose $X$ is a graph with all ends accumulated by loops and let $f: X \to X$ be a proper map so that $f_*=[\id] \in \Out(\pi_1(X))$. Then $f$ is properly homotopic to the identity on $X$. 
\end{THM}

\begin{EX}[Commutator is a loop swap]\label{ex:loopshift_commutator}
Following \Cref{fig:loopshift_commutator},  take the two loop shifts $h_2$ and $h_3$ on $\G_3$ and label the loops $\alpha$ and $\beta$. Consider the commutator $[h_2,h_3]=\bar{h}_2\bar{h}_3h_2h_3$.  It induces on $\pi_1(\G_3)$ the transposition automorphism of two basis elements corresponding to $\alpha$ and $\beta$. Applying \Cref{thm:AB_Alexander}, we obtain
\[
    [h_2,h_3] \simeq \sigma,
\]
where $\sigma$ is the loop swap map between $\alpha=a^{1}_2$ and $\beta=a^{1}_1$.

\begin{figure}[ht!]
    \centering
    \begin{overpic}[width=.95\textwidth]{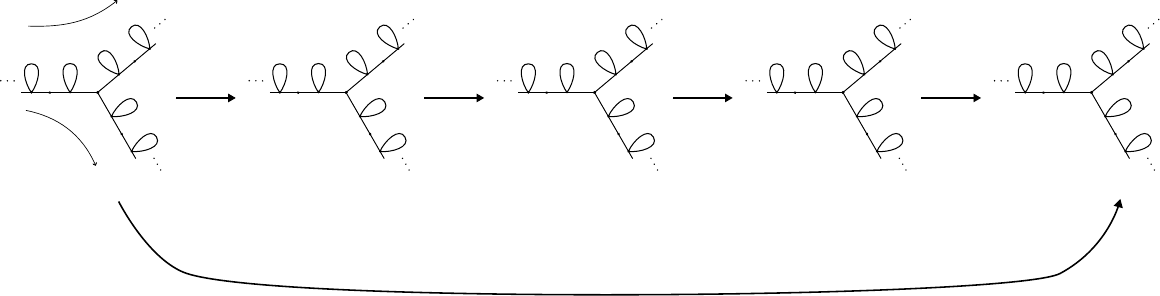}
    \put(-2.5,17.5){\scriptsize{$e_1$}}
    \put(14.5,24){\scriptsize{$e_2$}}
    \put(14,9){\scriptsize{$e_3$}}
    
    \put(5,24.5){\tiny{$h_2$}}
    \put(3.5,13){\tiny{$h_3$}}

    \put(16,18){\small{$h_3$}}
    \put(37.5,18){\small{$h_2$}}
    \put(59,18.4){\small{$\bar{h}_3$}}
    \put(80,18.4){\small{$\bar{h}_2$}}
    \put(50,1.5){$\sigma$}
    
    \put(2,20.5){\scriptsize{$\alpha$}}
    \put(5,20.5){\scriptsize{$\beta$}}
    
    \put(27,20.5){\scriptsize{$\alpha$}}
    \put(33.5,16){\scriptsize{$\beta$}}

    \put(51,21.5){\scriptsize{$\alpha$}}
    \put(55,16)
    {\scriptsize{$\beta$}}

    \put(72.5,21.5){\scriptsize{$\alpha$}}
    \put(69.5,20.5){\scriptsize{$\beta$}}

     \put(87.5,20.5){\scriptsize{$\beta$}}
    \put(91,20.5){\scriptsize{$\alpha$}}

    \end{overpic}
    \caption{Label the ends of $\Gamma_3$ as $e_1$, $e_2$, and $e_3$.  For $i=2,3$, let $h_i$ be the loop shift out of $e_1$ and into $e_i$.  The figure above illustrates how the commutator $[h_2, h_3]$ induces the loop swap $\sigma$.}
    \label{fig:loopshift_commutator}
\end{figure}
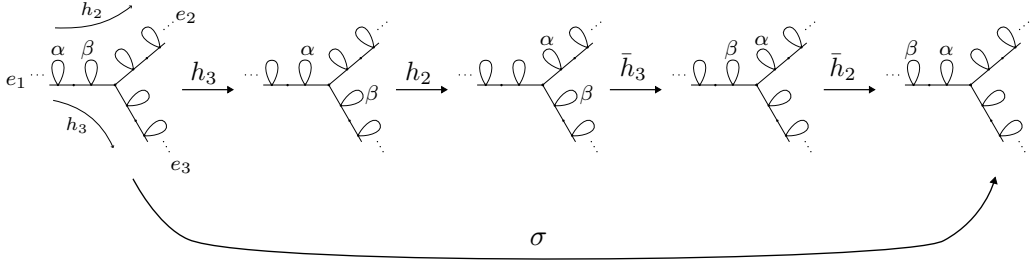
\end{EX}

Now we return to establishing the finite generation of $B_r$ for $r \ge 3$. This is similar to \Cref{prop:B2finitegen}, but we can omit $\sigma$ from the generating set thanks to \Cref{ex:loopshift_commutator}.
Let $e_1,\ldots,e_r$ be the ends of $\G_r$.
For each $i =2,\ldots,r$, let $h_i$ be the shift translating the loops from $e_{1}$ to $e_i$ as in \Cref{fig:br}. (This convention is consistent with $h \in B_2$ (see \Cref{fig:b2}).) The loop shifts $\{h_2,\ldots,h_{r}\}$ form a dual basis for the $r-1$ flux maps generating the $\Z^{r-1}$ quotient of $\PMap(\G_r)$. Let $C$ be the star graph, a tree graph isomorphic to the bipartite complete graph $K_{r,1}$. 
Denote the edges of $C$ by $C_1,\ldots,C_r$, which are incident to loops labeled by $a^1_{1},\ldots,a^r_{1}$ respectively. Label the rest of the loops $a^i_{j}$ for $1 \le i \le r$ and $j\ge 2$ accordingly.
Finally, for $i \neq j$, let $\rho_{ij}$ be a proper homotopy equivalence that swaps $e_i$ and $e_j$.

\begin{figure}[ht!]
    \centering
    \includegraphics[width=0.7\linewidth]{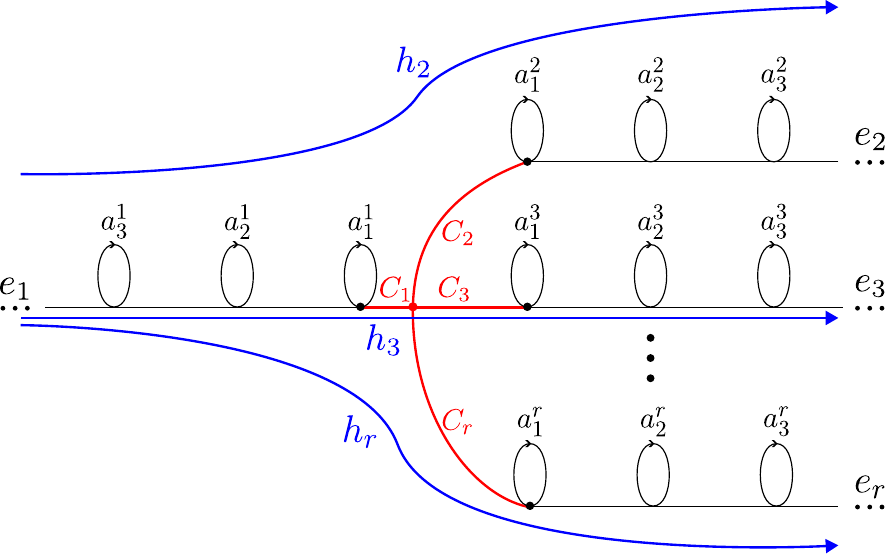}
    \caption{Let $e_1, \ldots, e_{r}$ be the ends of $\G_r$ and for $i=2,\ldots, r$, let $h_i$ denote the loop shift from $e_1$ to $e_i$. The star graph $C$ is in red, and the edge labeled by $C_i$ is the one contained in the complementary component of the middle point of the star graph with the end $e_i$. }
    \label{fig:br}
\end{figure}

\begin{PROP}\label{thm:Bnfinitegen}
For $r\geq 3$,
    \begin{align*}
    PB_r &= \<h_2,\ldots,h_{r}, \tau, \eta\>,\\
    B_r &= \<h_2,\ldots,h_{r}, \tau, \eta, \rho_{12},\ldots,\rho_{r-1,r}\>.
\end{align*}
\end{PROP}

\begin{proof}
    Note that $\rho_{12},\ldots,\rho_{r-1,r}$ generate $\Homeo(E(\G_r)) \cong S_r$.
    Hence, the reduction to generating $\Map(K)$ for a compact connected
    subgraph $K$ with $|\partial K|=r$ of $\G_r$ with the generating set
    $\{h_2,\ldots,h_{r}, \tau, \sigma, \eta\}$ is the same as in
    \Cref{prop:B2finitegen}.
    Note that $\Map(K) \cong F_k^{r-1} \rtimes \Aut(F_k)$, and $F_k^{r-1}$ consists of maps
    supported on the edges incident to $\partial K$. Again, following the proof of
    \Cref{prop:B2finitegen}, by conjugating $\tau,\sigma$ by shifts, we can
    obtain all the transpositions and loop flips in $K$. Similarly, we can embed $F_{k}^{r-1}$ into $\Aut(F_{k+r})$ by
    considering $\Map(K')$ where $K' \supset K$ contains the edges incident to $\partial K$.
    Hence, so far we have $F_k^{r-1} \rtimes \Aut(F_k) \le \<h_2,\ldots,h_{r}, \sigma,
    \tau, \eta\>$.

  Finally, by \Cref{ex:loopshift_commutator}, the transposition $\sigma$ can be
  expressed as a commutator of two loop shifts. This allows us to remove
  $\sigma$ from the generating set, showing $F_k^{r-1} \rtimes \Aut(F_k) \le \<
  h_2,\ldots,h_{r}, \tau, \eta\>$. Since the generating set does not
  depend on $k$, this shows that $\Map_c(\G_r) \le \<h_2,\ldots,h_{r},
  \tau, \eta\>$, so we obtain the desired generating sets for $PB_r$ and $B_r$.
\end{proof}

\subsection{Finite presentation of $PB_r, r \ge 3$}\label{ssec:outline}
The remainder of this section is devoted to constructing an explicit presentation for $PB_r$, $r \ge 3$.  
Recall by \Cref{prop:PBfluxSES} we have
\begin{align}
\label{eqn:SES_pure}
        1 \longrightarrow \Map_c(\G_r) \longrightarrow  &PB_r \longrightarrow \Z^{r-1} \longrightarrow 1,
\end{align}
Building on the finite generating set for $ PB_r $ established in \Cref{prop:B2finitegen,thm:Bnfinitegen}, we now derive a presentation by considering an extended (infinite) generating set for $ \Map_c(\G_r) $ modeled on the generating sets introduced by Armstrong--Forrest--Vogtmann~\cite{armstrong2008a-presentation}. These consist of the element $\eta $, \emph{all} loop flips, and \emph{all} loop swaps of \emph{adjacent} loops. For brevity, we refer to this generating set of \( \Map_c(\G_r) \) as the \emph{AFV-generators}.

The following is a summary of the relations we need to get the presentation of $PB_r$ in the subsequent subsections:
\begin{enumerate}[label=$\<$\Roman*$\>$]
     \item Relations of AFV-generators in $\Map_c(\G_r)$ --- \Cref{sssec:AFV}.
     \item Relations of (lifts of) $\Z^{r-1}$. --- The only relations here are the commutators, see \Cref{ex:loopshift_commutator}.
     \item Relations of conjugations of AFV-generators by shifts --- \Cref{sssec:conjugations_by_shifts}.
 \end{enumerate}
 Combining these we get a presentation for $PB_r$ in \Cref{sssec:presentation}, and in \Cref{thm:PBrPresentation} we reduce it to a finite presentation.
 
\subsubsection{A presentation of $\Map_c(\Gamma_r)$}\label{sssec:AFV}
Suppose $n \ge 4$. In \Cref{lem:compactsuppPresentation}, we will produce a presentation of $\Map_c(\Gamma_r)$ as a direct limit of subgroups $\{\Aut(F_{rn})\}_{n=4}^{\infty}$. Here, $\Aut(F_{rn})$ denotes the subgroup of automorphisms of $\pi_1(\G_r)$ which only involve the first $n$ loops in each ray.

To find a generating set for $\Aut(F_{rn})$, we use a variation of the Armstrong--Forrest--Vogtmann generating set \cite{armstrong2008a-presentation}. Recall in $\G_r$, $a_j^i$ is the $j$th loop in the $i$th ray, where indices run $1\leq i \leq r$, $1\leq j \leq n$. In the following definitions of maps, the loops not shown are sent to themselves.
\begin{enumerate}[label=(\roman*)]

    \item $\tau_j^i: a_j^i \mapsto \overline{a_j^i}$
    \qquad (i.e., flips $a_j^i$ to $\ov{a_j^i}$.)
    
    \item $\sigma_{j\ell}^{ik} : a_j^i \mapstoandfrom a_\ell^k$
    \qquad (i.e., swaps $a_j^i$ and $a_{\ell}^k$)  which can be expressed as a product of the \emph{transpositions} of adjacent loops (see \Cref{fig:Transpositions}): 
 \begin{itemize}
       \item $s_j^i := \sigma_{j,j+1}^{i,i}$, \qquad for $1 \le i \le r$, $1 \le j\le n$,  
       \item $s_0^i := \sigma_{11}^{1i}$, \qquad for $2 \le i \le r$. 
   \end{itemize} 

    \item $\eta:\begin{cases} 
      a_1^1\mapsto \overline{a_2^1} \; a_1^1 \\
      a_2^1\mapsto \overline{a_2^1}
   \end{cases}
   $
\end{enumerate}

 \begin{figure}[ht!]
    \centering
    \begin{overpic}[width=0.6\linewidth]{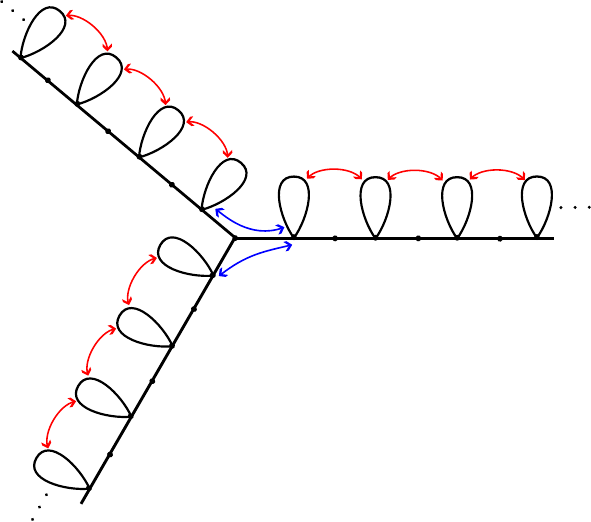}
        \put(0,2){$e_3$}
        \put(103,52){$e_1$}
        \put(-5,90){$e_2$}

        \put(41,51){\color[HTML]{0000FF}{$s_0^2$}}
        \put(42,41){\color[HTML]{0000FF}{$s_0^3$}}

        \put(8,88){$a_4^2$}
        \put(20,79.5){$a_3^2$}
        \put(30,71){$a_2^2$}
        \put(41,61){$a_1^2$}

        \put(11,80){\color[HTML]{FF0000}{$s_3^2$}}
        \put(21,71){\color[HTML]{FF0000}{$s_2^2$}}
        \put(32,62){\color[HTML]{FF0000}{$s_1^2$}}

        \put(47.5,60){$a_1^1$}
        \put(61,60){$a_2^1$}
        \put(74.5,60){$a_3^1$}
        \put(88,60){$a_4^1$}

        \put(54.5,55){\color[HTML]{FF0000}{$s_1^1$}}
        \put(68,55){\color[HTML]{FF0000}{$s_2^1$}}
        \put(82,55){\color[HTML]{FF0000}{$s_3^1$}}

        \put(23,48){$a_1^3$}
        \put(16,36){$a_2^3$}
        \put(8.5,23){$a_3^3$}
        \put(1,10){$a_4^3$}

        \put(23,39){\color[HTML]{FF0000}{$s_1^3$}}
        \put(16,27){\color[HTML]{FF0000}{$s_2^3$}}
        \put(9,14.5){\color[HTML]{FF0000}{$s_3^3$}} 
        
    \end{overpic}
    \caption{The $(nr + r - 1)$-transpositions $s^i_j$'s which, with $\tau^1_1$, generate the subgroup of $\Aut(F_{rn})$ that inverts and permutes generators. These transpositions simplify the expressions used to define the group relations.}
    \label{fig:Transpositions}
\end{figure}
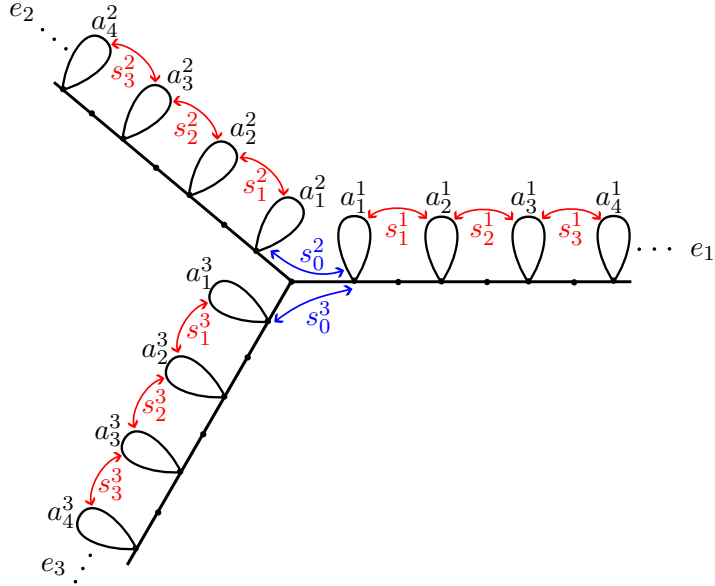

Denote by $I^r_{n}=\{(i,j)\mid 1 \le i \le r,\ 0 \le j \le n-1\} \setminus \{(1,0)\}$. Namely, if $(i,j) \in I^r_{n}$, then the notation $s^i_j$ makes sense as an element of $\Aut(F_{rn})$. 

\begin{THM}
[{\cite[Theorem 1]{armstrong2008a-presentation}}] \label{thm:AFVpresentation} Let $n\ge 4$. The set 
\[
A_{r,n}:= \{\tau, \eta\} \cup \{s^i_j\}_{(i,j) \in I^r_{n}}
\]
together with the following relations present $\Aut(F_{rn})$: for each $(i,j) \in I:= I^r_n$,
    \begin{enumerate}
        \item those generating the subgroup of $\Aut(F_{rn})$ that inverts and permutes the generators, namely $\tau=\tau_1^1$ and the $(nr + r - 1)$-transpositions $s^i_j$'s: 
        \begin{enumerate}
        \item $(s^i_j)^2=1$

        \item $\begin{cases}[s_j^i, s_\ell^k] = 1, & \text{if $(k,\ell) \in I$, and $|j - \ell| \ge 2$},\\
        [s^i_j, s^k_{\ell}] = 1 & \text{if $(k,\ell) \in I$, $i \neq k$, $|j - \ell| \le 1$, and $j,\ell \ge 1$.}
        
        \end{cases}$

        \item $\begin{cases}(s_j^i s_\ell^i)^3 = 1 &\text{if $(i,\ell) \in I$, and $|j - \ell| = 1$,} \\
        (s_0^i s_1^1)^3 = 1 & \text{if $(i,0) \in I$,}\\
        (s_0^i s_0^j)^3 = 1 & \text{if $(i,0),(j,0) \in I$ and $i\neq j$.}
        \end{cases}$  

        \item $\tau^2 = 1$.

        \item $[\tau,s_j^i]^2 = 1$ if $(i,j) \in \{(1,1), (i,0) \mid 2 \le i \le r\}$.

        \item $[\tau, s_j^i] = 1$ if $(i,j) \not\in \{(1,1), (i,0) \mid 2 \le i \le r\}$. 
        
    \end{enumerate}
        \item those involving $\eta$: 
        \begin{enumerate} \setlength{\itemsep}{3pt}
    \item $\eta^2=1$
    \item $(\eta s^1_1)^3=1$, 
    \item
    $\begin{cases}
    [\eta, \tau^{s^1_{j}\cdots s^1_{1}}]=1 & \text{if $j \ge 2$} \\
    [\eta, \tau^{s^i_{j}\cdots s^i_{0}}]=1 & \text{if $i \ge 2$}
    \end{cases}$ 
    \item $\begin{cases}
        [\eta, s^i_j]=1 & \text{if $j \neq 0$ and $(i,j)\neq (1,1), (1,2)$}\\
        [\eta, (s_2^1)^{s_0^i s_1^1}] =1 & \text{for $2 \le i \le r$}
        \end{cases}
        $
   \item $((\eta\tau)^2\tau^{s^1_1})^2=1$
   \item $(\eta (s_1^1)^{s_2^1}\tau^{s^1_1}\eta s^1_1)^4=1$
   \item $s^1_1\eta (s^1_1)^{s^1_2}\tau^{s^1_1}\eta s^1_1(s_2^1  \eta (s_1^1)^{s_2^1} \tau^{s_1^1} \eta)^2 = 1$
   \item $((s_1^1)^{{s^1_2}^{s^1_3}}s^1_2\eta)^4=1$
\end{enumerate}
    \end{enumerate}
\end{THM} 

\begin{proof}
    This presentation is very similar to that appearing in \cite[Theorem 1]{armstrong2008a-presentation}. Indeed, the only difference is relabeling, the replacement of $\tau^i_j$'s with the conjugations of $\tau$ throughout (2c),(2e),(2f), and (2g), as well as our relation (2d) compared to their relation (4). Namely, they explicitly include relations saying that \emph{all} loop swaps which should commute with $\eta$ do indeed commute, while we give the same relations but only for loop swaps of adjacent loops, along with a small collection of extra such relations. In fact, any loop swap whose support is disjoint from $a^1_1 \cup a^1_2$ (the support of $\eta$) can be conjugated by those $s^i_j$ commuting with $\eta$ to one of the following: $s^1_3, s^2_1,$ or $(s^1_2)^{s_0^i s_1^1}=\sigma_{13}^{i1}$ for $i\geq 2$, where the first two commute with $\eta$ by the first relations in (2d), and the third type commute with $\eta$ by the second type of relations in (2d).
\end{proof}

   \begin{RMK}\label{rmk:Relator2f}
        A small improvement upon this presentation can be made: the relator labeled (2f) is a consequence of (2g) and the fact that $\tau^{s_1^1}$ and $(s_1^1)^{(s_1^2)}$ commute. More precisely, observe that (2g) can be rewritten as $$\eta (s_1^1)^{s_2^1}\tau^{s_1^1}\eta s_1^1 \eta (s_1^1)^{s_2^1}\tau^{s_1^1}\eta s_1^1 = s_2^1\eta\tau^{s_1^1}(s_1^1)^{s_2^1}\eta s_2^1.$$ We can then simplify (2f) as follows: \begin{align*}
            (\eta(s_1^1)^{s_2^1}\tau^{s_1^1}\eta s_1^1)^4 &= (\eta(s_1^1)^{s_2^1}\tau^{s_1^1}\eta s_1^1\eta(s_1^1)^{s_2^1}\tau^{s_1^1}\eta s_1^1)^2\\
            &= (s_2^1\eta(\tau_1^1)^{s_1^1}(s_1^1)^{s_2^1}\eta s_2^1)^2
        \end{align*}
        After writing the final expression out, one can see that the only obstruction to all these involutions canceling each other out is swapping a pair of $\tau^{s_1^1}$ and $(s_1^1)^{s_2^1}$, which is a consequence of the relations (1f) and (1c).
   \end{RMK}
\begin{RMK}
We remark that our $r$ and $n$ in \Cref{thm:AFVpresentation} are flipped from the notation in \cite{lee2012geometry}, where the author uses $n$ for the number of rays, and $r$ for the number of points (which correspond to loops in $\G_r$) from the center.    
\end{RMK}

Denote by $Q_{r,n}$ the relations in \Cref{thm:AFVpresentation}. Define $A_r = \cup_{n=4}^{\infty} A_{r,n}$ and $Q_r=\cup_{n=4}^{\infty} Q_{r,n}$.

\begin{LEM}\label{lem:compactsuppPresentation}
    $\Map_c(\G_r)\cong \<A_r \mid Q_r\>$.
\end{LEM}
\begin{proof} 
    Endow $\G_r$ with the path metric with all edge lengths $1$.
    Let $\Delta_{r,n}$ be the finite subgraph of $\G_r$ lying in the closed $(n+\frac12)$-ball centered at the middle vertex of the core of $\G_r$. Then $\G_r$ is exhausted by $\{\Delta_{r,n}\}_{n=4}^\infty$, so it follows that \[
    \Map_c(\G_r) \cong \varinjlim_{n} \Map_c(\Delta_{r,n})=\varinjlim_{n} \Map(\Delta_{r,n}).
    \]
     
    However, we have the following diagram: 
    \[
    \begin{tikzcd}
        \Aut(F_{rn}) \rar[hook] \dar{\cong}
        &\Map_c(\G_r) \\
        \<A_{r,n} \mid Q_{r,n}\> \rar[hook] & \<A_r \mid Q_r\>,
    \end{tikzcd}
    \]
    where two hooked arrows are inclusions, and the left vertical arrow is an isomorphism by \Cref{thm:AFVpresentation}.

    On the other hand, note that $\Aut(F_{rn}) \le \Map(\Delta_{r,n})$ by identifying $\Aut(F_{rn})$ as the mapping class group of the union $(\Delta_{r,n})_c^*$ of an edge and the core graph of $\Delta_{r,n}$ (thinking of the valence $1$ vertex as a marked point). With this identification, we can conversely see that $\Aut(F_{rn}) \ge \Map(\Delta_{r,n-1})$ as $\Delta_{r,n-1} \subset (\Delta_{r,n})_c^*$. Therefore, the two nested sequences of subgroups $\{\Map(\Delta_{r,n})\}_{n \ge 4}$ and $\{\Aut(F_{rn})\}_{n \ge 4}$ yield the same direct limits.
    
    Taking the direct limits on the left column of the commutative diagram, the inclusions become isomorphisms for both cases. This induces an isomorphism $\Map_c(\G_r) \to \<A_r \mid Q_r\>$, concluding the proof.
\end{proof}

\subsubsection{Conjugations by shifts}\label{sssec:conjugations_by_shifts}
According to the beginning of \Cref{ssec:outline}, to complete the presentation of $PB_r$ it only remains to find the conjugation relations of generators of $\Map_c(\G_r)$ with shifts.

Recall we use the conjugation notation $g^h := hg\ov{h}$, where the functions in the word are applied from the right to left. Since $h_i(a^1_1) = a^i_1$ and $h_i(a^1_2) = a^1_1$, $h_i$ coincides with the 3-cycle $s_1^1 s_0^i = (a^i_1\ a^1_1\ a^1_2)$ on $a^1_1 \cup a^1_2$. Similarly, $\ov{h}_i$ coincides with $s_2^1s_1^1$ on $a^1_1 \cup a^1_2$. Hence we have the following relations for conjugations of $\eta,\tau_1^1$, and $s_1^1$ by shifts. For $2 \le i \le r$,

\begin{align}
    \label{eqn:Qr'}
    \tag{$Q_r'$}
    \begin{cases}
      \eta^{\ov{h}_i} = \eta^{s^1_1 s^1_2} \\ 
      \eta^{h_i} = \eta^{s_1^1 s_0^i}
    \end{cases}
    \quad
    \begin{cases}
    (\tau_1^1)^{\ov{h}_i} =(\tau_1^1)^{s_1^1}
    \\
    (\tau_1^1)^{h_i} = (\tau_1^1)^{s_0^i}
    \end{cases}
    \quad
    \begin{cases}
    (s_1^1)^{\ov{h}_i^n} = s^1_{n+1} & \text{for $n \ge 0$}
    \\
    (s_1^1)^{h_i^n} = s^i_{n-1} & \text{for $n \ge 1$.}
    \end{cases}
\end{align}

Denote by $Q_r'$ the above set of conjugation relations of AFV-generators by shifts.
\subsubsection{Presentation} \label{sssec:presentation}
 We define a group $G_r$ by the following presentation

\[
G_r=\<h_2, \ldots, h_{r}, A_r \mid Q_r,\ Q_r',\ \{s_1^1[h_i, h_j]\}_{2 \leq i < j \leq r}\>
\]
\begin{LEM}\label{lem:HrSES}
    We have the following short exact sequence
    $$1 \to \< A_r\> \to G_r \to \Z^{r-1} \to 1$$
\end{LEM}
\begin{proof}
    Note that the $Q_r'$ relations imply that $\<A_r\>$ is normal, as they allow one to write the conjugates of elements of $A_r$ by powers of $h_i$ as elements in $A_r$. 
    As the relations in $Q_r$ involve only elements in $A_r$, these relations are trivial in $G_r/A_r$. Similar reasoning shows that the $Q_r'$ relations are also trivial in $G_r/A_r$. The relation $s_1^1[h_i, h_j]$ implies that the images of $h_i$ and $h_j$ in $G_r/A_r$ commute. There are no other relations, so it follows that $G_r/A_r\cong \Z^{r-1}$.
\end{proof}
Of course, one can use the relations $\{s_1^1[h_i, h_j]\}_{2 \leq i < j \leq r}$ to remove the $s_1^1$ generator, but this relation will persist in our final presentation (see \refrel{rel:r4}) as it simplifies the expression of many relations.

\begin{LEM}\label{lem:HrPBr}
    For $r \ge 2, G_r \cong PB_r$.
\end{LEM}

\begin{proof}
    Define a map $\Theta:G_r \to PB_r$ by sending $h_i$'s to $h_i$'s, $\tau$ to $\tau$, and $\eta$ to $\eta$, and then extending it to be a homomorphism. By Equation~(\ref{eqn:SES_pure}) and \Cref{lem:HrSES}, we have two rows of short exact sequences
    \[
    \begin{tikzcd}
        1 \rar  & \<A_r\> \dar \rar & G_r \dar{\Theta} \rar & \Z^{r-1} \dar[equal] \rar & 1 \\
        1 \rar &  \Map_c(\G_r) \rar & PB_r \rar & \Z^{r-1} \rar & 1,
    \end{tikzcd}
    \]
    where the map $\langle A_r \rangle \to \Map_{c}(\G_r)$ is an isomorphism by \Cref{lem:compactsuppPresentation}. Hence, by the Five Lemma, we have an isomorphism $\Theta: G_r \to PB_r$.
\end{proof}

\begin{LEM} \label{lem:relationhelper} Let $G$ be a group and recall the notations
\[
    a^b := ba\ov{b}, \qquad [a,b]:= \ov{a}\ov{b}ab.
\]
Then we have the following for $a,b,c \in G$:
\begin{LEMenum}
    \item $[a,b]^c = [a^c,b^c]$, \label{helper:conjugating_commutator}
    \item $(a^b)^c = a^{cb}$, \label{helper:double_super_script}
    \item $[a,b]=1$ $\Longleftrightarrow$ $a^{b} = a$ $\Longleftrightarrow$ $a^{\ov{b}} = a$,\label{helper:commutator_to_conjugation}
    \item if $[b,c]=1$, then $[a,b]=[ca,b]$ and $[b,a]=[b, ca]$. \label{helper:commuting_prefix}
\end{LEMenum}

\begin{proof}
    All follow easily from the definitions. Part (ii) is worth noting;
    \[
        (a^b)^c = c(a^b)\ov{c} = cba\ov{b}\ov{c} = a^{cb},
    \]
    as the order is different from what might be expected. To see (iv),
    note $[b,c]=1$ implies $\ov{b}\ov{c}=\ov{c}\ov{b}$, so
    \[
        [ca,b] = \ov{a}\,\ov{c}\ov{b}cab = \ov{a}\ov{b}\ov{c}cab = [a,b].
    \]
    The second equality in (iv) follows similarly.
\end{proof}

\end{LEM}

\begin{namedthm*}{\Cref{thm:PBrPresentation}}
    For $r\geq 3$, the pure graph Houghton group $PB_r$ has the presentation
    $$PB_r\cong \<h_2, \ldots, h_{r}, \sigma, \tau, \eta \mid P_r\>$$
    where $P_r$ consists of (in each case, $2\leq i<j\leq n$)

\renewcommand{\arraystretch}{1.8}
\begin{tabularx}{\textwidth}{@{}X X X@{}}
{$r_1 : \sigma^2 = 1$}\relationlabel{rel:r1} &
{$r_2 : [\sigma, \sigma^{\overline{h}_i^2}] = 1$}\relationlabel{rel:r2} &
{$r_3 : (\sigma\sigma^{\overline{h}_i})^3 = 1$}\relationlabel{rel:r3} \\
{$r_4 : \sigma = [h_i,h_j]$}\relationlabel{rel:r4} &
{$r_5 : \sigma^{\overline{h}_i} = \sigma^{\overline{h}_j}$}\relationlabel{rel:r5} &
{$r_6 : \tau^2 = 1$}\relationlabel{rel:r6} \\
{$r_7 : [\tau,\sigma]^2 = 1$}\relationlabel{rel:r7} &
{$r_8 : [\tau,\sigma^{\overline{h}_i}] = 1$}\relationlabel{rel:r8} &
{$r_9 : \tau^{\overline{h}_i} = \tau^{\sigma}$}\relationlabel{rel:r9} \\
{$r_{10} : \eta^2 = 1$}\relationlabel{rel:r10} &
{$r_{11} : (\sigma\eta)^3 = 1$}\relationlabel{rel:r11} &
{$r_{12} : [\eta,\sigma^{h_i}]^2 = 1$}\relationlabel{rel:r12} \\
{$r_{13} : [\eta,\sigma^{h_i^2}] = 1$}\relationlabel{rel:r13} &
{$r_{14} : \eta^{h_i} = \eta^{\sigma\sigma^{h_i}}$}\relationlabel{rel:r14} &
{$r_{15} : [\eta, \eta^{\overline{h}_i^2}] = 1$}\relationlabel{rel:r15} \\
{$r_{16} : [\eta, \tau^{h_i}] = 1$}\relationlabel{rel:r16} &
{$r_{17} : ((\eta\tau)^2\tau^{\sigma})^2 = 1$}\relationlabel{rel:r17} &
\\
\multicolumn{3}{@{}l@{}}{
  {$r_{18} : \sigma\eta\sigma^{\sigma^{\overline{h}_i}}\tau^{\sigma}\eta\sigma(\sigma^{\overline{h}_i}\eta\sigma^{\sigma^{\overline{h}_i}}\tau^{\sigma}\eta)^2 = 1$}\relationlabel{rel:r18}
} \\
\end{tabularx}
\end{namedthm*}

\begin{proof} By \Cref{lem:HrPBr}, we have the presentation
\[
PB_r=\<h_2, \ldots, h_{r}, \tau,\ \eta,\ \{s^i_j\}_{(i,j) \in I^r_\infty} \mid Q_r,\ Q_r',\ \{\overline{s_1^1}[h_i, h_j]\}_{2 \leq i < j \leq r}\>,
\]
where $I^r_\infty := \bigcup_{n=4}^\infty I^r_n = \{(i,j) \in \Z \times \Z \mid 1 \le i \le r,\ j \ge 1\}.$
    First we replace the $s^i_j$'s with $\sigma$ conjugated by shifts according to \eqref{eqn:Qr'}. To reiterate, for $n \ge 0$, and every $i = 2,\ldots, r$, we have $
    s^{1}_{n+1} = \sigma^{\ov{h}_i^n}$ and $s^i_n = \sigma^{h_i^{n+1}}.$ In particular, $s_1^1 = \sigma$. Hence we recover our generating set $\{h_2,\ldots,h_r,\tau,\eta, \sigma\}$ (as in \Cref{thm:Bnfinitegen}) and it remains to reduce the relations. We replace $s^i_j$'s in the relations in $Q_r$ and $Q_r'$ by conjugates of $\sigma$ by powers of shifts, and denote them by $R_{1\bullet},R_{2\bullet}$ and $R'_{\bullet}$ respectively: For each $i,j \in \{2,\ldots,r\}$ with $i\neq j$, 
    \begin{enumerate}[label=($R_{1\alph*}$):]
        \item $\sigma^2=1$ {}\cdotfill{} $\equiv$\refrel{rel:r1}
        \item $\begin{cases}
        [\sigma^{h_i^m}, \sigma^{h_i^n}]=1 & \text{for } |m -n| \ne 1 \\
        [\sigma^{h_i^m}, \sigma^{h_j^n}]=1 & \text{for $m \ge 2 \text{ or } n\ge 2$}
        \\
        [\sigma^{h_i^m}, \sigma^{h_j^n}]=1 & \text{for  $m\leq 1, n\leq 1$, $|m-n| \ne 1$}
        \end{cases}$ \cdotfill{} $\mathbf{\Leftarrow} \text{\refrel{rel:r1}$-$\refrel{rel:r5}}$
        \item $\begin{cases}
            (\sigma \sigma^{h_i})^3 = 1 & \text{\tiny(The first two of (1c) collapse to this)}\\
            (\sigma^{h_i} \sigma^{h_j})^3 = 1 & \text{for $i \neq j$}
        \end{cases}$ \cdotfill{} $\mathbf{\Leftarrow}\text{\refrel{rel:r1}, \refrel{rel:r3}$-$\refrel{rel:r5}}$ 
        \item $\tau^2 =1$ {}\cdotfill{} $\mathbf{\equiv}$\refrel{rel:r6}
        \item $\begin{cases}
        [\tau, \sigma]^2=1, \\
        [\tau, \sigma^{h_i}]^2=1
            \end{cases}
            $ {}\cdotfill{} $\mathbf{\equiv}$ \refrel{rel:r7}
        \item $[\tau,  \sigma^{h_i^n}]=1$ \qquad \text{for $n \neq 0, 1$} {}\cdotfill{} $\mathbf{\Leftarrow}$ \refrel{rel:r1}, \refrel{rel:r4}$-$\refrel{rel:r6}, \refrel{rel:r9} 
    \end{enumerate}
    \begin{enumerate}[label=($R_{2\alph*}$):]
        \item $\eta^2=1$ {}\cdotfill{} $\mathbf{\equiv}$ \refrel{rel:r10}
        \item $(\eta \sigma)^3=1$ {}\cdotfill{} $\mathbf{\equiv }$ \refrel{rel:r11}
        \item $\begin{cases}
            [\eta, \tau^{\sigma^{\ov{h}_i^n}\cdots \sigma^{\ov{h}_i}\sigma}]=1 & \text{for $n \ge 1$}\\
            [\eta, \tau^{\sigma^{h_i^n}\cdots \sigma^{h_i}}]=1 & \text{for $n \ge 1$}
        \end{cases}$ {}\cdotfill{} $\mathbf{\Leftarrow} \text{\refrel{rel:r1}, \refrel{rel:r3}$-$\refrel{rel:r6}, \refrel{rel:r9}, \refrel{rel:r16}}$
        \item $\begin{cases}
            [\eta, \sigma^{h_i^n}] = 1 & \text{for $|n| \ge 2$}\\
            [\eta,  (\sigma^{\ov{h}_i})^{\sigma^{h_j} \sigma}] = 1 
        \end{cases}$ {}\cdotfill{} $\mathbf{\Leftarrow }\text{\refrel{rel:r3}$-$\refrel{rel:r5}, \refrel{rel:r13}, \refrel{rel:r14}}$
        \item $\left((\eta\tau)^2\tau^\sigma\right)^2 = 1$ {}\cdotfill{} $\mathbf{\equiv }$ \refrel{rel:r17}
        \item $(\eta \sigma^{\sigma^{\ov{h}_i}}\tau^\sigma \eta \sigma)^4=1$ {}\cdotfill{} (see \Cref{rmk:Relator2f})
        \item $\sigma \eta \sigma^{\sigma^{\ov{h}_i}}\tau^\sigma \eta \sigma (\sigma^{\ov{h}_i}\eta \sigma^{\sigma^{\ov{h}_i}} \tau^\sigma \eta)^2=1$ {}\cdotfill{} $\mathbf{\equiv}$ \refrel{rel:r18}
        \item $(\sigma^{{(\sigma^{\ov{h}_i})}^{(\sigma^{\ov{h}_i^2})}}\sigma^{\ov{h}_i}\eta)^4=1$ {}\cdotfill{} $\mathbf{\equiv}$ \refrel{rel:r15}
    \end{enumerate}
    \begin{enumerate}[label=($R'_{\alph*}$):]
    \item $\begin{cases}
        \eta^{\ov{h}_i} = \eta^{\sigma \sigma^{\ov{h}_i}}\\
        \eta^{h_i} = \eta^{\sigma \sigma^{h_i}}
    \end{cases}$ {}\cdotfill{} $\mathbf{\equiv }$ \refrel{rel:r14}
    \item $\begin{cases}
        \tau^{\ov{h}_i} = \tau^{\sigma}\\
        \tau^{h_i} = \tau^{\sigma^{h_i}}
    \end{cases}$ {}\cdotfill{} $\mathbf{\equiv}$ \refrel{rel:r9}
    \item $\sigma^{\ov{h}_i^n} = \sigma^{\ov{h}_j^n}$ \qquad \text{for $n \ge 0$}. \cdotfill{} $\mathbf{\Leftarrow} \text{\refrel{rel:r1}$-$\refrel{rel:r5}}$
    \item[$(R'')$:]$\sigma[h_i,h_j] = 1$ for $i \neq j$. {}\cdotfill{} $\mathbf{\equiv }$ \refrel{rel:r4}
    \end{enumerate}

The two finite sets of relations in $R_{1c}$ stem from \refrel{rel:r1}, \refrel{rel:r3}, \refrel{rel:r4}, and \refrel{rel:r5}. Indeed, the first one is equivalent to \refrel{rel:r3}, for by conjugating by $\ov{h}_i\sigma$,
    \[
        (\sigma \sigma^{h_i})^3 = 1 \quad\Longleftrightarrow\quad \left(\sigma^{\ov{h}_i\sigma} (\sigma^{h_i})^{\ov{h}_i\sigma}\right)^3=1
        \underset{\ref{helper:double_super_script},\text{\refrel{rel:r1}}}{\Longleftrightarrow} (\sigma \sigma^{\ov{h}_i})^3 = 1 :\text{\refrel{rel:r3}}.
    \]
    
    We use this first relation, denoted by $(R_{1c}-1)$, to prove
    the second relation in $R_{1c}$, which is equivalent to $(\sigma^{h_i})^{\sigma^{h_j}} = (\sigma^{h_j})^{\sigma^{h_i}}$. Note that
    \[
    (\sigma^{h_i})^{\sigma^{h_j}} = \sigma^{h_j\sigma \ov{h_j}h_i} \underset{\text{\refrel{rel:r4}}}{=} \sigma^{h_j\sigma h_i \sigma \ov{h}_j} \underset{\text{\refrel{rel:r5}}}{=} \sigma^{h_j\sigma \sigma^{h_i}} \underset{(R_{1c}-1)}{=} \sigma^{h_jh_i}.
    \]
    
    The two relations in $(R_{1e})$ are equivalent since 
    \begin{align*}
        1 = [\tau, \sigma^{h_i}]^2 \underset{\ref{helper:conjugating_commutator}}{=}[\tau^{\ov{h}_i}, \sigma]^2 \underset{\text{\refrel{rel:r9}}}{=} [\tau^{\sigma}, \sigma]^2 \underset{\ref{helper:conjugating_commutator},\text{\refrel{rel:r1}}}{\Longleftrightarrow} [\tau, \sigma]^2 = 1 : \text{\refrel{rel:r7}}.
    \end{align*}
    
    We see that in $(R_a')$, the two sets of relations can be obtained from one another by conjugating by $\sigma \ov{h}_i \sigma$ as follows:
    \[
        \text{\refrel{rel:r14}}: \eta^{h_i} = \eta^{\sigma \sigma^{h_i}} \Longleftrightarrow (\eta^{h_i})^{\sigma \ov{h}_i \sigma} = (\eta^{\sigma \sigma^{h_i}})^{\sigma \ov{h}_i \sigma} \underset{\ref{helper:double_super_script}, \text{\refrel{rel:r1}}}{\Longleftrightarrow} \eta^{\sigma \sigma^{\ov{h}_i}} = \eta^{\ov{h}_i}.
    \]
    The two sets of relations in $(R_b')$ are similarly equivalent by conjugating by $\sigma \ov{h}_i$:
    \[
        \tau^{h_i} = \tau^{\sigma^{h_i}} \quad\Longleftrightarrow\quad (\tau^{h_i})^{\sigma \ov{h}_i} = (\tau^{\sigma^{h_i}})^{\sigma \ov{h}_i} \underset{\ref{helper:double_super_script}, \text{\refrel{rel:r1}}}{\Longleftrightarrow} \tau^{\sigma} = \tau^{\ov{h}_i}: \text{\refrel{rel:r9}}.
    \]
    
    We can reduce $R_{1b}$ and $R'_c$ to finite sets of relations
    \[
        \text{\refrel{rel:r2}}: [\sigma, \sigma^{\ov{h}_i^2}] = 1, \qquad \text{\refrel{rel:r5}}:\sigma^{\ov{h}_i} = \sigma^{\ov{h}_j}
    \]
    {\'a} la Johnson~\cite[Section 4]{johnson1999embedding}. This only uses the relations \refrel{rel:r1}, \refrel{rel:r2}, \refrel{rel:r3}, \refrel{rel:r4}, and \refrel{rel:r5}  with $\sigma$ and shifts; hence does not depend on the other infinite relations $R_{1f},R_{2c}$ and $R_{2d}$, which are left to be shown.
    
    Now we show that $(R_{2h})$ is equivalent to \refrel{rel:r15}. Recall from \Cref{sssec:AFV}, the notation for the loop swap $\sigma_{j\ell}^{ik} \colon a_j^i \mapstoandfrom a_{\ell}^{k}$.  For simplicity, let $\sigma_{j \ell} := \sigma_{j \ell}^{11}$.  We begin by noting that $(R_{2h}) : (\eta^{\sigma_{14} \sigma_{23}} \eta)^2 = 1$.  We claim: 
    \begin{equation*}
        \eta^{\sigma_{14} \sigma_{23}} = \eta^{\ov{h}_i^2 \sigma}  
    \end{equation*}
    One can verify that $\sigma_{14} \sigma_{23} = \sigma_{24} \sigma \sigma_{24} \sigma^{\ov{h}_i^2}$, hence  
    \begin{align*}
        \eta^{\sigma_{14} \sigma_{23}} &\underset{\hphantom{\cref{helper:double_super_script}}}{=} \eta^{\sigma_{24} \sigma \sigma_{23} \sigma^{\ov{h}_i^2}} 
        \underset{(R_{2d})}= \eta^{\sigma_{24} \sigma \sigma_{23}}\\
        &\mathrel{\underset{\phantom{3!}\mbox{}(R_a')\mbox{}\!\phantom{3!}}{=}} \eta^{\ov{h}_i \sigma_{13} h_i \ov{h}_i} && [\text{ since } \sigma_{24} = (\sigma_{13})^{\ov{h}_i}]\\
        &\underset{\hphantom{\cref{helper:double_super_script}}}{=} \eta^{\ov{h}_i \sigma^{\sigma^{\ov{h_i}}}} && [\text{ since } \sigma_{13} = \sigma^{\sigma_{23}} \text{ and } \sigma_{23} = \sigma^{\ov{h}_i}]\\
        &\underset{\cref{helper:double_super_script}}=\left(\eta^{\sigma \sigma^{\ov{h}_i}}\right)^{\ov{h}_i^2 \sigma h_i} \\
        &\underset{\phantom{3!}\mbox{}(R_a')\mbox{}\!\phantom{3!}}{=} (\eta^{\ov{h}_i})^{\ov{h}_i^2 \sigma h_i}
        \underset{\cref{helper:double_super_script}}{=} \eta^{\ov{h}_i^2 \sigma}.
    \end{align*}
    Using this, we see that $(R_{2h})$ is equivalent to \refrel{rel:r15} as:  
    \[
        (\eta^{\ov{h}_i^2\sigma} \eta)^2 \underset{(R_{2d})}{=} \ov{h}_i^2 \sigma \eta \eta^{h_i^2} \eta \sigma h_i^2 \eta \quad
        \underset{\text{cyclic perm.}}{\Longleftrightarrow} \quad \sigma \eta^{h_i^2} \sigma \eta \eta^{h_i^2} \eta 
        \underset{(R_{2d})}{=} [\eta^{h_i^2}, \eta]. 
    \]
   
    Before proceeding, we give a list of auxiliary relations which will be useful to reduce the relations to those in $P_r$.

    \begin{LEM}\label{lem:auxRelations1}     
        The relations in $P_r$ imply the following set of relations, for any $i\neq j$:
        \begin{align}           
            \tag{$\hat{q}_1$} & [\tau^{h_i}, h_j]=1, \label{rel:q1h} \\
            \tag{$\hat{r}_9^k$} & \tau^{\ov{h}_i^{k}} = \tau^{\sigma^{\ov{h}_i^{k-1}}\cdots \sigma^{\ov{h}_i}\sigma} =  \tau^{\ov{h}_j^{k}} \qquad \text{for all $k \ge 1$}, \label{rel:r9kh} \\
            \tag{$\hat{q}_2^k$} & [\sigma^{h_i^k},h_j ]=1\qquad \text{for all $k \ge 2$}.\label{rel:q2kh}
        \end{align}
    \end{LEM}
    \begin{proof}\renewcommand{\qedsymbol}{$\triangle$}
       By \text{\refrel{rel:r6}},
         \[\tau^2=1 \quad\underset{\text{\refrel{rel:r9}}}\Longleftrightarrow \quad 1=\tau\sigma \tau^{\ov{h}_j}\sigma=[\tau,    h_i\sigma]\underset{\text{\refrel{rel:r4}}}=[\tau, h_j^{\ov{h}_i}]\quad\underset{\ref{helper:conjugating_commutator}}\Longleftrightarrow \quad [\tau^{h_i}, h_j]=1\]
        
         For \eqref{rel:r9kh}, we use $\text{\refrel{rel:r9}}$ as a base case. Assume the result is true for a given $k\geq 1$. It follows that
         \[
         \tau^{\sigma^{\ov{h}_i^{k-1}}\cdots \sigma^{\ov{h}_i}\sigma} =\tau^{\ov{h}_j^{k}}\quad\underset{\ref{helper:double_super_script}}\Longleftrightarrow \quad\tau^{\sigma^{\ov{h}_i^{k}}\cdots \sigma^{\ov{h}_i}\sigma} = \tau^{\sigma^{\ov{h}_i^k}\ov{h}_j^{k}}\underset{\text{IH}}{=}\tau^{\sigma^{\ov{h}_i^k}\ov{h}_i^{k}}=\tau^{\ov{h}_i^k\sigma}\underset{\text{\refrel{rel:r9}}}{=}\tau^{\ov{h}_i^{k+1}},
    \]
    where IH stands for the inductive hypothesis.

    To prove \eqref{rel:q2kh}, we first show the base case $k=2$:

    \begin{align*}
    \text{\refrel{rel:r3}}:
    (\sigma^{\ov{h}_j}\sigma)^3=1\quad&\underset{\phantom{\ref{helper:conjugating_commutator}}}{\Longleftrightarrow} \quad  
    1= [h_j\sigma \sigma^{\ov{h}_j}, \sigma] \underset{\text{\refrel{rel:r5}}}{=}
    [h_j\sigma \sigma^{\ov{h}_i}, \sigma] \underset{\text{\refrel{rel:r4}}}{=} [h_j^{\ov{h}_i^2}, \sigma]\\
    &\underset{\ref{helper:conjugating_commutator}}\Longleftrightarrow \quad 
    [h_j, \sigma^{h_i^2}]=1,
    \end{align*}
    Here, we applied $\text{\refrel{rel:r4}}$ to plug in $h_j\sigma=h_j^{\ov{h}_i}$ twice. Now suppose the relation holds for a given $k\geq 2$. Then by $R_{1b}$, 
    \[
    1=[\sigma, \sigma^{h_i^k}]\underset{\substack{\text{IH}\&\\\ref{helper:commuting_prefix}}}{=}[h_j\sigma, \sigma^{h_i^k}]\underset{\text{\refrel{rel:r4}}}=[h_j^{\ov{h}_i}, \sigma^{h_i^k}] \quad\underset{\ref{helper:conjugating_commutator}}\Longleftrightarrow \quad [h_j, \sigma^{h_i^{k+1}}]=1. \qedhere
    \] 
    \end{proof}

    We now reduce 
    $R_{1f}: [\tau,  \sigma^{h_i^n}]=1$ for $n \neq 0, 1$.
    For $n\geq 2$,
    \[
    \text{\refrel{rel:r8}}: 1= [\tau, \sigma^{\ov{h}_j}] \underset{\ref{helper:conjugating_commutator}}{=} [\tau^{h_j}, \sigma] \underset{\substack{\eqref{rel:q1h}\&\\\cref{helper:conjugating_commutator}\\\ref{helper:commutator_to_conjugation}}}{=} [\tau^{h_j}, \sigma^{h_i^n}]\quad\underset{\substack{\eqref{rel:q2kh}\&\\\ref{helper:conjugating_commutator}}}\Longleftrightarrow \quad [\tau, \sigma^{h_i^n}]=1.
    \]
    For $n\leq -1$, we have
    \[
    \text{\refrel{rel:r8}}: 1= [\tau, \sigma^{\ov{h}_j}] \underset{\ref{helper:conjugating_commutator}}{=} [\tau^{h_j}, \sigma] \underset{\substack{\eqref{rel:q1h}\&\\\cref{helper:conjugating_commutator}\\\ref{helper:commutator_to_conjugation}}}{=}  [\tau^{h_j}, \sigma^{h_i^{n+1}}]\quad\underset{\substack{\eqref{rel:r9kh}\&\\\ref{helper:conjugating_commutator}}} \Longleftrightarrow \quad[\tau, \sigma^{h_i^n}]=1.
    \]
    showing that $R_{1f}$ follows from $P_r$.

    We now give one more auxiliary lemma. Note that \eqref{rel:-r9kh} has a slightly different form depending on the sign of $k$.

    \begin{LEM}\label{lem:auxRelations2}
        The relations of $P_r$ imply the following set of relations, for any $i\neq j$:
        \begin{align}  
        \tag{$\hat{q}_1^k$}
            & [\tau^{h_i^k}, h_j]=1 \qquad \text{for all $k \ge 1$} \label{rel:q1kh}\\
        \tag{$\hat{r}_9^k$} & \tau^{\ov{h}_i^k} = \tau^{\sigma^{\ov{h}_i^k}\cdots \sigma^{\ov{h}_i}} \qquad \text{for all $k \le -1$} \label{rel:-r9kh} \\
        \tag{$\hat{q}_3$} & [\eta^{h_i^2}, h_j]=1 \label{rel:q3h}
        \end{align}
    \end{LEM}
    \begin{proof} \renewcommand{\qedsymbol}{$\triangle$}
         The base case $k=1$ of \eqref{rel:q1kh} is precisely $\hat{q}_1$, hence the name. Now suppose the result is true for a given $k\geq 1$. Then by $R_{1f}$ and \Cref{helper:conjugating_commutator}, 

        \[1=[\tau^{h_i^k}, \sigma]\underset{\substack{\text{IH}\&\\\ref{helper:commuting_prefix}}}{=} [\tau^{h_i^k}, h_j\sigma]\underset{\text{\refrel{rel:r4}}}=[\tau^{h_i^k}, h_j^{\ov{h}_i}]\quad\underset{\ref{helper:conjugating_commutator}}\Longleftrightarrow \quad [\tau^{h_i^{k+1}}, h_j]=1,\]
        finishing the induction.

        Similarly, it is straightforward to see that \refrel{rel:r9} is equivalent to the base case $k=1$ of \eqref{rel:-r9kh}, hence the name. Now assume that \eqref{rel:-r9kh} is true for some $k \ge 1$. Then
        \[
         \tau^{\sigma^{h_i^{n}}\cdots \sigma^{h_i}} \underset{\text{IH}}{=}(\tau^{h_i^{n-1}})^{\sigma^{h_i^n}} \underset{\ref{helper:double_super_script}}{=} \tau^{h_i^n \sigma \ov{h}_i} \underset{\text{\refrel{rel:r9},\refrel{rel:r1}}}{=}\tau^{h_i^n},
    \]
    concluding the induction.

    For \eqref{rel:q3h}, note that by the first relation of $(R_a')$, we have
    \begin{align}
    \tag{$\hat{q}_4$}
    \eta^{\ov{h}_i}=\eta^{\sigma \sigma^{\ov{h}_i}} \underset{\text{\refrel{rel:r5}}}{=}  \eta^{\sigma \sigma^{\ov{h}_j}}  \underset{\text{\refrel{rel:r14}}}{=}  \eta^{\ov{h}_j}.
    \label{rel:q4h}
    \end{align}

    We also have,
    \[
    \eta^{h_i^2} \underset{\substack{\text{\refrel{rel:r14}}\&\\\ref{helper:double_super_script}}}{=} 
    \eta^{h_i \sigma \sigma^{h_i}} \underset{\substack{\eqref{rel:q4h}\&\\\ref{helper:double_super_script}}}{=} (\eta^{\ov{h}_j})^{h_i\sigma h_i \sigma} \underset{\substack{\text{\refrel{rel:r4}}\&\\\ref{helper:double_super_script}}}{=} \eta^{\ov{h}_j h_i^2} \underset{\ref{helper:double_super_script}}{=} (\eta^{h_i^2})^{\ov{h}_j},
    \]
    where in the second to the last equality we used \refrel{rel:r4} to plug in $h_i \sigma \ov{h}_j = \ov{h}_j h_i$ twice. By \Cref{helper:commutator_to_conjugation}, we obtain \eqref{rel:q3h}: $[\eta^{h_i^2}, h_j]=1$ as desired.
    \end{proof}

    We now proceed to reduce $R_{2c}$. We first show that the first and second relations in $R_{2c}$ reduce to $[\eta, \tau^{h_i^n}]=1$ for $n \le -2$ and $n \ge 1$ respectively. Indeed,

    \[
    \begin{cases}
            [\eta, \tau^{\sigma^{\ov{h}_i^n}\cdots \sigma^{\ov{h}_i}\sigma}]=1 & \text{for $n \ge 1$}\\
            [\eta, \tau^{\sigma^{h_i^n}\cdots \sigma^{h_i}}]=1 & \text{for $n \ge 1$}
    \end{cases}
    \quad\overunderset{\raisebox{.3em}{\scriptsize\eqref{rel:r9kh}}}{\eqref{rel:-r9kh}}{\Longleftrightarrow}\quad
    \begin{cases}
            [\eta, \tau^{\ov{h}_i^{n+1}}]=1 & \text{for $n \ge 1$}\\
            [\eta, \tau^{h_i^n}]=1 & \text{for $n \ge 1$}
    \end{cases},
    \]
    which is equivalent to $[\eta, \tau^{h_i^n}]=1$ for $n \neq 0, -1$. Thus $R_{2c}$ can be rewritten as 
    $R_{2c}: [\eta, \tau^{h_i^n}]=1$ for $n \neq 0, -1$. Using \refrel{rel:r16} $:[\eta, \tau^{h_i}]=1$ as the base case, applying
    \eqref{rel:q1h}, \Cref{helper:conjugating_commutator,helper:commutator_to_conjugation} we get $[\eta^{h_j^2}, \tau^{h_i}]=1$. Then for $n\geq 1$, 
    \[[\eta^{h_j^2}, \tau^{h_i}]=1\quad\underset{\substack{\eqref{rel:q3h}\&\\\ref{helper:conjugating_commutator}\\\ref{helper:commutator_to_conjugation}}}\Longleftrightarrow  \quad[\eta^{h_j^2}, \tau^{h_i^n}]=1\quad\underset{\substack{\eqref{rel:q1kh}\&\\\ref{helper:conjugating_commutator}\\\ref{helper:commutator_to_conjugation}}}\Longleftrightarrow  \quad [\eta, \tau^{h_i^n}]=1.\]
    
    For $n\leq -2$, we have

    \[[\eta^{h_j^2}, \tau^{h_i}]=1\quad\underset{\substack{\eqref{rel:q3h}\&\\\ref{helper:conjugating_commutator}\\\ref{helper:commutator_to_conjugation}}}\Longleftrightarrow  \quad [\eta^{h_j^2}, \tau^{h_i^{n+2}}]=1\quad\underset{\substack{\hat{r}_9^k\&\\\ref{helper:conjugating_commutator}}}\Longleftrightarrow  \quad[\eta, \tau^{h_i^n}]=1,\]
    showing that $R_{2c}$ is a consequence of $P_r$.

    Finally, we reduce the relations in $R_{2d}$. We begin with the first infinite set of relations of $R_{2d}$; $(R_{2d}-1): [\eta, \sigma^{h_i^n}]=1$ for $|n| \ge 2$. We use \refrel{rel:r13} as the base case, and note that it along with \eqref{rel:q2kh} and \Cref{helper:conjugating_commutator,helper:commutator_to_conjugation} imply that $[\eta^{h_j^2}, \sigma^{h_i^2}]=1$. Then for $n\geq 2$,
    \[[\eta^{h_j^2}, \sigma^{h_i^2}]=1 \quad\underset{\substack{\eqref{rel:q3h}\&\\\ref{helper:conjugating_commutator}\\\ref{helper:commutator_to_conjugation}}}\Longleftrightarrow  \quad[\eta^{h_j^2}, \sigma^{h_i^n}]=1\quad\underset{\substack{\eqref{rel:q2kh}\&\\\ref{helper:conjugating_commutator}\\\ref{helper:commutator_to_conjugation}}}\Longleftrightarrow  \quad[\eta, \sigma^{h_i^n}]=1.\]
    On the other hand, for $n\leq -2$, 
    \[[\eta^{h_j^2}, \sigma^{h_i^2}]=1 \quad\underset{\substack{\eqref{rel:q3h}\,\&\\\ref{helper:conjugating_commutator}\\\ref{helper:commutator_to_conjugation}}}\Longleftrightarrow  \quad[\eta^{h_j^2}, \sigma^{h_i^{n+2}}]=1\quad\underset{\substack{\text{\refrel{rel:r5}}\&\\\ref{helper:conjugating_commutator}}}\Longleftrightarrow  \quad[\eta, \sigma^{h_i^n}]=1,\]
    showing that the infinite set of relations in $(R_{2d}-1)$ is a consequence of $P_r$.

    For the second relation of $R_{2d}$, $[\eta, (\sigma^{\ov{h}_i})^{\sigma^{h_j}\sigma}]=1$, note that $(R_{2d}-1)$ (with $n=-2$) and \Cref{helper:commutator_to_conjugation} imply

    \[1=[\eta^{h_j}, \sigma^{\ov{h}_j}]\underset{\text{\refrel{rel:r14},\refrel{rel:r5}}}{=}[\eta^{\sigma \sigma^{h_j}}, \sigma^{\ov{h}_i}]\quad\underset{\ref{helper:conjugating_commutator}}\Longleftrightarrow \quad [\eta, (\sigma^{\ov{h}_i})^{\sigma^{h_j}\sigma}]=1.\] 
     This concludes the proof of \Cref{thm:PBrPresentation}.
    \end{proof}

\section{Brown's criterion and the Stein--Farley cube complex} 
\label{sec:BrownsCrit}

\subsection{Brown's criterion and complete join maps}
In order to prove \Cref{thm:finitenessGraphHoughton}, we follow a strategy similar to that found in the work of Aramayona--Bux--Flechsig--Petrosyan--Wu~\cite{aramayona2024asymptotic} and Aramayona--Bux--Kim--Leininger \cite{aramayona2023surface}. In particular, we use the following version of Brown's criterion:

\begin{THM}[{\cite[Theorem 3.1]{aramayona2023surface}}]\label{thm:BrownsCriterion}
    Let $G$ be a group acting by cell-wise isometries on a contractible piecewise Euclidean CW-complex $\mathbb{X}$. Suppose $\mathbb{X}$ is equipped with a discrete $G$-invariant Morse function $h:\mathbb{X}\to \R$, and let $\mathbb{X}^{\leq s}$ denote the largest subcomplex of $\mathbb{X}$ fully contained in $h^{-1}(-\infty, s]$. Suppose the following points hold.
    \begin{enumerate}[label=(\alph*)]
        \item The quotient $\mathbb{X}^{\leq s}/G$ is finite for all critical values $s$.
        \item Every cell stabilizer is of type $F_{\infty}$.
        \item There exists $d\geq 1$ so that for all sufficiently large critical values $s$ and every vertex $v\in \mathbb{X}$ with $h(v)\geq s$, the descending link of $v$ in $\mathbb{X}$ is $d$-spherical, i.e., $(d-1)$-connected and of dimension $d$.
        \item For every critcal value $s$, there is a vertex $v$ with noncontractible descending link with $h(v)\geq s$.
    \end{enumerate}
    Then $G$ is of type $F_d$ but not of type $FP_{d+1}$.
\end{THM}

Recall that a group $G$ is type $F_d$ if there exists a classifying space for $G$ with finite $d$-skeleton. A group is type $F_\infty$ if it is type $F_d$ for all $d > 0$.  
A group $G$ is type $FP_d$ if there exists a projective resolution of the trivial $\Z G$-module $\Z$ so that the first $d$ terms are finitely generated $\Z G$-modules.

Part (c) above is the part of the criterion which takes the most work. Proving it for the relevant complexes will be the bulk of this section.
\par 

We now establish the tools we will need to utilize  \Cref{thm:BrownsCriterion}. All of the following was introduced by Hatcher--Wahl in \cite{hatcher2010stabilization}.

\begin{DEF}[Complete join complex]
    Let $A$ and $B$ be simplicial complexes, and $\pi:A \to B$ a simplicial map. We say that $A$ is a \textit{complete join complex} over $B$ (with respect to $\pi$) if the following two properties hold.
    \begin{enumerate}
        \item $\pi$ is surjective, and injective on simplices.
        \item For each simplex $\sigma=\<v_0, \ldots, v_d\>$ in $B$, its preimage can be written as a join of fibers over the vertices:
        $$\pi^{-1}(\sigma)=\pi^{-1}(v_0)*\cdots \ast \pi^{-1}(v_d).$$
    \end{enumerate}
    In this case, $\pi$ is called a \emph{complete join}. 
\end{DEF}

For notational simplicity, we say every space is $(-2)$-connected, and every nonempty space is $(-1)$-connected.
\begin{DEF}[Weakly Cohen--Macaulay]
    A simplicial complex is \textit{weakly Cohen--Macaulay} of dimension $k+1$ if it is $k$-connected and if the link of every simplex $\sigma$ is $(k-\text{dim}(\sigma)-1)$-connected.
\end{DEF}

The reason why we are interested in this condition on the link is to be able to `lift' the connectivity properties of the image complex to the domain.
\begin{PROP}[{\cite[Proposition 3.5]{hatcher2010stabilization}}]\label{prop:CohenMacaulayLift}
    If $\pi:A\to B$ is a complete join and $B$ is weakly Cohen--Macaulay of dimension $k+1$, then so is $A$.
\end{PROP}

On the other hand, the connectivity descends as well. Recall $\pi:A \to B$ is a \textit{retract} if there is a section $\iota:B \to A$ such that $\pi \circ \iota = \id_{B}$. Note this implies that $\pi$ induces surjection on homotopy groups; $\pi_*: \pi_{\bullet}(A) \to \pi_{\bullet}(B)$.
\begin{PROP}[{\cite[Remark A.15]{aramayona2024asymptotic}}]\label{prop:completeJoinRetract}
    Suppose $\pi:A\to B$ is a complete join. Then $B$ is a retract of $A$. In fact, $B$ inherits all properties from $A$ that can be expressed by the vanishing of group-valued functors or cofunctors. In particular, if $A$ is $k$-connected, then so is $B$. 
\end{PROP}

We remark here that we are skipping the discussion of \textit{bad simplices}, which are used in \cite{aramayona2023surface}. This is because our proof of \Cref{thm:injTetheredHandleConnectivity} exactly mirrors their proof of \cite[Theorem 5.7]{aramayona2023surface}, which is the only place bad simplices are used.

\subsection{Finiteness of $\Map(M_{n,s})$}
To apply \Cref{thm:BrownsCriterion}, we will need to show that cell stabilizers are type $F_{\infty}$. This will follow from the next two results. Here we denote by $\Gamma_{n,s}$ a rank $n$ graph with $s$ rays attached, and $M_{n,s}$ its associated doubled handlebody (sometimes the notation $M_{n,s}$ will be used to denote the rank $n$ doubled handlebody with $s$ open balls removed as well).
\begin{LEM}[{\cite[Corollary 2.7]{aramayona2024asymptotic}}]\label{lem:DHB_Finfinity}
    The group $\Map(M_{n,s})$ has type $F_{\infty}$.
\end{LEM}

\begin{THM}[{\cite[Proof of Proposition 1]{hatcher2004homology}}, {\cite[Proposition 3.3]{udall2024the-sphere}}]
\label{thm:TwoMCGfiniteKernel}
There is a homomorphism $\Map(M_{n,s})\to \Map(\G_{n,s})$ with a finite kernel.
\end{THM}
In particular, $\Map(\Gamma_{n,s})$ is also of type $F_{\infty}$ as it is the quotient of $\Map(M_{n,s})$ with kernel a finite group.

\subsection{Two contractible Stein--Farley cube complexes}
Recall that we write $B_r$ to denote the group $B(0,1,r)$. Consider ordered pairs $(Z, f)$ where $Z\subset \Gamma_r$ is a suited subgraph and $f\in B_r$. Two such pairs $(Z_1, f_1)$ and $(Z_2, f_2)$ are \textit{equivalent} if $f_2$ has a proper homotopy inverse $g_2$ such that $g_2 \circ f_1(Z_1)=Z_2$ and $g_2\circ f_1$ is rigid in the complement of $Z_1$. We denote the equivalence class of $(Z,f)$ by $[Z,f]$, and the set of all equivalence classes by $\mathcal{S}$. Then $B_r$ acts on $\mathcal{S}$ on the left, letting
$$g\cdot [Z, f] = [Z, g\circ f].$$

We define the \emph{complexity} $h(x)$ of $x=[Z, f]$ to the rank of $Z$. To see that this doesn't depend on the representative, suppose $[Z, f]=[Z', g]$. Then by definition, $g^{-1}\circ f(Z)=Z'$ and $g^{-1}\circ f$ is rigid outside of $Z$. This means that $g^{-1}\circ f$ sends pieces in the complement of $Z$ to pieces in the complement of $Z'$. As $g^{-1}\circ f$ induces a $\pi_1(\G_r)$ automorphism, we see that $g^{-1}\circ f|_Z$ induces an automorphism between $\pi_1(Z)$ and $\pi_1(Z')$. In particular their ranks are equal.

Given $x_1, x_2 \in \mathcal{S}$, we say that $x_1\preceq x_2$ if there are representatives $(Z_i,f_i)$ of $x_i$, for $i=1,2$, so that $f_1=f_2$, $Z_1\subset Z_2$, and $\overline{Z_2\setminus Z_1}$ is a (possibly empty) disjoint union of pieces. We construct a \emph{Stein--Farley cube complex} $X$ via this relation with $0$-skeleton $\mathcal{S}$. Given $x_1 \preceq x_2$ with $d=h(x_2)-h(x_1)$, the set $\{x \ | \ x_1 \preceq x \preceq x_2\}$ comprises the vertices of a $d$-cube. Then $X$ is $r$-dimensional (see \Cref{fig:CubeCplx}).
\begin{figure}[ht!]
    \centering
    \includegraphics[width=.85\textwidth]{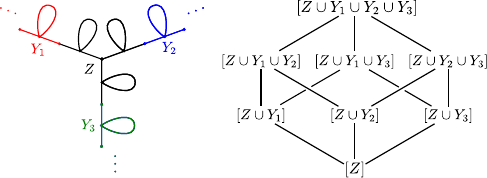}
    \caption{The Stein--Farely cube complex $X$ associated to $B_3$ is 3-dimensional, and a
      3-cube of it is illustrated above. For simplicity, we take $f = \mathrm{Id}$ and suppress the notation to $[Z] = [Z, \mathrm{Id}]$, etc.}
    \label{fig:CubeCplx}
\end{figure}

Clearly $B_r$ preserves the cubical structure, and $h$ extends to a $B_r$-invariant complexity function
$$h:X\to \R_+.$$
Let $X^{\leq k}$ denote the subcomplex of $X$ spanned by vertices whose complexity is less than or equal to $k$.
\par 
We can define a similar complex $\mathcal{X}$ which is associated to
$\BB_r:=\BB(0,1,r)$, defined in the doubled handlebody $M_{\G_r}$ of $\Gamma_r$,
which we will denote by $M_r$. Consider ordered pairs $\{(\ZZ, f)\}$ where $\ZZ$ is a suited submanifold and $f\in \BB_r$. Two such pairs $(\ZZ_1, f_1)$ and $(\ZZ_2, f_2)$ are \textit{equivalent} if $f_2^{-1} \circ f_1(\ZZ_1)=\ZZ_2$ and $f_2^{-1}\circ f_1$ is rigid in the complement of $\ZZ_1$. Equivalence classes and the $\BB_r$-action can all be defined in the same way as for $X$. Complexity can be defined in the same way as for graphs, using the rank of the fundamental group of the suited submanifold.

Recall the surjective map $\Psi_B: \BB_r\to B_r$ from \Cref{lem:psirestriction}.

\begin{LEM}\label{lem:CCSurj}
    There is a complexity invariant surjection $\Phi: \mathcal{X} \to X$, where $\Phi([\ZZ, f])=[\rho(\ZZ), \Psi_B(f)]$ which is equivariant with respect to $\Psi_B$, i.e.\ $\Phi(g(x))=\Psi_B(g)(\Phi(x))$ for $g\in \BB_r$ and $x\in\mathcal{X}$. The pre-images of any cube are in one-to-one correspondence with the pre-images of its bottom vertex $[Z,f]$, 
    which are in one-to-one correspondence with the subgroup generated by sphere twists supported in $\overline{M_r \setminus \ZZ}$, where $\ZZ = \rho^{-1}(Z)$. 
\end{LEM}
\begin{proof}
    We first show that $\Phi$ is well defined. Suppose $(\ZZ_1, f_1)$ is equivalent to $(\ZZ_2, f_2)$. Then $f_2^{-1}f_1(\ZZ_1)=\ZZ_2$, so $\rho(f_2^{-1}f_1(\ZZ_1))=\rho(\ZZ_2)$. It follows from the definition of $\Psi_B$ that $\Psi_B(f_2^{-1}f_1)\rho(\ZZ_1)=\rho(\ZZ_2)$ and $\Psi_B(f_2^{-1}f_1)$ is rigid outside of $\rho(\ZZ_1)$. This is because $\Psi_B$ is defined using $\rho$, and $\rho$ is chosen to be compatible with the rigid structures of $\BB_r$ and $B_r$. In particular, $(\rho(\ZZ_1), \Psi_B(f_1))$ is equivalent to $(\rho(\ZZ_2), \Psi_B(f_2))$.
    \par 
    By construction, $\Phi$ is surjective since $\Psi_B$ is surjective, and $\rho$ induces a bijection from the set of suited submanifolds to the set of suited subgraphs. For a suited submanifold $\ZZ$, the ranks of $\ZZ$ and $\rho(\ZZ)$ are the same and using the natural identifications of $\pi_1(M_r)$ and $\pi_1(\G_r)$ using $\rho$ and $i$, the actions of $f\in \BB_r$ and $\Psi_B(f)\in B_r$ on the fundamental group are the same. Thus, $\Phi$ is complexity invariant. Equivariance follows from the fact that $\Psi_B$ is a homomorphism. 
    \par 
    Let $C$ be a cube with bottom vertex $[Z, f]$. Suppose $\Phi(C')=C$. Then because $\Phi$ is height preserving and there is a unique vertex in $C'$ of minimal height, it follows that the set of cubes which map to $C$ are in bijection with the set of vertices mapping to $[Z, f]$. 
    \par 
    To understand the preimage of a vertex $[Z,f]$, we may assume by equivariance that $f=\id$. Given some $\Phi([\ZZ', g'])=[Z,\id]$, we can choose a representative $[\ZZ',g'] = [\ZZ,g]$, with $\rho(\ZZ) = Z$, $g$ compactly supported, and $g(\ZZ) = \ZZ$. Modifying $g$ by a mapping class supported in $\ZZ$, we can assume that $g|_{\ZZ} = \id|_{\ZZ}$. Thus, as $\id^{-1}\circ\Psi_B(g)$ is rigid outside $Z$, we have that $\Psi_B(g) = \id$. By \Cref{thm:doubledHandlebodytoGraph}, $g$ is a product of sphere twists. Finally, we observe that sphere twists supported in $\ZZ$ do not change $[\ZZ,g]$, while those supported outside $\ZZ$ do.
\end{proof}

\begin{RMK}
    Recall from \Cref{rmk:SphereTwistsGenerators} that the sphere twists are
    generated by those which occur about the core spheres of the pieces.
    This justifies that for a suited submanifold $\ZZ$ and an adjacent piece
    $\YY$, we need only consider sphere twists supported in $\ZZ$, in $\YY$, or outside $\ZZ\cup \YY$, \end{RMK}

Recall the descending link $\lkd(x)$ of a vertex $x$ is the part of its link
coming from the cells that contain $x$ as their unique top vertex. As $\Phi$~
preserves adjacency and height, it descends to a map $\Phid$ of descending links.

\begin{LEM}\label{lem:descendingLinksCompleteJoin}
    Let $x\in \mathcal{X}$ be a vertex. Then $\Phid:\lkd_{\mathcal{X}}(x) \to \lkd_{X}(\Phi(x))$ is a complete join.
\end{LEM}
\begin{proof}
    Let $v\in\lkd_{X}(\Phi(x))$ be a vertex. Recall that a vertex of the descending link of $\Phi(x)$ corresponds to an edge whose top vertex is $\Phi(x)$. Given any representative $[Z,f]$ for the bottom vertex of this edge in $X$, we can represent $\Phi(x)$ as $[Z\cup Y,f]$, for $Y$ a piece adjacent to $Z$ in $\G_r$. 
    
    Then the fiber over $v$ is a copy of the subgroup generated by sphere twists supported in $\mathcal{Y} = \rho^{-1}(Y)$. To see this, it suffices by \Cref{lem:CCSurj} to look at sphere twists supported outside $\ZZ=\rho^{-1}(Z)$. But note that a sphere twist supported outside $\ZZ\cup \mathcal{Y}$ would change $x$, and composing with such a sphere twist would not yield an element of $\lkd_{\mathcal{X}}(x)$. Therefore, only sphere twists supported in $\mathcal{Y}$ can preserve the $\Phid$ fiber of $v$.
    
    Given any simplex $\sigma$ of $\lkd_{X}(\Phi(x))$, choose a representative $[Z,f]$ for the bottom vertex of the corresponding cube in $X$. Then $\Phi(x)$ is represented by $[Z\cup Y_1 \cup \cdots \cup Y_d, f]$, where the $Y_i$ are pieces adjacent to $Z$ in distinct ends. Denote the vertices of $\sigma$ by $v_i$, where $v_i$ has a representative $[Z\cup Y_1 \cup \cdots \cup \widehat{Y}_i\cup \cdots \cup Y_d, f]$, where the hat indicates exclusion. Writing $\ZZ=\rho^{-1}(Z)$ and $\mathcal{Y}_i = \rho^{-1}(Y)$, we have that $x=[\ZZ\cup \mathcal{Y}_1 \cup \cdots \cup \mathcal{Y}_d, h]$ for some $h\in\BB_r$ such that $\Psi_B(h)=f$. Modifying $h$ by a sphere twist supported in $\mathcal{Y}_j$ only changes the preimage under $\Phid$ of $v_j$, leaving the $\Phid$-preimages of all other $v_i$ alone. As $\Phid$ is injective on simplices, and is easily seen to be surjective, we see that it is a complete join.  
\end{proof}
\begin{THM}\label{thm:CubeComplexContractible} 
    The Stein--Farley cube complexes $X$ and $\mathcal{X}$ are contractible, and the following holds.
    \begin{itemize}
        \item Let $C$ be a cube in either $X$ or $\mathcal{X}$ with bottom
          vertex $x=[Z,f]$. Then the $B_r$-stabilizer (or the $\BB_r$-stabilizer respectively) of $C$ is isomorphic to a finite index subgroup of the mapping class group of $Z$, (thinking of the valence $1$ vertices in the case that $Z$ is a graph as ends). In particular, every cube stabilizer is of type $F_{\infty}$. 
        
        \item For every $k\geq 1$, the quotient spaces $X^{\leq k}/B_r$ and $\mathcal{X}^{\leq k}/\BB_r$ are compact. 
    \end{itemize}
\end{THM}
\begin{proof}
    The contractibility follows from the arguments of Proposition 5.7 of \cite{aramayona2024asymptotic}. 
    \par 
    The first and second bullet points follow from the arguments in Lemma 6.3 and Lemma 6.2 of \cite {aramayona2024asymptotic}, respectively.
\end{proof}

    \par

\subsection{Piece complex and tethered handle complex} Following \cite{aramayona2024asymptotic} and \cite{aramayona2023surface}, we can determine the connectivity of descending links using the \emph{piece complex}, since each $\lkd(x)$ is a complete join over it. We now recall the definition.

    \begin{DEF}[Piece complex]
        Let $\ZZ$ be a compact doubled handlebody with finitely many open balls removed, and let $Q$ be a collection of boundary spheres of $\ZZ$. The \textit{piece complex} $\mathcal{P}(\ZZ,Q)$ is the simplicial complex whose vertices are isotopy classes of \textit{pieces}, i.e.\ submanifolds of $\ZZ$ homeomorphic to $M_{1,2}$, with one boundary component in $Q$, and with simplices given by pairwise disjointness. If $Q=\partial \ZZ$, we write $\mathcal{P}(\ZZ,Q)=\mathcal{P}(\ZZ)$.
    \end{DEF}
    
    The complex $\mathcal{P}(\ZZ,Q)$ can be identified with a subcomplex of the sphere complex of $\ZZ$, as pieces are in one-to-one correspondence with their non-peripheral separating boundary components, and the boundary components of two pieces intersect if and only if the pieces themselves intersect. Of course, here the word piece is being used in a slightly different way compared to the definition of a piece in a rigid structure, but the two notions are closely related (see \Cref{fig:PieceComplex}). 
    To keep figures simple, we will illustrate doubled handlebodies by drawing only one handlebody, with the other understood to be a symmetric copy.
    
    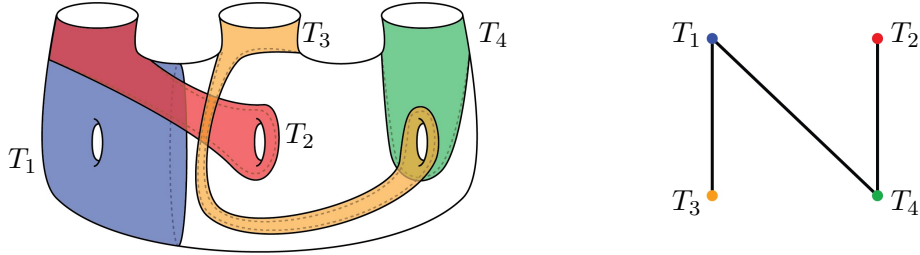
\begin{figure}[ht!]
    \centering
    \begin{overpic}[width=.8\textwidth]{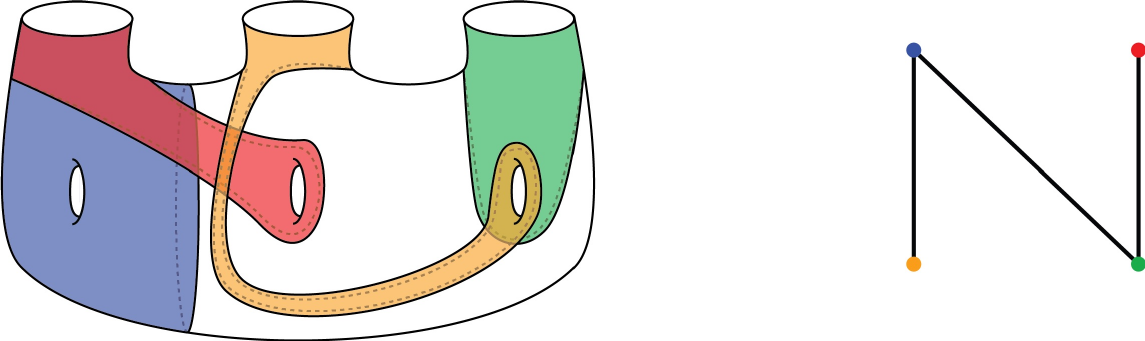}
    \put(-4,10){$T_1$}
    \put(29,13){$T_2$}
    \put(31,25){$T_3$}
    \put(52,25){$T_4$}

    \put(75,25){$T_1$}
    \put(101,25){$T_2$}
    \put(75,5){$T_3$}
    \put(101,5){$T_4$}
    \end{overpic}
    \caption{For the doubled handlebody $\ZZ = M_{3,3}$, we illustrate a part of the piece complex involving the spheres $T_1, T_2, T_3,$ and $T_4$ pictured above.  Here $Q = \partial \ZZ$.}
    \label{fig:PieceComplex}
\end{figure}

Note that every vertex of $\XX$ can be sent by an element of $\BB_r$ to one of the form $[\ZZ, \id]$, so it suffices to analyze the descending links of vertices of this form. We will implicitly use this fact for the rest of the paper.

We define for $x=[\ZZ, \id]\in \XX$ the map 
    $$\Pi: \lkd(x)\to \mathcal{P}(\ZZ)$$
    as follows, following the definition of the map used in \cite[Proposition 6.6]{aramayona2024asymptotic}. Fix a $p$-simplex $z$ in $\lkd(x)$. This simplex is determined by a cube, which is itself determined by its top $[\ZZ', g]$ and bottom $[\ZZ'',g]$ vertices. Here, the manifold $\ZZ'$ is obtained from $\ZZ''$ by adding pairwise disjoint pieces $\mathcal{Y}_0, \ldots, \mathcal{Y}_p$ so that $g$ maps $\ZZ'$ to $\ZZ$ and $\ZZ'$ is a support for $g$. We define
    $$\Pi(z)=\{g(\mathcal{Y}_0), \ldots, g(\mathcal{Y}_p)\}.$$ 
    
    \par 
    The proofs of the following results follow arguments to those referenced in the statements.  
    \begin{PROP}[{\cite[Proposition 6.6]{aramayona2024asymptotic}}]\label{prop:descendingPieceComplex}
        The map $\Pi$ is well defined and is a complete join.
    \end{PROP}

\begin{LEM}[{\cite[Lemma 4.5]{aramayona2023surface}}]\label{lem:infiniteFibers}
    For every vertex $v\in \mathcal{P}(\ZZ)$, its $\Pi$-fiber in $\lkd(x)$ is infinite.
\end{LEM}

Our main goal then is to prove the following theorem.
\begin{THM}\label{thm:piececomplexkconnected}
    The piece complex $\mathcal{P}(\ZZ,Q)$ of a compact doubled handlebody with finitely many open balls removed is $k$-connected, provided that $\text{rk}(\pi_1(\ZZ))\geq 4k+4$ and $|Q|\geq k+2$.
\end{THM}
To prove \Cref{thm:piececomplexkconnected}, we investigate two more closely related complexes.
\par 
A \textit{tethered handle} of $\ZZ$ is a pair of a \textit{handle} $T$, which is a submanifold of $\ZZ$ homeomorphic to $M_{1,1}$, and a \textit{tether}, which is an arc joining $\partial T$ to a component of $\partial \ZZ$. 

\begin{DEF}
    Let $\ZZ$ be a compact doubled handlebody with finitely many balls removed, and let $Q$ be a collection of boundary spheres of $\ZZ$. The \textit{tethered handle complex} $\mathcal{TH}(\ZZ,Q)$ is the simplicial complex whose $d$-simplices are sets of isotopy classes of $d+1$ pairwise disjoint tethered handles, each tethered to an element of $Q$.
\end{DEF}

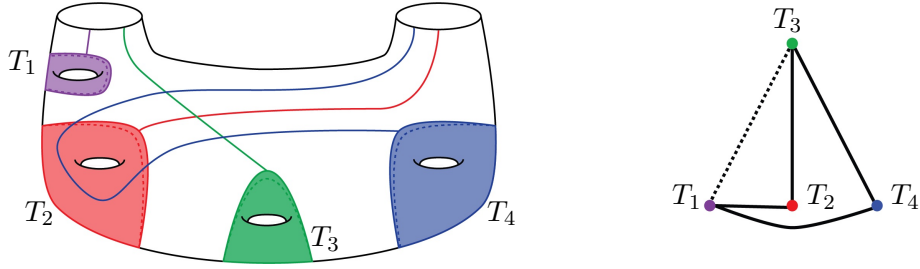
\begin{figure}[ht!]
    \centering
    \begin{overpic}[width=.8\textwidth]{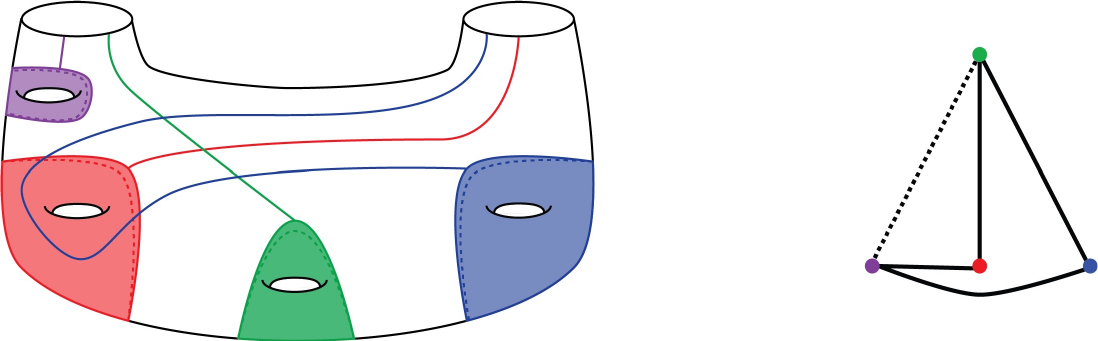}
        \put(-4,23){$T_1$}
        \put(-2,5){$T_2$}
        \put(32,2){$T_3$}
        \put(53,5){$T_4$}

        \put(75,7){$T_1$}
        \put(91,7){$T_2$}
        \put(87,28){$T_3$}
        \put(101,7){$T_4$}
    \end{overpic}
    \caption{For the doubled handlebody $\ZZ = M_{4,2}$, we illustrate a part of the tethered handle complex, $\mathcal{TH}(\ZZ, Q)$, involving the tethered handles $T_1, T_2, T_3$ and $T_4$ pictured above.  Here $Q = \partial \ZZ$. While the tether of $T_3$ appears to intersect the tethers of $T_2$ and $T_4$, the tethers are realized disjointly in the doubled handlebody. The dotted edge does not appear in ${TH}_1(\ZZ,Q)$ as $T_1$ and $T_3$ are tethered to the same boundary component.}
    \label{fig:TetheredHandleComplex}
\end{figure}

\begin{LEM}\label{lem:thetheredHandleConnected}
    The tethered handle complex $\mathcal{TH}(\ZZ,Q)$ is $k$-connected, provided that $Q$ is not empty and $\text{rk}(\pi_1(\ZZ))\geq 4k+4$.
\end{LEM}
\begin{proof}
    This follows directly from \cite[Lemma 8.11]{aramayona2024asymptotic} and the proof of  \cite[Corollary 8.13]{aramayona2024asymptotic}.
\end{proof}

\begin{DEF}
    The \textit{injective tethered handle complex} $\mathcal{TH}_1(\ZZ,Q)$ is the subcomplex of $\mathcal{TH}(\ZZ,Q)$ consisting of simplices whose involved handles are tethered to pairwise distinct boundary components in $Q$.
\end{DEF}
Note that a neighborhood of a tethered handle and the boundary component of $\ZZ$ the tether connects to is a piece of $\ZZ$. From this, one obtains a simplicial map
$$\pi: \mathcal{TH}_1(\ZZ,Q)\to \mathcal{P}(\ZZ,Q)$$
given by sending a simplex of $\mathcal{TH}_1(\ZZ,Q)$ to collections of pieces obtained by taking neighborhoods of the vertices of the simplex along with the boundary component the vertex connects to. This map is simplicial as by the definition of $\mathcal{TH}_1(\ZZ,Q)$, each pair of vertices of a simplex connect to two distinct elements of $Q$, so the neighborhoods of the disjoint tethered handles are disjoint.
\par 
Then we have the following for $\pi$.
\begin{LEM}\label{lem:tetheredHandletoPiece}
    The map $\pi: \mathcal{TH}_1(\ZZ,Q) \to \mathcal{P}(\ZZ,Q)$ is a complete join.
\end{LEM}
\begin{proof}
    This follows as in \cite[Lemma 7.8]{aramayona2024asymptotic}. The proof there deals with the so-called ``handle-tether-ball" complex as opposed to the tethered handle complex that we have defined here. To modify the proof of \cite[Lemma 7.8]{aramayona2024asymptotic}, one can drop all references to the balls, and replace them with the boundary components that the tethers connect to.
\end{proof}
\par 
We now come to the key result to obtain \Cref{thm:piececomplexkconnected}.

\begin{THM}\label{thm:injTetheredHandleConnectivity}
    The injective tethered handle complex $\mathcal{TH}_1(\ZZ,Q)$ is $k$-connected, provided $\text{rk}(\pi_1(\ZZ))\geq 4k+4$ and $|Q|\geq k+2$.
\end{THM}
\begin{proof}
    This follows as in the proof of \cite[Theorem 5.7]{aramayona2023surface}. Namely, the bound on the rank of $\pi_1(\ZZ)$ comes from \Cref{lem:thetheredHandleConnected}, and the bound on the number of boundary components is a result of the inductive argument given in  \cite[Theorem 5.7]{aramayona2023surface}. Indeed, following their argument we can identify $\mathcal{TH}_1(\ZZ,Q)$ as the \emph{good} subcomplex of $\mathcal{TH}(\ZZ,Q)$, and the \emph{good link} of a \emph{bad simplex} of $\mathcal{TH}(\ZZ,Q)$ coincides with some injective tethered handle complex of a doubled handlebody of reduced rank and reduced number of boundary components, where we can use the inductive argument. As the inductive argument is similar to the one appearing in \Cref{cor:pieceCpxCohenMacaulay}, we do not repeat it here. 
\end{proof}

\begin{proof}[Proof of \Cref{thm:piececomplexkconnected}]
    \Cref{lem:tetheredHandletoPiece} and \Cref{thm:injTetheredHandleConnectivity} along with \Cref{prop:completeJoinRetract} imply the result. 
\end{proof}

In particular, we have the following:
\begin{COR}[{\cite[Corollary 5.2]{aramayona2023surface}}]\label{cor:pieceCpxCohenMacaulay}
   The piece complex $\mathcal{P}(\ZZ,Q)$ is weakly Cohen--Macaulay of dimension $k+1$, provided that $g(\ZZ)\geq 4k+4$ and $|Q|\geq k+2$.
\end{COR}
\begin{proof}
    The link of a $d$-simplex $\sigma$ in $\mathcal{P}(\ZZ,Q)$ is isomorphic to $\mathcal{P}(\ZZ',Q')$, where $\ZZ'$ is obtained from $\ZZ$ by cutting off the pieces corresponding to the vertices of $\sigma$, and $Q'$ is the subset of $Q$ that lies in $\ZZ'$. We have the following.
   $$g(\ZZ')= g(\ZZ)-(\mathrm{dim}(\sigma)+1)\geq 4k+4-\mathrm{dim}(\sigma)-1\geq 4(k-\mathrm{dim}(\sigma)-1)+4$$
   and 
   $$|Q'|=|Q|-(\text{dim}(\sigma)+1)\geq (k-\text{dim}(\sigma)-1)+2.$$
   Hence the link of $\sigma$ is $(k-\text{dim}(\sigma)-1)$-connected by \Cref{thm:piececomplexkconnected}. Thus, by definition, $\mathcal{P}(\ZZ,Q)$ is weakly Cohen--Macaulay.
\end{proof}

\begin{COR}\label{cor:descendingLinkswCM}
    The descending link $\lkd(x)$ of a vertex $x$ with $h(x)\geq 4r-4$ is weakly Cohen-Macaulay of dimension $r-1$.
\end{COR}
\begin{proof}
    Combining \Cref{prop:descendingPieceComplex} and \Cref{cor:pieceCpxCohenMacaulay} with \Cref{prop:CohenMacaulayLift} completes the proof, as we can take $Q$ to be the entire set of boundary components, and in this case $|Q|=r$.
\end{proof}

\begin{RMK}
    An application of Bestvina--Brady Morse theory (as in \cite{bestvina1997morse}) leads to the computation that $H^{r}(B_r;\Z)$ and $H^{r}(PB_r;\Z)$ are infinite dimensional, a (possibly) stronger statement than being not of type $FP_r$. To see this, note that (starting from a sufficient height), increasing under the Morse function involves filling in not only infinitely many $r$-spheres, but infinitely many $(P)B_r$-families of $r$-spheres. Specifically, there is at least one new $r$-sphere at each height (above some finite height).
\end{RMK}

\section{Proof of \Cref{thm:finitenessGraphHoughton}}\label{sec:proofOfMainThm}

In order to prove \Cref{thm:finitenessGraphHoughton} in full generality, we first adapt the techniques of \cite{aramayona2024isomorphisms} to prove the following result in the context of graphs and doubled-handlebodies:

\begin{LEM}[{cf. \cite[Theorem 1.1 (1)]{aramayona2024isomorphisms}}] 
    Let $g \ge 0$, $h \ge 1$ and $r \ge 2$.  Then $B(g, h, r)$ is commensurable to $B(g',h',r')$ if $r = r'$. The same statement holds for $\BB(g,h,r)$ and $\BB(g',h',r')$. 
    \label{thm:commensurability}
\end{LEM}

We note that the converse of \Cref{thm:commensurability} also holds. It is not required for the proof, but instead follows from \Cref{thm:finitenessGraphHoughton}, as explained after its proof.   For completeness, we outline the proof of \Cref{thm:commensurability} and describe the necessary adaptations to our setting. While the statements are written only in the language of graphs for clarity, the same statements and arguments apply to doubled-handlebodies and their Houghton groups with the appropriate substitutions.  

\begin{DEF}
    A $(0,1)$-rigid structure $\Gamma_r = C \cup \bigcup_{j \in J} Y_j$ with markings $\phi_i \colon Y_i \to Y^1$, is said to \emph{engulf} a $(g,h)$-rigid structure $\Gamma_r = C' \cup \bigcup_{\ell \in L} Y_\ell'$ with $\phi_k': Y_k' \to Y^h$ provided that: 
    \begin{itemize}
        \item The core $C'$ is formed by adding $g$-many pieces $Y_i$ of the $(0,1)$-rigid structure to its core $C$, so that $C'$ is connected. 
        \item Each piece $Y_k'$ of the $(g,h)$-rigid structure is formed by joining $h$-many pairwise adjacent pieces, say $Y_1, \dots, Y_h$, of the $(0,1)$-rigid structure (see \Cref{fig:Engulfing}).
        \item The markings $\phi_k' : Y_k' \to Y^h$ are constructed from
          markings $\phi_i:Y_i \to Y^1$ on the pieces it contains.  
          By wedging together the copies
          of $Y^1$ associated to adjacent pieces $Y_i$, we obtain a
          corresponding decomposition of $Y^h$ into $h$-many copies of $Y^1$.
          Let $\iota_i \colon Y^1 \to Y^h$ be the inclusion of $Y^1$ to the
          $i$-th piece of $Y^h$.  Then we construct the marking on the
          $(g,h)$-rigid structure $\phi_k'$ so that $\phi_k'|_{Y_i} = \iota_i
          \circ \phi_i$ for all $i$. 
    \end{itemize}
\end{DEF}

 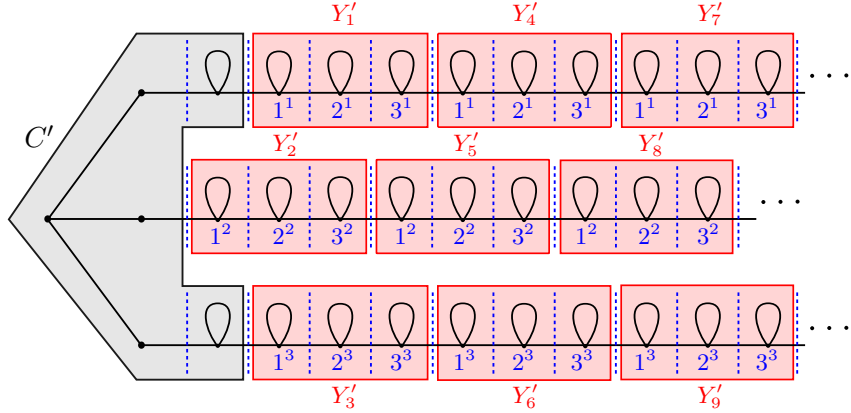
\begin{figure}[ht!]
    \centering
    \vspace{.5cm}
    \begin{overpic}[width=.8\textwidth]{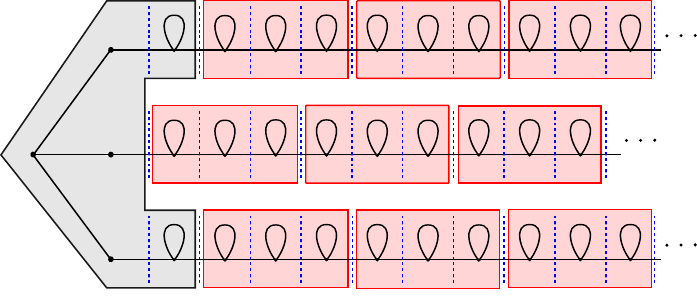}  
        \put(31.5,33.5){\color{blue}\footnotesize{$1^1$}}
        \put(38.5,33.5){\color{blue}\footnotesize{$2^1$}}
        \put(45.5,33.5){\color{blue}\footnotesize{$3^1$}}
        \put(53,33.5){\color{blue}\footnotesize{$1^1$}}
        \put(60,33.5){\color{blue}\footnotesize{$2^1$}}
        \put(67,33.5){\color{blue}\footnotesize{$3^1$}}
        \put(74.5,33.5){\color{blue}\footnotesize{$1^1$}}
        \put(82,33.5){\color{blue}\footnotesize{$2^1$}}
        \put(89,33.5){\color{blue}\footnotesize{$3^1$}}

        \put(24,18.5){\color{blue}\footnotesize{$1^2$}}
        \put(31.5,18.5){\color{blue}\footnotesize{$2^2$}}
        \put(38.5,18.5){\color{blue}\footnotesize{$3^2$}}
        \put(46,18.5){\color{blue}\footnotesize{$1^2$}}
        \put(53,18.5){\color{blue}\footnotesize{$2^2$}}
        \put(60,18.5){\color{blue}\footnotesize{$3^2$}}
        \put(67.5,18.5){\color{blue}\footnotesize{$1^2$}}
        \put(75,18.5){\color{blue}\footnotesize{$2^2$}}
        \put(82,18.5){\color{blue}\footnotesize{$3^2$}}

        \put(31.5,3.5){\color{blue}\footnotesize{$1^3$}}
        \put(38.5,3.5){\color{blue}\footnotesize{$2^3$}}
        \put(45.5,3.5){\color{blue}\footnotesize{$3^3$}}
        \put(53,3.5){\color{blue}\footnotesize{$1^3$}}
        \put(60,3.5){\color{blue}\footnotesize{$2^3$}}
        \put(67,3.5){\color{blue}\footnotesize{$3^3$}}
        \put(74.5,3.5){\color{blue}\footnotesize{$1^3$}}
        \put(82,3.5){\color{blue}\footnotesize{$2^3$}}
        \put(89,3.5){\color{blue}\footnotesize{$3^3$}}

        \put(2,30){$C'$}

        \put(38.5,45){\footnotesize\color{red}$Y'_1$}
        \put(60,45){\footnotesize\color{red}$Y'_4$}
        \put(82,45){\footnotesize\color{red}$Y'_7$}

        \put(31.5,29.5){\footnotesize\color{red}$Y'_2$}
        \put(53,29.5){\footnotesize\color{red}$Y'_5$}
        \put(75,29.5){\footnotesize\color{red}$Y'_8$}

        \put(38.5,-0.5){\footnotesize\color{red}$Y'_3$}
        \put(60,-0.5){\footnotesize\color{red}$Y'_6$}
        \put(82,-0.5){\footnotesize\color{red}$Y'_9$}
        
    \end{overpic}
    \caption{A $(0,1)$-rigid structure on $\Gamma_3$ engulfing a $(2,3)$-rigid structure. Note that loop shifts induce cyclic permutations of pieces of the $(0,1)$-rigid structure; that is, a shift into end $n$ induces the permutation $1^n \mapsto 2^n \mapsto 3^n \mapsto 1^n$.}
    \label{fig:Engulfing}
\end{figure}

    \begin{proof}[Proof of \Cref{thm:commensurability}]
        The proof of \cite[Theorem 1.1 (1)]{aramayona2024isomorphisms} can be adapted from the setting of surface Houghton groups as follows.     

        If a $(g,h)$-rigid structure is engulfed by a $(0,1)$-rigid
        structure, then $B(g, h, r)$ is a finite index subgroup of $B_r$ (see
        \cite[Lemma 3.1]{aramayona2024isomorphisms}).  If the $(g,h)$-rigid
        structure is not engulfed by the $(0,1)$-rigid structure, then $B(g, h,
        r)$ is conjugate to such a finite index subgroup of $B_r$ (see
        \cite[Remark 2]{aramayona2024isomorphisms}).    Because of compatibility
        between the pieces of the $(g,h)$-rigid structure and the $(0,1)$-rigid
        structure, any suited subgraph for $B(g,h,r)$, is suited for $B_r$.
        Furthermore, since the marking maps $\phi_k’: Y_k’ \to Y^h$ are
        constructed from markings of the engulfing (0,1)-rigid structure,
        $B(g,h,r) \le B_r$.  Since there are only finitely many ends, the pure
        graph Houghton groups have finite index in their respective full graph
        Houghton groups, so it suffices to prove $[PB_r: PB(g,h,r)]< \infty$.
        Fix an end $e_n$ and consider the pieces, $Y_k’ \cong Y^h$, in a
        neighborhood of that end.  By engulfing, each piece decomposes as $h$
        many pieces $Y_i \cong Y^1$ which can be labeled $1^n,...,h^n$ (see \Cref{fig:Engulfing}).  In a
        neighborhood of an end, an $f \in PB_r$ acts as a shift, and can at most cyclically permute the order of the labels in the piece $Y_k’$.  Doing this in neighborhoods of all $r$ ends, gives an action of $PB_r$ on a finite set, namely $\sqcup_{n = 1}^r \{1^n,..., h^n\}$ with kernel precisely $PB(g,h,r)$. 
\end{proof}

Equipped with \Cref{thm:commensurability} and the results of \Cref{sec:BrownsCrit}, we can now prove: 

\begin{namedthm*}{\Cref{thm:finitenessGraphHoughton}}
    The groups $B(0, 1, r)$ and $\BB(0, 1, r)$, and thus $B(g,h,r)$ and $\BB(g,h,r)$, are of type $F_{r-1}$ but not of type $FP_r$.
\end{namedthm*}
     
\begin{proof}
    
    By \Cref{thm:commensurability}, $B(0,1,r)$ and $B(g,h,r)$ share the same finiteness properties, (and similarly for $\mathcal{B}(0,1,r)$ and $\mathcal{B}(g,h,r)$).  So, it suffices to prove the statement for $B(0,1,r)$ and $\mathcal{B}(0,1,r)$.  
    
   From the results of \Cref{sec:BrownsCrit}, we see the conditions of Brown's criterion (\Cref{thm:BrownsCriterion}) are satisfied, thereby establishing the desired result.  \Cref{thm:CubeComplexContractible} shows that $X$ and $\mathcal{X}$ are contractible.
    \begin{itemize}
        \item[(a)] For $k\geq 1$, both the quotient of $\mathcal{X}^{\leq k}$
          and that of $X^{\leq k}$ are compact by \Cref{thm:CubeComplexContractible}.
        \item[(b)] Cell stabilizers are $F_{\infty}$ by \Cref{lem:DHB_Finfinity} and \Cref{thm:CubeComplexContractible}. 
        \item[(c)] The descending links of vertices in both $X$ and $\mathcal{X}$ are $r-1$ dimensional, and by \Cref{cor:descendingLinkswCM} as long as the height of a vertex $x\in \mathcal{X}$ is at least $4r-4$, then $\lk^{\downarrow}(x)$ is $(r-2)$-connected. By \Cref{lem:descendingLinksCompleteJoin} and \Cref{prop:completeJoinRetract}, the same is true of the descending links of vertices of $X$.
        \item[(d)] By \Cref{lem:infiniteFibers} and \cite[Observation 3.5]{aramayona2023surface}, descending links of vertices in $\mathcal{X}$ are not contractible, and thus by  \Cref{prop:completeJoinRetract} the same is true of the descending links of vertices in $X$. 
    \end{itemize}
    In particular, it follows that $B(0,1,r)$ and $\BB(0,1,r)$ are of type $F_{r-1}$ but not of type $FP_{r}$.
\end{proof}

Since finiteness properties are preserved under commensurability, the converse of \Cref{thm:commensurability} also holds; that is, $B(g,h,r)$ and $B(g',h',r')$ are commensurable \emph{if and only if} $r = r'$, with a similar statement for doubled handlebody Houghton groups.

\section{Cantor graphs}
\label{sec:CantorGraphs}

We now turn our attention to asymptotically rigid mapping class groups of graphs $\G$ with a Cantor set of ends, all accumulated by loops.  The asymptotically rigid mapping class groups of doubled handlebodies associated to such graphs, $M_\G$, are studied in \cite[Section 8]{aramayona2023surface}. 
Specifically, decompose $M_\G = C \cup \bigcup_{j \in J} Y_j$ where the core $C$ is a 3-sphere with $r$ open balls removed, and the pieces $Y_j$ are each diffeomorphic to a copy of $S^2 \times S^1$ with $d+ 1$ many open balls removed. 
 Fixing a marking on each piece to a preferred model for the pieces, one can define a rigid structure on $M_\G$, and as before, a suited submanifold is a connected submanifold of $M_\G$ that can be obtained from the union of the core and finitely many pieces.  
 
 Let $\mathcal{B}_{d,r}$ denote the group of isotopy classes of self-diffeomorphisms of $M_\G$ that preserve the rigid structure outside of some suited submanifold with suited image (see \cite[Definition 3.8]{aramayona2024asymptotic}). Again, as with the doubled handlebody Houghton groups, we consider maps up to isotopy instead of proper isotopy, differing from \cite{aramayona2024asymptotic}.

 When $d = 1$, the doubled handlebody has exactly $r < \infty$ many ends, and this group coincides with the doubled handlebody Houghton group defined in \Cref{sec:doubledHandleHoughton}; that is $\mathcal{B}_{1,r}= \mathcal{B}(0,1,r)$. The finiteness properties when $d > 1$ are already well understood, in particular Aramayona--Bux--Flechsig--Petrosyan--Wu show: 

 \begin{THM}[{\cite[Theorem 1.8]{aramayona2024asymptotic}}]\label{thm:CantorMflFinProp} For all $d > 1$ and $r \ge 1$, $\mathcal{B}_{d,r}$ is of type $F_\infty$.  
 \end{THM}

As in previous sections, the definitions and results for doubled handlebodies extend naturally to the setting of graphs. We now define the corresponding structures for graphs and give an analog of \Cref{thm:CantorMflFinProp} for graphs by directly appealing to the arguments in \cite{aramayona2024asymptotic}.

\subsection{Asymptotically rigid mapping groups of Cantor graphs}

Let $\G$ be an infinite graph with a Cantor set of ends all accumulated by loops, i.e., a blooming Cantor tree. Consider the decomposition $\Gamma= \Gamma_{d, r}(g,h) := C \cup \bigcup_{j \in J} Y_j$, where the core graph $C$ is a bipartite graph $K_{1,r}$ wedged with $g$-many loops at the $r$-valent vertex, and where $Y_j$ is a subgraph of $\Gamma$ that is isomorphic to the graph $Y^{h,d}$, the wedge of the bipartite graph $K_{1,d+1}$ with $h$-many loops at the $(d+1)$-valent vertex. We again call the $Y_j$'s the \textit{pieces} of the decomposition.
Then a $(g,h,r,d)$-rigid structure on $\Gamma$ is a decomposition for which all subgraphs in the decomposition have disjoint interiors, and pieces only meet the core once at its frontier (see \Cref{fig:CantorTree}).
 \begin{figure}[ht!]
    \centering
    \begin{overpic}[width=.7\textwidth]{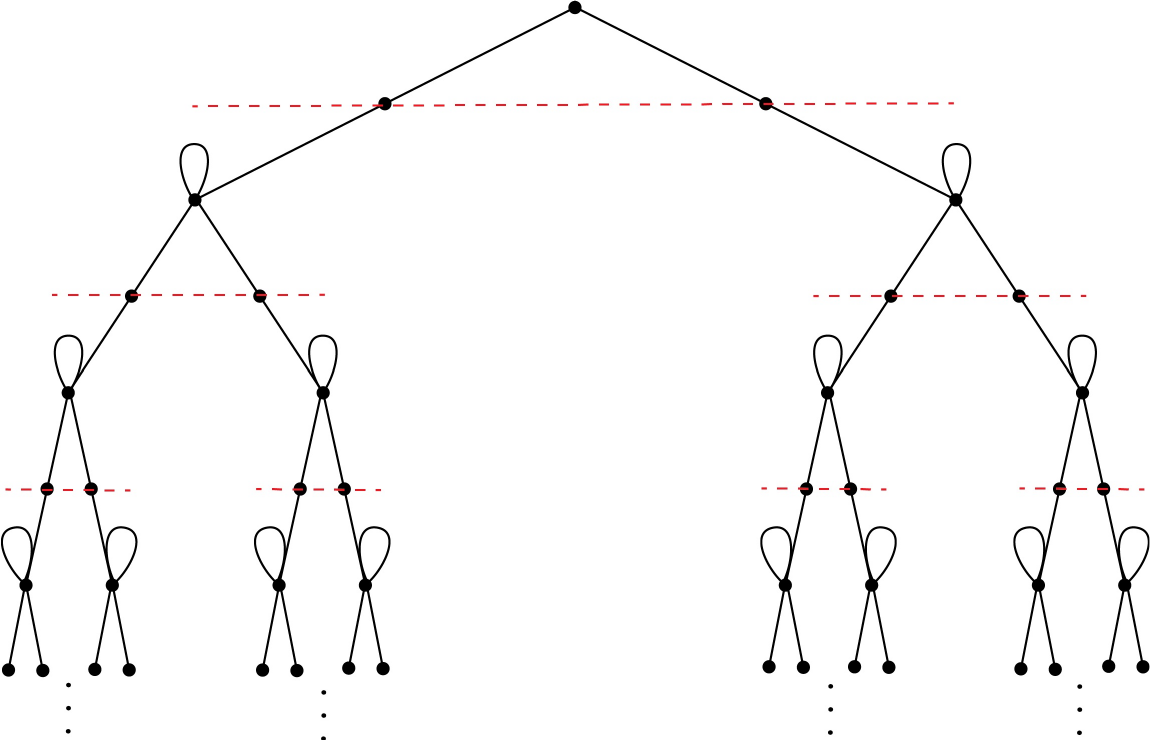}     
    \end{overpic}
    \caption{A $(0,1,2,2)$-rigid structure on the blooming Cantor tree graph.}
    \label{fig:CantorTree}
\end{figure}

The asymptotic mapping class group $B_{d,r}(g,h)$ is defined as before, as the group of maps which eventually respect certain marking maps of the pieces with a fixed copy of $Y^{h, d}$. When $d = 1$, this coincides with the graph Houghton group $B_{1,r}(g,h) = B(g,h,r)$ whose underlying graph has finitely many ends all accumulated by loops. 
\par 

Denote $B_{d,r} := B_{d,r}(0,1)$.  We have a direct analog of \Cref{thm:CantorMflFinProp}:  

\begin{THM}
    \label{thm:CantorFinProp}
    For all $d > 1$ and $r \ge 1$, the group $B_{d,r}$ is of type $F_\infty$.  
\end{THM}

The proof of \Cref{thm:CantorFinProp} does not extend directly to the groups $B_{d,r}(g,h)$, and the finiteness properties in this generality remain open. For comparison, recall that in the Houghton group case, the groups $B(g,h,r)$ inherit their finiteness properties from $B(0,1,r)$, since for all $g \ge 0$ and $h \ge 1$, $B(g,h,r)$ is conjugate to a finite-index subgroup of $B(0,1,r)$ that arises from engulfing the $(g,h)$-rigid structure by the $(0,1)$-rigid structure (see \Cref{thm:commensurability}). On the other hand, embedding a conjugate of $B_{d,r}(g,h)$ into $B_{d', r'}$ is not as easily done, and even when such an embedding occurs it is not as a finite index subgroup. 

The proof of \Cref{thm:CantorFinProp} uses the following version of Brown's criterion: 

\begin{THM} \label{thm:BrownFl} Let $G$ be a group acting by cellwise isometries on a contractible piecewise Euclidean CW-complex $\mathbb{X}$.  Suppose $\mathbb{X}$ is equipped with a discrete $G$-invariant Morse function $h \colon \mathbb{X} \to \mathbb{R}$, and let $\mathbb{X}^{\le s}$ be the largest subcomplex of $\mathbb{X}$ fully contained in $h^{-1}(-\infty, s].$ Suppose that: 
\begin{itemize}
    \item[(a)] $\mathbb{X}$ is contractible. 
    \item[(b)] The quotient of $\mathbb{X}^{\le s}$ by $G$ is finite for all critical values $s$. 
    \item[(c)] Every cell stabilizer is of type $F_\infty$. 
    \item[(d)] For all $l\geq 1$, there is a critical value $s$ such that for every vertex $v \in \mathbb{X}$ with $h(v) \ge s$, we have that $\lkd_{\mathbb{X}}(v)$ is $(\ell -1)$-connected. 
\end{itemize}
    Then $G$ is type $F_\infty$. 
\end{THM}

The proof of \Cref{thm:BrownFl} follows directly from \cite[Theorem 3.1]{aramayona2023surface} and its proof.  To prove \Cref{thm:CantorFinProp}, we will show that $B_{d,r}$ satisfies \Cref{thm:BrownFl}. 

\begin{proof}[Proof of \Cref{thm:CantorFinProp}]
   Let $X$ (resp. $\XX$) be the Stein--Farley complex defined analogously to the one appearing in \Cref{sec:BrownsCrit}, but with underlying graph $\G=\G_{d,r}(0,1)$ (resp. with underlying doubled handlebody $M_{\G}$ for $\G=\G_{d,r}(0,1)$). Complexity functions for both $X$ and $\XX$ can also be defined in the same way as in \Cref{sec:BrownsCrit}. By a simple modification of \Cref{thm:CubeComplexContractible}, $X$ is contractible, and for all $s$, the quotient $X^{\le k} / B_{d,r}$ is finite. Thus, conditions (a), (b), (c) of \Cref{thm:BrownFl} are satisfied. More specifically, direct translations of the arguments in \cite[Proposition 5.7, Lemmas 6.2 and 6.3]{aramayona2023surface} yield the analogous results in our setting. 

    Note that \Cref{lem:CCSurj} and \Cref{lem:descendingLinksCompleteJoin} still work here to show that descending links of vertices in $\XX$ are complete joins over descending links of vertices of $X$ (where the two vertices corresponding to the descending links have the same complexity). In particular, to show that $X$ satisfies condition (d), it suffices to show that the descending links of $\XX$ satisfy (d), by \Cref{prop:completeJoinRetract}. But this follows by a direct application of the results of \cite{aramayona2024asymptotic}. More specifically, Proposition 6.6 and Theorem 6.8 of \cite{aramayona2024asymptotic} imply that (d) holds for $\XX$. Thus we can apply \Cref{thm:BrownFl}, and the result follows.
\end{proof}

\subsection{Stable homology}
The asymptotic mapping class groups $B_{d,r}(g,h)$ are closely related to Higman--Thompson groups $V_{d,r}$ as explained below.  In this section, we use the relationship of $B_{2,1}$ with the classical Thompson group $V = V_{2,1}$ to study the homology of $B_{2,1}$.  We follow arguments similar to those in \cite{Funar_2009, Aramayona2017AsymptoticMC, aramayona2024asymptotic, domingozubiaga2025finitenesspropertiesasymptoticallyrigid}. 
\begin{THM}\label{thm:StableHomology}
    For $g \ge 2i + 4$, the homology of $B_{2,1}$ coincides with the stable homology of $\Aut(F_n)$, that is 
    \begin{equation*}
        H_i(B_{2,1}) \cong H_i(\Aut(F_g)). 
    \end{equation*}
\end{THM}
        
\begin{proof}
    We have that $\Map_c(\Gamma_{d,r}(g,h)) \cong \varinjlim_{n} \Map(\Delta_{r,n})$, analogously to the proof of \Cref{lem:compactsuppPresentation}. Recall $\Delta_{r,n}$ is the subgraph of $\Gamma_{d,r}(g,h)$ lying in the closed $(n + \frac12)$-ball centered at the middle vertex of the core.  By homological stabilization of $\Aut(F_n)$ due to Hatcher--Vogtmann \cite{hatcher2004homology}, we have 
    \begin{equation*}
        H_i(\Map(\Delta_{r,n})) \cong H_i(\Aut(F_{rn}))
    \end{equation*}
    provided that $rn > 2i + 4$.  As the maps in homology stabilize, by taking the direct limit, we obtain
    \begin{equation*} 
        H_i(\Map_c(\Gamma_{d,r}(g,h)) = H_i(\Aut(F_{g}))
    \end{equation*}
    if $g > 2i + 4$.  Furthermore, it is well known that a finite index subgroup of $\Aut(F_g)$ acts freely, properly discontinuously, and cocompactly on a contractible simplicial complex \cite{culler1986moduli}, which implies $\Aut(F_g)$ is type $F_{\infty}$. In particular, $H_i(\Aut(F_g))$ is finitely generated.

    When $d\ge 2$ and $r \ge 1$, the action of $B_{d,r}(g,h)$ on the end space of $\Gamma_{d,r}(g,h)$ induces a homeomorphism of a Cantor set, in particular an element of the Higman--Thompson group $V_{d, r}$, just as in \cite{aramayona2024asymptotic, domingozubiaga2025finitenesspropertiesasymptoticallyrigid}.  Likewise, we have the following short exact sequence (see \cite[Proposition 4.3]{aramayona2024asymptotic}), 
    \begin{equation*}
        1 \to \Map_c(\Gamma_{d,r}(g,h))\to B_{d,r}(g,h) \to V_{d, r} \to 1. 
    \end{equation*}
    In particular, when $d = 2$ and $r = 1$, $V_{2,1} = V$ is the classical Thompson group, which is simple and acyclic.  Thus, after making the appropriate modifications, the proof of \cite[Theorem 6.4]{domingozubiaga2025finitenesspropertiesasymptoticallyrigid} can be directly adapted to show that $H_i(B_{2,1}) \cong H_i(\Map_c(\Gamma_{d,r}(g,h))$, from which the result follows.  
\end{proof}
\section{Classical geometric group theory properties}\label{sec:GGTprops}

As the graph Houghton group $B_r$ is finitely presented for $r \ge 3$, and the pure subgroup has an explicit presentation given in \Cref{thm:PBrPresentation}, one may ask how this group fits both in the classical study of infinite groups as well as into Gromov's program of studying finitely generated groups according to their large-scale geometry \cite{gromov1992asymptotic}.   In this section, we investigate several properties of $B_r$ from this perspective.  
 
We consider below a variety of classical group theory and geometric group theory properties for $B_r$. We begin by listing several easily obtainable properties of $B_r$.
\begin{itemize}
    \item The group is \emph{not residually finite}, as it contains the classical Houghton group $H_r$, which is not residually finite.
    \item The group has \emph{infinite asymptotic dimension} as it contains free abelian subgroups of arbitrarily large rank.
    \item The group is \emph{not virtually torsion free}, as it contains infinite torsion subgroups.
    \item The group has \emph{exponential word growth} as it contains nonabelian free groups.
\end{itemize}

Let us consider the relationship between $\Map(\Gamma)$ and $B_r$.  On the one hand, the graph Houghton group is dense in $\Map(\Gamma)$. This is because $\Map_c(\Gamma_r) \subset B_r$, and $B_r$ surjects onto the quotient $\PMap(\G)/\Map_c(\G_r) \cong \Z^{r-1}$ in \Cref{def:fluxmap}. Since $\PMap(\G) \cong \Map_c(\G_r) \rtimes \Z^{r-1}$, the density claim follows.

On the other hand, $B_r$ is not quasi-isometrically embedded in $\Map(\G)$. By \cite[Theorem A]{domat2025generating}, $\PMap(\G_r)$ admits a well defined quasi-isometry type arising from \emph{coarsely bounded} generating sets (see the book of Rosendal for more details \cite{rosendal2022coarsegeometry}.) 
The generating set they produce contains all compactly supported elements whose support and image do not contain the central point of $\G_r$. Thus, for any infinite order element $f$ of $B_r$ in this set, $f^n$ has word length 1 for any $n \neq 0$ with respect to the word metric on $\PMap(\G_r)$. In particular, the group $B_r$ \emph{does not} quasi-isometrically embed into $\Map(\Gamma_r)$. This raises the question of the extent to which these properties provide information about $\Map(\Gamma)$.

By \cite{hill2025large} and \cite[Section 8.1]{udall2024the-sphere}, the pure mapping class group of a surface and a doubled handlebody with finitely many ends, $r \ge 3$, all accumulated by genus similarly admits a well-defined quasi-isometry type.  Adaptations of these arguments can be used to prove a similar result for the pure handlebody group.  By the same arguments as above, the surface and (doubled) handlebody Houghton groups are not quasi-isometrically embedded in their respective mapping class groups.

\subsection{Ends of $B_r$} We begin by computing the number of ends of the graph Houghton group. As a coarse geometric invariant, the number of ends reflects how a group can be decomposed over finite subgroups. In particular, being 1-ended implies that a group does not split nontrivially over finite subgroups per Stallings' theorem on ends of groups \cite{stallings1968torsion}. 

\begin{THM}\label{thm:1ended}
    The group $B_r$ is 1-ended for $r \ge 2$.  
\end{THM}
\begin{proof}
    In the case when $r = 1$, the group is not even finitely generated by \Cref{prop:B1NotFinGen}. We will prove the statement for $r = 2$; the case of $r \ge 2$ follows by a similar argument. We will also prove the result for $PB_r$, and then the result for $B_r$ follows as the two groups are quasi-isometric.

    From the generating set of $PB_2$ given in \Cref{prop:B2finitegen}, it is not difficult to show that 
    \begin{equation*}
        S =  \{\hat{h}_1, \hat{h}_2, \sigma_1\eta_1, \sigma_2\eta_2, \tau \eta_2, \eta_2\}
    \end{equation*}
    is also a generating set (which we justify below). As before, $\tau$ is the loop flip $a_1^1 \mapsto \overline{a_1^1}$.  Here, $\sigma_i$ are the transposition of loops $a_3^i$ and $a_4^i$ and $\eta_i$ is the map inducing the map $a_1^i \mapsto \ov{a_2^i} a_1^i$, $a_2^i \mapsto \ov{a_2^i}$.  The shift maps $\hat{h}_i$ are defined by shifting between $e_1$ and $e_2$ but passing over $a_j^i$ for $1 \le j \le 4$, so that their support is disjoint from that of $\sigma_i$ and $\eta_i$.
    
    The four compactly supported elements together with their conjugates by $\hat{h}_i$'s generate $\Map_c(\G_2)$.  Both $\hat{h}_i$'s differ from the standard generating shift $h$ by compactly supported elements, so $S$ generates $PB_2$.
    
    By \cite[Theorem 9]{anderson2005simplecriterionnonrelativehyperbolicity}, it follows that $PB_2$ is 1-ended, since the commutativity graph $K(PB_r, S)$ is connected, and $\langle\sigma_1\eta_1, \sigma_2\eta_2 \rangle$ is rank 2 abelian. One can repeat this setup in $PB_r$, defining similar pairs of modified shift maps for each $h_i$, $2\leq i \leq r$, and the proof of the general case follows.  
\end{proof}

On the other hand, for $r\geq 2$, $PB_r$ acts on the real line in many different ways. Indeed, $PB_r$ admits many homomorphisms to $\Z$ via its abelianization. In particular, $PB_r$ splits as an HNN extension.

\subsection{Tits alternative}

A group $G$ satisfies the \emph{Tits alternative} if every subgroup of $G$ is either virtually nilpotent or contains a nonabelian free subgroup. Bestvina--Feighn--Handel show that $\Out(F_n)$ satisfies the Tits alternative in a series of three papers \cite{bestvina2000titsalternative, bestvina2004solvable, bestvina2005kolchin} (in fact, they also show that $\Out(F_n)$ satisfies the strong Tits alternative, see below). 

Here we note that the same fact doesn't hold for $B_r$. Namely, we have the following.

\begin{THM}\label{thm:TitsAlternative}
    The groups $H_r$, $B_r$ and $\mathcal{B}_r$  do not satisfy the Tits alternative for any $r\geq 1$.
\end{THM}
\begin{proof}
    It was noted in \Cref{sec:HoughtonVariants} that the classical Houghton group $H_r$ embeds in $B_r$, and consequently also $\mathcal{B}_r$. It thus suffices to show that $H_r$ is not virtually nilpotent and also doesn't contain any free subgroups.

    The lower central series of $H_r$ does not terminate. This is because $H_r$ contains a minimal nontrivial normal subgroup, namely the infinite alternating group. Every finite index subgroup of $H_r$ contains the infinite alternating group (otherwise the infinite alternating group would have a proper finite index normal subgroup). Thus $H_r$ is not virtually nilpotent.

    On the other hand, $H_r$ does not contain a nonabelian free group. Indeed, for any $f_1, f_2\in H_r$, the element $[f_1, f_2]$ is supported on finitely many points (as its image in the abelianization is $0$), and in particular has finite order. 
\end{proof}

\begin{Q}
    Do any variants of the Houghton group discussed in \Cref{sec:HoughtonVariants} satisfy the Tits alternative? 
\end{Q}

On the other hand, none of the Houghton group variants satisfy the \emph{strong Tits Alternative}. A group $G$ satisfies the strong Tits alternative if every subgroup is either virtually abelian or contains a nonabelian free group. We have shown that $H_r$ does not satisfy the (strong) Tits alternative. For the rest of them, one can modify the construction in \cite[Theorem 6.1]{domat2025generating} to produce a wreath product $\Z \wr \Z$ subgroup of any of the variants, which implies that the group does not satisfy the strong Tits alternative. 

\subsection{BNSR-invariants}
To any finitely generated group one can associate certain geometric invariants, called the BNSR-invariants. These have been computed for the classical Houghton groups by Zaremsky in \cite{zaremsky2017invariants, zaremsky2020bnsr}, and essentially the same computation was carried out for the surface Houghton groups by Noah Torgerson and the fourth author \cite{torgerson2024bnsrinvariantssurfacehoughtongroups}. Rather than reproduce any of these details here, we shall simply note that (with occasional modifications), the entire analysis of \cite[Section 3 \& 4]{torgerson2024bnsrinvariantssurfacehoughtongroups} can be carried out unchanged. Thus, we have the following theorems.
\begin{THM} \label{thm:XisCat0}
    The Stein--Farley complexes $X$ and $\XX$ are CAT(0).
\end{THM}

\begin{THM}\label{thm:BSNRInvs}
    Let $G_n$ be either the pure graph Houghton group or the pure doubled handlebody Houghton group. Then the BNSR-invariants of $G_n$ are the same as for the ordinary and surface Houghton groups, that is, given a character $\chi$ in ascending standard form $\chi = a_1\chi_1 + \cdots + a_n\chi_n$, with $a_1\leq \cdots \leq a_{m(\chi)}< a_{m(\chi)+1} = \cdots = a_n = 0$, then $[\chi] \in \Sigma^{m(\chi)-1}(G_n) \setminus \Sigma^{m(\chi)}(G_n)$.
\end{THM}
\subsection{Solvable word problem and Dehn function} \label{sec:WordProblemDF}

We assume for the rest of this section that $r\geq 3$. By \Cref{thm:PBrPresentation}, for such $r$ the group $PB_r$ is finitely presented, and thus has a Dehn function with respect to this presentation. We record here our attempts to prove an upper bound for the Dehn function $\delta_{PB_r}(x),$ and explain where the gap in the proof is (see the discussion after \Cref{lem:ExpBd}).

In the following, we write $g_1 \equiv g_2$ for $g_1, g_2$ words in a fixed generating set to denote that $g_1$ and $g_2$ are freely equivalent. The equality symbol will be used when relators are applied. 
\begin{LEM} \label{lem:BdRelLength}
    Given a null-homotopic word $w \in PB_r$ with $|w|\leq x$, it can be rewritten as a word $w'$ in $\Map_c(\G_r)$ with length at most $O(x^6)$. Further, the area between $w$ and $w'$ is at most $x^2$.
\end{LEM}
\begin{proof}
    For $\beta$ equal to $\sigma, \tau$, or $\eta$, and any $k$ with $2\leq k \leq r$, we can write
    $$h_k\beta\equiv \beta^{h_k}h_k.$$
    In particular, we have that
    $$w\equiv fH$$
    where $f\in \Aut(F_{rx})$ is a product of no more than $x$ involutions and $H$ is a product of powers of the shift maps $\{h_i\}$. Let $2\leq i < j \leq r$. Using the relation $\sigma = [h_i, h_j]$, we can write
    $$\ov{h}_j \ov{h}_i=\sigma \overline{h}_i\ov{h}_j,
    \qquad 
    \ov{h}_j h_i=\sigma^{h_i}h_i \ov{h}_j$$
    $$h_j \ov{h}_i = \sigma^{h_j} \ov{h}_i h_j, \text{ and } h_j h_i = \sigma^{h_j h_i} h_i h_j$$
    By repeated applying these identities along with the conjugation trick above, we can write
    $$w=ff_1f_2\cdots f_{\ell}h_2^{n_2}\cdots h_r^{n_r}$$
    where each $f_p\in \Aut(F_{rx})$ is a conjugate of $\sigma, \tau$, or $\eta$, and $n_k\in \Z$. As $w$ represents the trivial element, the image of $w$ in $\Z^{r-1}$ under the $(r-1)$-fold flux map is trivial as well. But the image of $w$ under this map is $(n_2, \ldots, n_r)$, so $n_k=0$ for all $k$. In particular, we have rewritten $w$ as $w'\in \Map_c(\G_r)$ where
    $$w'=ff_1f_2\cdots f_{\ell}.$$
    \par 
    We note that $\ell\leq x^2$, as the number of terms that appears at this step is at most the number of times an $h_j$ had to move past an $h_i$. But the length of $H$ was at most $x$, so the number of such moves is at most $x^2$. 
    \par 
     Let $\gamma$ denote any conjugate of $\sigma, \tau$, or $\eta$ by shifts appearing in the expression of $w'$. We think of $\gamma$ as an element of $\Aut(F_{rx})$.  
    Note that conjugating $\sigma$ (resp. $\tau$) by shifts will result in a loop swap (resp. loop flip) in $\Aut(F_{rx})$, i.e., an element  $\sigma_{lk}^{ij}$ (resp. $\tau_j^i$).  Likewise, conjugating $\eta$ will result in a form similar to $\eta$ but with support on a different pair of loops. 

    By the proof of \cite[Lemma 4.4]{lee2012geometry}, if $\gamma$ is a conjugate of $\sigma$ by shifts, then the length of $\gamma$ is at most $O(x^3)$.
    
    As $S:= \{s_j^i\}_{(i, j) \in I^r_x}$ acts transitively on the loops (resp. transitively on pairs of loops), if $\gamma$ is a conjugate of $\tau$ by shifts (resp. a conjugate of $\eta$ by shifts), then $\gamma$ can be written as a conjugate of $\tau$ (resp. $\eta$) by elements of $S$. Further, at most $O(x)$ elements of $S$ are needed in the conjugating element. As each element of $S$ has length at most $O(x^3)$, it follows that $\gamma$ in either case has length at most $O(x^4)$.

    As noted above, $\ell \le x^2$ and each $f_i$ is a conjugate by shifts of either $\eta$, $\tau$ or $\sigma$, and $f$ is a product of at most $x$ such conjugates. By the previous paragraph, each conjugate has at most length $O(x^4),$ hence $w'$ has length at most $O(x^6)$.

    The claim on the area between $w$ and $w'$ follows as in \cite[Lemma 4.4]{lee2012geometry}. Namely, in the steps performed to transform $w$ to $w'$, the relation $\sigma=[h_i, h_j]$ was applied at most $x^2$ times. This is the only relation we applied.
\end{proof}

    As an immediate consequence, we can borrow the solvability of the word problem of $\Aut(F_n)$.
\begin{COR}\label{cor:wordproblem}
    The group $PB_r$ has solvable word problem. 
\end{COR}

\begin{proof}
    Given a word $w$, if its image in the abelianization is nontrivial, then $w\neq 1$. Otherwise, the proof of \Cref{lem:BdRelLength} shows that we can rewrite $w$ as a word $w'\in \Map_c(\G)$, so in particular it lies in $\Aut(F_n)$ for some $n$. Thus, we have reduced the word problem for $w$ in $PB_r$ to the word problem for $w'$ in $\Aut(F_n)$. Thus as $\Aut(F_n)$ has solvable word problem, so does $PB_r$. 
\end{proof}  

\begin{LEM} \label{lem:ExpBd}
    For all relators $r$ of $\mathrm{Aut}(F_{rx})$, $\mathrm{Area}_{PB_r}(r)$ is at most $O(e^x)$. 
\end{LEM}

\begin{proof}
    Relators 1(a), 1(c), 1(d), 1(e), 2(a), 2(b), 2(e),  2(g), and 2(h) listed in \Cref{thm:AFVpresentation} already in appear in the presentation of $PB_r$ as (possibly conjugates of) relators \refrel{rel:r1}, \refrel{rel:r3}, \refrel{rel:r6}, \refrel{rel:r7}, \refrel{rel:r10}, \refrel{rel:r11}, \refrel{rel:r17}, \refrel{rel:r18}, and \refrel{rel:r15} in \Cref{thm:PBrPresentation}, respectively. 

    It remains to consider the commutator relations 1(b), 1(f), 2(c) and 2(d). Consider the relators 
    \begin{equation*}
        R_{\sigma, \sigma} : \left[\sigma, \sigma^{h_i^{k + 1}}\right], \;  R_{\tau, \sigma} : \left[\tau, \sigma^{h_i^{k + 1}}\right], \;  R_{\eta, \tau} : \left[\eta, \tau^{h_i^{k + 1}}\right], \;  R_{\eta, \sigma} : \left[\eta, \sigma^{h_i^{k + 1}}\right].
    \end{equation*}
    One can produce all the expressions in $(R_{1b})$, $(R_{1f})$, $(R_{2c})$, and $(R_{2d})$ from the proof of \Cref{thm:PBrPresentation}, by taking appropriate conjugations of the relations $R_{a,b}$ above.  We also have the following auxiliary relations
    \begin{equation*}
        u_\sigma = \sigma^{h_i^k}\sigma^{h_j^k} \quad \text{and} \quad u_\tau = \tau^{h_i^k}\tau^{h_j^k}
    \end{equation*}
    which likewise follow from the proof of \Cref{thm:PBrPresentation}.  One can then show that for $a \in \{\sigma, \tau, \eta\}$ and $b \in \{\sigma, \tau\}$, the areas of $R_{a,b}$ and $u_b$ are bounded above by $O(e^x)$, by following the argument of Lee in \cite[Lemma 4.7]{lee2012geometry}.  Specifically, the result follows by a double induction on $k$, and with $R_{a,b}$ in place of $v_k$, $u_b$ in place of $u_k$, and $A_a = a a^{h_j^{h_i^2}}$ in place of $A$. 
\end{proof}

To follow the argument of \cite[Theorem D]{lee2012geometry}, we require the following three points to hold.

\begin{enumerate}
    \item For a given word $w \in PB_r$, with $|w|\le x$, we can rewrite $w$ as $w'\in \Aut(F_{rx})$ in such a way that 
    \begin{enumerate}
        \item $w'$ has length at most $O(x^6)$, 
        \item the area between $w$ and $w'$ is at most $x^2$. 
    \end{enumerate}
    \item There is an upper bound $f(x)$ for $\delta_{\Aut(F_{rx})}(x)$.
    \item For all relators $r$ of $\Aut(F_{rx})$ as in the presentation in \Cref{thm:AFVpresentation}, $\mathrm{Area}_{PB_r}(r)$ is at most $O(e^x)$.  
\end{enumerate}
From this strategy, it follows that 
\begin{equation*}
    \mathrm{Area}_{PB_r}(w) \le f(x^6) e^{Ax + A} + x^2
\end{equation*}
for some constant $A \ge 1$.  Point (1) follows from \Cref{lem:BdRelLength}, and (3) is \Cref{lem:ExpBd}.  

We do not have an analog of \cite[Lemma 4.3]{lee2012geometry}, from which (2) follows in Lee's context.   
While for every $m\geq 3$ there is an exponential upper bound for the function $\delta_{\Aut(F_{m})}$ by \cite{hatcher1996isoperimetric}, we have no control over the base of this exponential, and thus we do not have control over $f(x)$. Further, they obtain their bound by proving an exponential bound for the isoperimetric inequality in the spine of Auter space, and then transferring it to $\Aut(F_{m})$ by a quasi-isometry. The quasi-isometry coefficients depend on the choice of generating set, and we have found no way to control these coefficients with our specific generating set.  We leave open the question of how to bound the Dehn function with respect to this generating set.

\bibliography{bib}

\newcommand{\etalchar}[1]{$^{#1}$}
\begin{thebibliography}{ADMQ90}

\bibitem[AAS05]{anderson2005simplecriterionnonrelativehyperbolicity}
James~W. Anderson, Javier Aramayona, and Kenneth~J. Shackleton.
\newblock A simple criterion for non-relative hyperbolicity and one-endedness
  of groups.
\newblock {\em arXiv preprint arXiv:0504271}, 2005.

\bibitem[AB25]{algom-kfir2025groups}
Yael Algom{-}Kfir and Mladen Bestvina.
\newblock Groups of proper homotopy equivalences of graphs and {N}ielsen
  {R}ealization.
\newblock In {\em Topology at infinity of discrete groups}, volume 812 of {\em
  Contemp. Math.}, pages 1--31. Amer. Math. Soc., Providence, RI, 2025.

\bibitem[ABF{\etalchar{+}}24]{aramayona2024asymptotic}
Javier Aramayona, Kai-Uwe Bux, Jonas Flechsig, Nansen Petrosyan, and Xiaolei
  Wu.
\newblock Asymptotic mapping class groups of {C}antor manifolds and their
  finiteness properties.
\newblock {\em Rev. Mat. Iberoam.}, 40(6):2003--2072, 2024.
\newblock With an appendix by Oscar Randal-Williams.

\bibitem[ABKL23]{aramayona2023surface}
Javier Aramayona, Kai-Uwe Bux, Heejoung Kim, and Christopher~J. Leininger.
\newblock Surface {H}oughton groups.
\newblock {\em Mathematische Annalen}, Nov 2023.

\bibitem[ADL24]{aramayona2024isomorphisms}
Javier Aramayona, George Domat, and Christopher~J. Leininger.
\newblock Isomorphisms and commensurability of surface {H}oughton groups.
\newblock {\em J. Group Theory}, 27(5):1129--1141, 2024.

\bibitem[ADMQ90]{ayala1990proper}
R.~Ayala, E.~Dominguez, A.~M\'{a}rquez, and A.~Quintero.
\newblock Proper homotopy classification of graphs.
\newblock {\em Bull. London Math. Soc.}, 22(5):417--421, 1990.

\bibitem[AF17]{Aramayona2017AsymptoticMC}
Javier Aramayona and Louis Funar.
\newblock Asymptotic mapping class groups of closed surfaces punctured along
  {Cantor} sets.
\newblock {\em Moscow Mathematical Journal}, 2017.

\bibitem[AFV08]{armstrong2008a-presentation}
Heather Armstrong, Bradley Forrest, and Karen Vogtmann.
\newblock A presentation for {${\rm Aut}(F_n)$}.
\newblock {\em J. Group Theory}, 11(2):267--276, 2008.

\bibitem[Alo94]{alonso1994finiteness}
Juan~M Alonso.
\newblock Finiteness conditions on groups and quasi-isometries.
\newblock {\em Journal of Pure and Applied Algebra}, 95(2):121--129, 1994.

\bibitem[BB97]{bestvina1997morse}
Mladen Bestvina and Noel Brady.
\newblock Morse theory and finiteness properties of groups.
\newblock {\em Invent. Math.}, 129(3):445--470, 1997.

\bibitem[BFH00]{bestvina2000titsalternative}
Mladen Bestvina, Mark Feighn, and Michael Handel.
\newblock The {T}its alternative for {${\rm Out}(F_n)$}. {I}. {D}ynamics of
  exponentially-growing automorphisms.
\newblock {\em Ann. of Math. (2)}, 151(2):517--623, 2000.

\bibitem[BFH04]{bestvina2004solvable}
Mladen Bestvina, Mark Feighn, and Michael Handel.
\newblock Solvable subgroups of $\textrm{Out}({F}_n)$ are virtually abelian.
\newblock {\em Geometriae Dedicata}, 104(1):71--96, 2004.

\bibitem[BFH05]{bestvina2005kolchin}
Mladen Bestvina, Mark Feighn, and Michael Handel.
\newblock The {T}its alternative for {${\rm Out}(F_n)$}. {II}. {A} {K}olchin
  type theorem.
\newblock {\em Ann. of Math. (2)}, 161(1):1--59, 2005.

\bibitem[BFHZ25]{belk2025boonehigmanembeddingsmathrmautfnmapping}
James Belk, Francesco Fournier{-}Facio, James Hyde, and Matthew C.~B. Zaremsky.
\newblock {Boone--Higman} embeddings of {${\rm Aut}(F_n)$} and mapping class
  groups of punctured surfaces.
\newblock {\em arXiv preprint arXiv:2503.21882}, 2025.

\bibitem[BH74]{boone1974algebraic}
William~W. Boone and Graham Higman.
\newblock An algebraic characterization of groups with soluble word problem.
\newblock {\em J. Austral. Math. Soc.}, 18:41--53, 1974.

\bibitem[Bro87]{BROWN198745}
Kenneth~S. Brown.
\newblock Finiteness properties of groups.
\newblock {\em Journal of Pure and Applied Algebra}, 44(1):45--75, 1987.

\bibitem[CV86]{culler1986moduli}
Marc Culler and Karen Vogtmann.
\newblock Moduli of graphs and automorphisms of free groups.
\newblock {\em Inventiones mathematicae}, 84(1):91--119, 1986.

\bibitem[Deg00]{degenhardt2000endlichkeitseigeinschaften}
F.~Degenhardt.
\newblock Endlichkeitseigeinschaften gewisser gruppen von zöpfen unendlicher
  ordnung.
\newblock {\em University of Frankfurt}, 2000.
\newblock PhD Thesis.

\bibitem[DHK23]{domat2023coarse}
George Domat, Hannah Hoganson, and Sanghoon Kwak.
\newblock Coarse geometry of pure mapping class groups of infinite graphs.
\newblock {\em Adv. Math.}, 413:Paper No. 108836, 57, 2023.

\bibitem[DHK24]{dickmann2024surfacesproperhomotopyequivalent}
Ryan Dickmann, Hannah Hoganson, and Sanghoon Kwak.
\newblock Surfaces proper homotopy equivalent to graphs and their
  {D}ehn-{N}ielsen-{B}aer maps.
\newblock {\em arXiv preprint arXiv:2410.20877}, 2024.

\bibitem[DHK25]{domat2025generating}
George Domat, Hannah Hoganson, and Sanghoon Kwak.
\newblock Generating {Sets} and {Algebraic} {Properties} of {Pure} {Mapping}
  {Class} {Groups} of {Infinite} {Graphs}.
\newblock {\em Annales Henri Lebesgue}, 8:373--416, 2025.

\bibitem[DZ25]{domingozubiaga2025finitenesspropertiesasymptoticallyrigid}
Sergio Domingo-Zubiaga.
\newblock Finiteness properties of asymptotically rigid handlebody groups.
\newblock {\em arXiv preprint arXiv:2504.05787}, 2025.

\bibitem[Far05]{farley2005homological}
Daniel Farley.
\newblock Homological and finiteness properties of picture groups.
\newblock {\em Transactions of the American Mathematical Society},
  357(9):3567--3584, 2005.

\bibitem[FK09]{Funar_2009}
Louis Funar and Christophe Kapoudjian.
\newblock An infinite genus mapping class group and stable cohomology.
\newblock {\em Communications in Mathematical Physics}, 287(3):787–804,
  February 2009.

\bibitem[Fun07]{funar2007braided}
L.~Funar.
\newblock Braided {H}oughton groups as mapping class groups.
\newblock {\em An. Ştiinţ. Univ. Al. I. Cuza Iaşi. Mat. (N.S.)},
  52(2):229--240, 2007.

\bibitem[GLU22]{Genevois_2022}
Anthony Genevois, Anne Lonjou, and Christian Urech.
\newblock Asymptotically rigid mapping class groups, {I}: Finiteness properties
  of braided {T}hompson’s and {H}oughton’s groups.
\newblock {\em Geometry $\&$ Topology}, 26(3):1385–1434, August 2022.

\bibitem[Gro93]{gromov1992asymptotic}
M.~Gromov.
\newblock Asymptotic invariants of infinite groups.
\newblock In {\em Geometric group theory, {V}ol.\ 2 ({S}ussex, 1991)}, volume
  182 of {\em London Math. Soc. Lecture Note Ser.}, pages 1--295. Cambridge
  Univ. Press, Cambridge, 1993.

\bibitem[Hig61]{higman1961subgroups}
G.~Higman.
\newblock Subgroups of finitely presented groups.
\newblock {\em Proc. Roy. Soc. London Ser. A}, 262:455--475, 1961.

\bibitem[Hil25]{hill2025large}
Thomas Hill.
\newblock Large-scale geometry of pure mapping class groups of infinite-type
  surfaces.
\newblock {\em Proceedings of the American Mathematical Society},
  153(06):2667--2680, 2025.

\bibitem[HKR{\etalchar{+}}24]{hill2024automorphismsspherecomplexinfinite}
Thomas Hill, Michael~C. Kopreski, Rebecca Rechkin, George Shaji, and Brian
  Udall.
\newblock Automorphisms of the sphere complex of an infinite graph.
\newblock {\em arXiv preprint arXiv:2410.06531}, 2024.

\bibitem[Hou78]{houghton1978first}
CH~Houghton.
\newblock The first cohomology of a group with permutation module coefficients.
\newblock {\em Archiv der Mathematik}, 31:254--258, 1978.

\bibitem[HV96]{hatcher1996isoperimetric}
Allen Hatcher and Karen Vogtmann.
\newblock Isoperimetric inequalities for automorphism groups of free groups.
\newblock {\em Pacific Journal of Mathematics}, 173(2):425--441, 1996.

\bibitem[HV04]{hatcher2004homology}
Allen Hatcher and Karen Vogtmann.
\newblock Homology stability for outer automorphism groups of free groups.
\newblock {\em Algebr. Geom. Topol.}, 4:1253--1272, 2004.

\bibitem[HW10]{hatcher2010stabilization}
Allen Hatcher and Nathalie Wahl.
\newblock Stabilization for mapping class groups of 3-manifolds.
\newblock {\em Duke Mathematical Journal}, 155(2), November 2010.

\bibitem[HW24]{he2024relative}
Yan~Mary He and Chenxi Wu.
\newblock Relative train tracks and endperiodic graph maps.
\newblock {\em arXiv preprint arXiv:2408.13401}, 2024.

\bibitem[Joh99]{johnson1999embedding}
D.~L. Johnson.
\newblock Embedding some recursively presented groups.
\newblock In {\em Groups {S}t. {A}ndrews 1997 in {B}ath, {II}}, volume 261 of
  {\em London Math. Soc. Lecture Note Ser.}, pages 410--416. Cambridge Univ.
  Press, Cambridge, 1999.

\bibitem[Lau74]{laudenbach1974topologie}
Fran{\c{c}}ois Laudenbach.
\newblock {\em Topologie de la dimension trois: homotopie et isotopie},
  volume~12.
\newblock Soci{\'e}t{\'e} math{\'e}matique de France, 1974.

\bibitem[Lee12]{lee2012geometry}
Sang~Rae Lee.
\newblock {\em Geometry of {H}oughton's groups}.
\newblock ProQuest LLC, Ann Arbor, MI, 2012.
\newblock Thesis (Ph.D.)--The University of Oklahoma.

\bibitem[Mea24]{meadow2024end}
Ruth Meadow{-}MacLeod.
\newblock {\em End-{P}eriodic {T}rain {T}rack {M}aps and {D}ynamics on
  {F}ree-by-{C}yclic {G}roups}.
\newblock ProQuest LLC, Ann Arbor, MI, 2024.
\newblock Thesis (Ph.D.)--Temple University.

\bibitem[PV18]{Patel_2018}
Priyam Patel and Nicholas Vlamis.
\newblock Algebraic and topological properties of big mapping class groups.
\newblock {\em Algebraic \& Geometric Topology}, 18(7):4109–4142, December
  2018.

\bibitem[Ros22]{rosendal2022coarsegeometry}
Christian Rosendal.
\newblock {\em Coarse geometry of topological groups}, volume 223 of {\em
  Cambridge Tracts in Mathematics}.
\newblock Cambridge University Press, Cambridge, 2022.

\bibitem[She25]{shen2025}
Jiayi Shen.
\newblock {\em Normal Closure of Finite Subgroups of $\mathrm{Aut}(F_n)$ and
  $\mathrm{Out}(F_n)$}.
\newblock PhD thesis, Notre Dame, 2025.

\bibitem[Sta68]{stallings1968torsion}
John~R Stallings.
\newblock On torsion-free groups with infinitely many ends.
\newblock {\em Annals of Mathematics}, 88(2):312--334, 1968.

\bibitem[Ste92]{stein1992groups}
Melanie Stein.
\newblock Groups of piecewise linear homeomorphisms.
\newblock {\em Transactions of the American Mathematical Society},
  332(2):477--514, 1992.

\bibitem[TW24]{torgerson2024bnsrinvariantssurfacehoughtongroups}
Noah Torgerson and Jeremy West.
\newblock {BNSR}-invariants of surface {H}oughton groups.
\newblock {\em arXiv preprint arXiv:2403.04941}, 03 2024.

\bibitem[Uda24]{udall2024the-sphere}
Brian Udall.
\newblock The sphere complex of a locally finite graph.
\newblock {\em arXiv preprint arXiv:2407.07976}, 07 2024.

\bibitem[Zar17]{zaremsky2017invariants}
Matthew~CB Zaremsky.
\newblock On the ${\Sigma}$-invariants of generalized {T}hompson groups and
  {H}oughton groups.
\newblock {\em International Mathematics Research Notices},
  2017(19):5861--5896, 2017.

\bibitem[Zar20]{zaremsky2020bnsr}
Matthew~CB Zaremsky.
\newblock The {BNSR}-invariants of the {H}oughton groups, concluded.
\newblock {\em Proceedings of the Edinburgh Mathematical Society}, 63(1):1--11,
  2020.

\end{thebibliography}
\bibliographystyle{alpha}

\end{document}